\documentclass[journal]{IEEEtran}
\usepackage{graphicx}          % Include this line if your
\usepackage[dvips]{epsfig}    % or this line, depending on which
\usepackage{latexsym,amssymb}
\usepackage{amsmath, amsbsy, amsthm}
\usepackage{amsopn, amstext}
\usepackage{colortbl}
\usepackage{cite}
\newtheorem{remark}{Remark}
\newtheorem{theorem}{Theorem}
\newtheorem{definition}{Definition}
\newtheorem{lemma}{Lemma}

\newtheorem{problem}{Problem}
\newtheorem{assumption}{Assumption}

\ifCLASSINFOpdf
\else
\fi
\allowdisplaybreaks
    
\begin{document}
\title{Stabilization Control for Linear Continuous-time Mean-field Systems}

\author{Qingyuan~Qi,~\IEEEmembership{}
        Huanshui~Zhang,~\IEEEmembership{Senior Member,~IEEE}
        % <-this % stops a space
\thanks{This work is supported by the National Natural Science Foundation of
China under Grants 61120106011, 61573221, 61633014.

Q. Qi and H. Zhang are with School of Control Science and Engineering, Shandong University, Jinan, Shandong,
P.R.China 250061. H. Zhang is the corresponding author. (e-mail: qiqy123@163.com;
hszhang@sdu.edu.cn)}}

\maketitle

\begin{abstract}
This paper investigates the stabilization and control problems for linear continuous-time mean-field systems (MFS). Under standard assumptions, necessary and sufficient conditions to stabilize the mean-field systems in the mean square sense are explored {\em for the first time}. It is shown that, under the assumption of exact detectability (exact observability), the mean-field system is stabilizable if and only if a coupled algebraic Riccati equation (ARE) admits a unique positive semi-definite solution (positive definite solution), which coincides with the classical stabilization results for standard deterministic systems and stochastic systems.

One of the key techniques in the paper is the obtained solution to the forward and backward stochastic differential equation (FBSDE) associated with the maximum principle for an optimal control problem. Actually, with the analytical FBSDE solution, a necessary and sufficient solvability condition of the optimal control, under mild conditions, is derived. Accordingly, the stabilization condition is presented by defining an Lyaponuv functional via the solution to the FBSDE and the optimal cost function.

It is worth of pointing out that the presented results are different from the previous works \cite{huangjh} for stabilization and also different from the works \cite{yong}, \cite{sunj}, \cite{sunj2} on optimal control.
\end{abstract}


% \fi
\IEEEpeerreviewmaketitle

\section{Introduction}
\label{sec:1}
This paper mainly considers the stabilization and control problems for linear continuous-time mean-field systems. Different from the standard optimal control problems, the mean field LQ control is involved with the dynamic systems described by mean-field stochastic differential/difference equations (MF-SDEs). The study of MF-SDEs have received much attention since 1950s, please see \cite{kac}-\cite{dawson2} and references therein. Based on the system theory developed on MF-SDEs as mentioned in the above, the optimal control problems especially the linear quadratic optimal control and related problems have been studied in recent years, one can refer to \cite{buck1}-\cite{ni2}.

Particularly, the pioneering study of the LQ control problem for continuous-time mean-field systems was given by \cite{yong}, in which solvability conditions in terms of operator criteria were provided. \cite{sunj} and \cite{sunj2} further investigated the open-loop solvability and closed-loop solvability, respectively. \cite{huangjh} dealt with the infinite horizon case, several notions of stability and relationships between them was discussed, where the optimal controller of infinite horizon case can be presented via algebraic Riccati equations. It is also noted that the maximum principle for mean-field systems was presented in \cite{sunj}, \cite{buck1}, \cite{haf}, \cite{juanli} and \cite{jli2}. % while the obtained results cannot be applied in solving problems to be studied in this paper.

It is worth of pointing out that all the aforementioned literatures mainly focused on the optimal control problem. The stabilization problem for mean-field system remains least investigated and little progress was made.

 %The main challenging problems are as: 1) The solvability conditions for mean-field LQ control were given in operator form in previous works, which is difficult to be verified in practice \cite{yong}; 2) The stabilization of mean-field systems remains to be studied as the sufficient and necessary stabilization condition has not been explored in previous works \cite{huangjh} under standard assumptions.

The study of stabilization and optimal control problems for mean-field system of infinite horizon case is significant and essential. In fact,  if the system could not be stabilizable, the study of optimal control problem of infinite horizon case would be meaningless. On the other hand, the study of stabilization problem is an important aspect in classical control problem, see \cite{ande} and \cite{fll}.

This paper focuses on investigating the stabilization problems for continuous-time mean-field systems which is a companion paper of \cite{qz} where the discrete-time case has been considered. Firstly, with the maximum principle, the optimal controller is designed based on a coupled Riccati equation which is derived from the solution to a FBSDE regarding with the dynamics of costate and state. The solvability condition (necessary and sufficient) of optimal control is then obtained via the coupled Riccati equation. Secondly, a coupled ARE is obtained through the convergence analysis of Riccati equations for finite horizon case, and the infinite horizon optimal controller is then obtained accordingly. Finally, the mean square stabilization for the mean-field systems is investigated with an Lyapunov function defined with the optimal cost function. It is to be shown that, under the assumption of exact detectability, we show that the mean-field system is stabilizable if and only if the coupled ARE admits a unique positive semi-definite solution. Moreover, under the exact observability assumption, the mean-field system is stabilizable if and only if the coupled ARE admits a unique positive definite solution.

It should be highlighted that the weighting matrices $R$ and $R+\bar{R}$ in cost function are only required to be positive semi-definite in exploring the stabilizing conditions and optimal control, which is a weaker assumption than the one in previous works \cite{yong} and \cite{huangjh} where the weighting matrices are assumed to be positive definite. Furthermore, another thing to note is that the results obtained in this paper can be reduced to classic stochastic LQ control case including the solvability condition and the stabilization condition.

The rest of the paper is organized as follows. As the preliminary work for stabilization control, the finite horizon mean-field LQ control problem is firstly investigated in Section II. In Section III, we are devoted to solve infinite horizon mean-field LQ control and stabilization problems. Some numerical examples are provided in Section IV. This paper is concluded in Section V. Finally, relevant proofs are given in Appendices.

Throughout this paper, the following notations and definition will be used.

\textbf{Notations and definition}: Superscript $'$ signifies the transpose of a matrix. $\mathcal{R}^{n}$ represents the $n$-dimensional Euclidean space; $I_{n}$ denotes the unit matrix with rank $n$; Real symmetric matrix $A>0$ (or $\geq 0$) is used to indicate that $A$ is strictly positive definite (or positive semi-definite). $B^{-1}$ represents the inverse of real matrix $B$, and $C^\dag$ means the Moore-Penrose inverse of $C$. $\{\Omega,\mathcal{F},\mathcal{P},\{\mathcal{F}_{t}\}_{t\geq 0}\}$ denotes a complete probability space, with natural filtration $\{\mathcal{F}_{t}\}_{t\geq 0}$ generated by the standard Brownian motion $W_{t}$ and system initial state augmented by all the $\mathcal{P}$-null sets. $E[\cdot|\mathcal{F}_{t}]$ means the conditional expectation with respect to $\mathcal{F}_{t}$. $\mathcal{I}_{V}$ denotes the indicator function of set $V$ with $w\in V$, $\mathcal{I}_{V}=1$, otherwise $\mathcal{I}_{V}=0$. $a.s.$ means in the sense of `almost surely'.

\begin{definition}\label{def:1}
For random vector $x$, if $E(x'x)=0$, we call it zero random vector, i.e., $x=0$, $a.s.$.
\end{definition}

\section{Finite Horizon Stochastic Mean-Field LQ Control Problem}
%\label{sec:2}
%\subsection{Problem Formulation and Preliminaries }
%\subsubsection{Problem Formulation}

%it is necessary to stabilization problem is converted into the optimalthe optimal controller is the only candidate to stabilize a linear dynamic system in the sense of `necessary and sufficient` \\

Consider the linear continuous-time stochastic mean-field system as follows:
\begin{equation}\label{ps1}
\left\{ \begin{array}{ll}
dx_{t}=(Ax_{t}+\bar{A}Ex_{t}+Bu_{t}+\bar{B}Eu_{t})dt\\
~~~~~+(Cx_{t}+\bar{C}Ex_{t}+Du_{t}+\bar{D}Eu_{t})dW_{t},\\
x_{0}=\zeta,\\
\end{array} \right.
\end{equation}
where $x_{t}\in\mathcal{R}^{n}$, $u_{t}\in \mathcal{R}^{m}$ are the system state process and the control process, respectively. The coefficients $A,\bar{A},C,\bar{C}\in \mathcal{R}^{n\times n}$, and  $B,\bar{B},D,\bar{D}\in \mathcal{R}^{n\times m}$ are known deterministic coefficient matrices. $W_{t}$ is one dimensional standard Brownian motion, defined on a complete filtered probability space $(\Omega, \mathcal{F}, \mathcal{P}, \{\mathcal{F}_{t}\}_{t\geq 0})$. $\mathcal{F}_{t}$ is the natural filtration generated by $W_{t}$ and the initial state $\zeta$ augmented by all the $\mathcal{P}$-null sets. $E$ is the mathematical expectation.

By taking expectations on both sides of \eqref{ps1}, we have that
\begin{equation}\label{ps20}
  dEx_{t}=\left[(A+\bar{A})Ex_{t}+(B+\bar{B})Eu_{t}\right]dt.
\end{equation}

Associated with system equation \eqref{ps1}, the cost function is given as:
\begin{align}\label{ps2}
  J_{T}&=E\Big\{\int_{0}^{T}\Big[x_{t}'Qx_{t}+(Ex_{t})'\bar{Q}Ex_{t}+u_{t}'Ru_{t}\notag\\
  &+(Eu_{t})'\bar{R}Eu_{t}\Big]dt+x_{T}'P_Tx_{T}+(Ex_{T})'\bar{P}_TEx_{T}\Big\},
\end{align}
where $Q,\bar{Q},R,\bar{R}$, $P_T,\bar{P}_T$ are deterministic symmetric matrices with appropriate dimensions.

The admissible control set is defined as
\begin{align}\label{adm}
  \mathcal{U}[0,T]&=\Big\{u:[0,T]\times \Omega\rightarrow \mathcal{R}^m\Big|u_{t}~ \text{is} ~\mathcal{F}_{t}-\text{adapted},\notag\\
  &\text{and}~ E\int_{0}^{T}|u_{s}|^{2}ds<\infty  \Big\}.
\end{align}

It can be easily shown that by using the contraction mapping theorem, for arbitrary $(x_0, u_t)\in \mathcal{R}^n \times \mathcal{U}[0,T]$, mean-field SDE \eqref{ps1} admits a unique solution, see \cite{yong}.

Like in classical control theory, it is clear that a linear mean-field system is mean-square staibilizable if and only if there exists optimal controller to stabilize the system.
Thus, in order to investigate the stabilization, the first step is to derive the optimal controller for the mean-filed systems. The LQ control problem of finite horizon is stated as follows:
\begin{problem}\label{prob:prob1}
For system \eqref{ps1} associated with cost function \eqref{ps2}, find $\mathcal{F}_{t}$-adapted optimal controller $u_{t}\in \mathcal{U}[0,T]$ to minimize cost function \eqref{ps2}.
\end{problem}

The following standard assumption for the weighting matrices is made below to solve Problem \ref{prob:prob1}.
\begin{assumption}\label{ass:ass1}
For $t\in [0,T]$, $Q\geq 0$, $Q+\bar{Q}\geq 0$, $R\geq0$, $R+\bar{R}\geq 0$ and $P_T\geq 0$, $P_T+\bar{P}_T\geq 0$ in cost function \eqref{ps2}.
\end{assumption}

%\subsubsection{Preliminaries} ({\color{red} delete this subsection?})
%Before we explore the solvability of Problem \ref{prob:prob1}, we present the following lemma first which will be used throughout this paper.
%\begin{lemma}\label{lem:lemma01} For arbitrary real symmetric matrix $M$, we have
%
%1) For any nonzero random vector $x\in \mathcal{R}^{n}$ (i.e., $E(x'x)\neq 0$), $E(x'Mx)\geq0$ if and only if~$M\geq 0$.
%
%2) For any deterministic $x$, i.e., $x=Ex\neq 0$, $x'Mx\geq 0$ if and only if $M\geq 0$.
%
%3) For any nonzero $x$ with $Ex=0$, $x'Mx\geq 0$ if and only if $M\geq 0$.
%\end{lemma}
%
%\begin{proof}
%The proof is straightforward and is omitted here.
%\end{proof}

%\begin{remark}\label{rem:rem2}
%If "$\geq$" in Lemma \ref{lem:lemma01} and Remark \ref{rem:rem1} is replaced by "$\leq$", "$<$", "$>$" or "$=$", the conclusions remain valid.
%\end{remark}

\subsection{Maximum Principle}
The necessary condition for system \eqref{ps1} to minimize cost function \eqref{ps2}, i.e., maximum principle, is introduced in this section which serves as a basic tool in solving Problem \ref{prob:prob1}.

From \eqref{adm} we know that $\mathcal{U}[0,T]$ is closed convex subset of $\mathcal{R}^n$, then the maximum principle can be derived by using convex variational method. In the case of control set being non-convex, the maximum principle can be studied by using spike variation, which will not be discussed here.

%The maximum principle (necessary condition) is presented as follows.

\begin{theorem}\label{thm:maximum}
The optimal controller $u_{t}$ minimizing $J_T$ satisfies the following equilibrium equation:
\begin{align}\label{uu}
 0&= Ru_{t}+\bar{R}Eu_{t}+E\Bigg\{\left[
  \begin{array}{cc}
   \hspace{-2mm}  B\hspace{-2mm}\\
    \hspace{-2mm}0\hspace{-2mm}
  \end{array}
\hspace{-2mm}\right]'p_{t}+\left[
  \begin{array}{cc}
   \hspace{-2mm}  D\hspace{-2mm}\\
    \hspace{-2mm}0\hspace{-2mm}
  \end{array}
\hspace{-2mm}\right]'q_{t}\hspace{-1mm}\notag\\
&~~~~~+E\Big\{\left[
  \begin{array}{cc}
   \hspace{-2mm}  \bar{B}\hspace{-2mm}\\
    \hspace{-2mm}B+\hspace{-1mm}\bar{B}\hspace{-2mm}
  \end{array}
\hspace{-2mm}\right]'p_{t}+\hspace{-1mm}\left[
  \begin{array}{cc}
   \hspace{-2mm}  \bar{D}\hspace{-2mm}\\
    \hspace{-2mm}0\hspace{-2mm}
  \end{array}
\hspace{-2mm}\right]'q_{t}\Big\}\Bigg|\mathcal{F}_{t}
  \Bigg\},
\end{align}
and $p_{t}$, $q_{t}$ satisfy the following backward stochastic differential equation (BSDE).
\begin{equation}\label{bsde}
\left\{ \begin{array}{ll}
dp_{t}=-\Bigg\{\left[
  \begin{array}{cc}
   \hspace{-2mm} A\hspace{-2mm} & \hspace{-2mm}\bar{A} \\
   \hspace{-2mm} 0    \hspace{-2mm}            & \hspace{-2mm}A\hspace{-1mm}+\hspace{-1mm}\bar{A}\\
  \end{array}
\hspace{-2mm}\right]'p_{t}\hspace{-1mm}+\hspace{-1mm}\left[
  \begin{array}{cc}
   \hspace{-2mm} C\hspace{-2mm} & \hspace{-2mm}\bar{C} \\
   \hspace{-2mm} 0    \hspace{-2mm}            & \hspace{-2mm}0\\
  \end{array}
\hspace{-2mm}\right]'q_{t}\\
~~~~+\left[
  \begin{array}{cc}
   \hspace{-2mm}  I_{n}\hspace{-2mm}\\
    \hspace{-2mm}0\hspace{-2mm}
  \end{array}
\hspace{-2mm}\right](Qx_{t}+\bar{Q}Ex_{t})\Bigg\}dt+q_{t}dW_{t},\\
p_{T}=\left[
  \begin{array}{cc}
   \hspace{-2mm} P_{T}\hspace{-2mm} & \hspace{-2mm}\bar{P}_{T}^{(1)} \\
   \hspace{-2mm} \bar{P}_{T}^{(2)}    \hspace{-2mm}            & \hspace{-2mm}\bar{P}_{T}^{(3)}\\
  \end{array}
\hspace{-2mm}\right]\left[
  \begin{array}{cc}
   \hspace{-2mm}  x_{T}\hspace{-2mm}\\
    \hspace{-2mm}Ex_{T}\hspace{-2mm}
  \end{array}
\hspace{-2mm}\right],
\end{array} \right.
\end{equation}
where $P_{T}$, $\bar{P}_{T}^{(1)}$ are given in cost function \eqref{ps2} and $\bar{P}_{T}^{(2)}=\bar{P}_{T}^{(3)}=0$. The adjoint equation \eqref{bsde} together with system state \eqref{ps1} form the system of FBSDE.

\end{theorem}
\begin{proof}
See Appendix A.
\end{proof}

\begin{remark}Obviously, the presented maximum principle is different from the one given in previous works \cite{yong} and \cite{jli2} in the following aspects:
 \begin{itemize}
 \item The costate defined in this paper has different dimension from the costate defined in previous works of \cite{yong} and \cite{jli2}.
\item The adjoint equation of previous works (\cite{yong} and \cite{jli2}) are mean-field BSDE, i.e., the mathematical expectation of the costate $EY(s), EZ(s)$ are involved in the BSDE which are more complicated than the BSDE of this paper. The mathematical expectations of the costate are not involved in \eqref{bsde}, which provides the convenience for dealing with the stabilization problem in the next section. Actually, the Lyapunov functional candidate for stabilization is to be defined with the solution to the FBSDE given in this paper, and thus the necessary and sufficient stabilization conditions can be derived.

 \item  It is noted that the term $Eu_t$ was not involved in both the dynamic system and the cost function in \cite{jli2}. % it seems difficult which indicates that the results in \cite{jli2} cannot be applied to solve the problems investigated in this paper.
%Different from the mean-field BSDE in previous literatures \cite{yong} and \cite{jli2}, the costate $p_t\in \mathcal{R}^{2n}$ and \eqref{bsde} is not mean-field BSDE. The maximum principle is one of the key techniques for us to solve in essentially the stabilization and problem for mean-field systems.
\end{itemize}
\end{remark}

\subsection{Solution to Problem \ref{prob:prob1}}
Once the maximum principle is derived in Theorem \ref{thm:maximum}, we are in the position to state the main result in this section.

\begin{theorem}\label{thm:main}
Under Assumption \ref{ass:ass1}, Problem \ref{prob:prob1} is uniquely solved if and only if $\Upsilon_{t}^{(1)}>0$ and $\Upsilon_{t}^{(2)}>0$ for $t\in[0,T]$, where $\Upsilon_{t}^{(1)}$ and $\Upsilon_{t}^{(2)}$ are given as:
\begin{align}
\Upsilon_{t}^{(1)}&=R+D'P_{t}D,\label{upsi1}\\
\Upsilon_{t}^{(2)}&=R+\bar{R}+(D+\bar{D})'P_{t}(D+\bar{D}),\label{upsi2}
\end{align}
and $P_{t}$ and $\bar{P}_{t}$ satisfy the coupled Riccati equation:
\begin{align}
-\dot{P}_{t}&=Q\hspace{-1mm}+\hspace{-1mm}P_{t}A\hspace{-1mm}+\hspace{-1mm}A'P_{t}\hspace{-1mm}+\hspace{-1mm}C'P_{t}C\hspace{-1mm}-\hspace{-1mm}[M_{t}^{(1)}]'[\Upsilon_{t}^{(1)}]^{-1}M_{t}^{(1)}, \label{Ric1}\\
-\dot{\bar{P}}_{t}&=\bar{Q}+P_{t}\bar{A}+\bar{A}'P_{t}+(A+\bar{A})'\bar{P}_{t}+\bar{P}_{t}(A+\bar{A})\notag\\
&+\bar{C}'P_{t}\bar{C}+C'P_{t}\bar{C}+\bar{C}'P_{t}C\notag\\
&+[M_{t}^{(1)}]'[\Upsilon_{t}^{(1)}]^{-1}M_{t}^{(1)}-[M_{t}^{(2)}]'[\Upsilon_{t}^{(2)}]^{-1}M_{t}^{(2)}, \label{Ric2}
\end{align}
with final condition $P_T, \bar{P}_T$ given in \eqref{ps2}, where  $M_{t}^{(1)}$, $M_{t}^{(2)}$ are given by
\begin{align}
M_{t}^{(1)}&=B'P_{t}+D'P_{t}C,\label{mt1}\\
M_{t}^{(2)}&=(B\hspace{-1mm}+\hspace{-1mm}\bar{B})'(P_{t}+\bar{P}_{t})\hspace{-1mm}+\hspace{-1mm}(D+\bar{D})'P_{t}
(C+\bar{C}).\label{mt2}
\end{align}

In this case, the optimal controller $u_{t}$ can be presented as,
\begin{equation} \label{opti}
u_{t}=K_{t}x_{t}+\bar{K}_{t}Ex_{t},
\end{equation}
where
\begin{align}
K_{t}&=-[\Upsilon_{t}^{(1)}]^{-1}M_{t}^{(1)},\label{kt}\\
\bar{K}_{t}&=-\left\{[\Upsilon_{t}^{(2)}]^{-1}M_{t}^{(2)}-[\Upsilon_{t}^{(1)}]^{-1}M_{t}^{(1)}\right\},\label{barkt}
\end{align}

and the optimal cost function is given as:
\begin{equation}\label{op}
J_{T}^{*}=E(x_{0}'P_{0}x_{0})+Ex_{0}'\bar{P}_{0}Ex_{0}.
\end{equation}
Moreover, the optimal costate $p_{t}$ of \eqref{bsde} and state $x_{t}$, $Ex_{t}$ obeys the following relationship (the solution to FBSDE \eqref{ps1} and \eqref{bsde}),
\begin{align}\label{pt}
p_{t}=\left[
  \begin{array}{cc}
   \hspace{-2mm} P_{t}\hspace{-2mm} & \hspace{-2mm}\bar{P}_{t}^{(1)} \\
   \hspace{-2mm} \bar{P}_{t}^{(2)}    \hspace{-2mm}            & \hspace{-2mm}\bar{P}_{t}^{(3)}\\
  \end{array}
\hspace{-2mm}\right]\left[
  \begin{array}{cc}
   \hspace{-2mm}  x_{t}\hspace{-2mm}\\
    \hspace{-2mm}Ex_{t}\hspace{-2mm}
  \end{array}
\hspace{-2mm}\right],
\end{align}
where $P_{t}$ satisfies Riccati equation \eqref{Ric1}, and $\bar{P}_{t}^{(1)}+\bar{P}_{t}^{(2)}+\bar{P}_{t}^{(3)}=\bar{P}_{t}$ satisfies Riccati equation \eqref{Ric2}.

\end{theorem}

\begin{proof}
See Appendix B.
\end{proof}

\begin{remark}
It is noted that the obtained results in Theorem \ref{thm:main} differs from previous works \cite{yong} in the following aspects:

\begin{itemize}
  \item Theorem \ref{thm:main} provides the necessary and sufficient solvability condition in terms of the positive definiteness of matrices $\Upsilon_{t}^{(1)}$ and $\Upsilon_{t}^{(2)}$ which is new to our knowledge. While the solvability conditions in previous works \cite{yong} was given in terms of operator.

\item The weighting matrices $R$ and $R+\bar{R}$ in cost function (3) are only assumed to be positive semi-definite, which weakened the assumption of positive definiteness in \cite{yong}.

  \item The solution to FBSDE composed of \eqref{ps1} and \eqref{uu}-\eqref{bsde} is presented in this paper as in \eqref{pt}, which differs from that in previous works \cite{yong}. It can be seen that \eqref{pt} will play an important role in solving stabilization problem in infinite horizon case.
\end{itemize}

\end{remark}

%{\color{red}
%\begin{remark}\begin{itemize}
%\item For arbitrary initial state $x_0$,  it is easy to see that the solvability and controller design for closed loop control are the same as the open loop control.
%\item The closed-loop controller satisfies the maximum principle for open loop control in any case since the closed loop control is a special case of open loop control. The closed-loop controller can be solved by the
%maximum principle developed in Theorem \ref{thm:maximum} in the case of arbitrary initial values.
%\end{itemize}
%\end{remark}}

\section{Infinite Horizon Stochastic Mean-field Problem}
\label{sec:3}

\subsection{Problem Formulation}
The optimal control in infinite horizon case and stabilization problem  will be explored in this section.

In order to consider the stabilization problem for mean-field systems,
%the time-invariant mean-field system is introduced as below:
%\begin{equation}\label{ps10}
%\left\{ \begin{array}{ll}
%dx_{t}=(Ax_{t}\hspace{-1mm}+\hspace{-1mm}\bar{A}Ex_{t}\hspace{-1mm}+\hspace{-1mm}Bu_{t}\hspace{-1mm}+\hspace{-1mm}\bar{B}Eu_{t})dt
%\hspace{-1mm}+\hspace{-1mm}(Cx_{t}\hspace{-1mm}+\hspace{-1mm}\bar{C}Ex_{t}\hspace{-1mm}+\hspace{-1mm}Du_{t}\hspace{-1mm}+\hspace{-1mm}\bar{D}Eu_{t})dW_{t},\\
%x_{0}=\xi,\\
%\end{array} \right.
%\end{equation}
%where the coefficient matrices $A,~\bar{A},~B,~\bar{B},~C,~\bar{C},~D,~\bar{D}$ are all constant with compatible dimensions.
%
%Accordingly,
 the infinite horizon cost function is described as:
\begin{equation}\label{ps200}\begin{split}
  J\hspace{-1mm}=\hspace{-1mm}E\int_{0}^{\infty}[x_{t}'Qx_{t}\hspace{-1mm}+\hspace{-1mm}
  (Ex_{t})'\bar{Q}Ex_{t}\hspace{-1mm}+\hspace{-1mm}u_{t}'Ru_{t}\hspace{-1mm}+\hspace{-1mm}(Eu_{t})'\bar{R}Eu_{t}]dt,
\end{split}\end{equation}
where weighting matrices $Q$, $\bar{Q}$, $R$, $\bar{R}$ are symmetric with appropriate dimensions.

The admissible control set for the infinite horizon case is given as below:
\begin{align}\label{adm2}
  \mathcal{U}[0,\infty)&=\Big\{u:[0,T]\times \Omega\rightarrow \mathcal{R}^m\Big|u_{t}~ \text{is} ~\mathcal{F}_{t}-\text{adapted},\notag\\
  &\text{and}~ E\int_{0}^{T}|u_{s}|^{2}ds<\infty  \Big\}.
\end{align}

To investigate the stabilization conditions for mean-field system, the following basic assumption is made throughout this section.
\begin{assumption}\label{ass:ass2}
$R\geq0$, $R+\bar{R}\geq0$, and $Q\geq 0$, $Q+\bar{Q}\geq 0$.
\end{assumption}
\begin{remark}
The weighting matrices $R, R+\bar{R}$ in Assumption \ref{ass:ass2} are just required to be positive semi-definite, which is weaker than previous works \cite{huangjh}, including the traditional stabilization results \cite{fll}, \cite{zhangw} and \cite{zhw}.
\end{remark}

Before stating the problem to be studied, we give several definitions at first.
\begin{definition}
System \eqref{ps1} with $u_{t}=0$ is called asymptotically mean square stable if for any initial values $x_{0}$, there holds $\lim_{t\rightarrow +\infty}E(x_{t}'x_{t})=0.$

\end{definition}

\begin{definition}
System \eqref{ps1} is said to be stabilizable in the mean square sense if there exists $\mathcal{F}_{t}$-adapted controller $u_{t}\in \mathcal{U}[0,\infty)$, such that for any random vector $x_{0}$, the closed loop of system \eqref{ps1} is asymptotically mean square stable.
\end{definition}

\begin{definition}\label{def:4}
Consider the following mean-field stochastic system
\begin{equation}\label{mf}
 \left\{ \begin{array}{ll}
 dx_{t}=(Ax_{t}+\bar{A}Ex_{t})dt+(Cx_{t}+\bar{C}Ex_{t})dW_{t},\\
\mathcal{Y}_{t}=\mathcal{Q}^{1/2}\mathbb{X}_{t},
\end{array} \right.
\end{equation}
where $\mathbb{X}_{t}=\left[
  \begin{array}{cc}
   \hspace{-2mm} x_{t}-Ex_{t}\hspace{-2mm}\\
    Ex_{t}     \hspace{-2mm}           \\
  \end{array}
\hspace{-2mm}\right]$, and $\mathcal{Q}=\left[
  \begin{array}{cc}
   \hspace{-2mm} Q\hspace{-2mm}&\hspace{-2mm} 0\hspace{-2mm}\\
    0     \hspace{-2mm}     &\hspace{-2mm} Q+\bar{Q}\hspace{-2mm}      \\
  \end{array}
\hspace{-2mm}\right]$.

System \eqref{mf}, $(A,\bar{A},C,\bar{C},\mathcal{Q}^{1/2})$ for simplicity, is said to be exact observable, if for any $T\geq 0$,
\begin{equation*}
  \mathcal{Y}_{t}= 0, ~a.s.,  ~\forall~ t\in[0,T]~\Rightarrow~x_{0}=0,
\end{equation*}
where $\mathcal{Y}_{t}=0$, $x_{0}=0$ implies $E(\mathcal{Y}_{t}'\mathcal{Y}_{t})=0, E(x_0'x_0)=0$, whose meaning is given in Definition \ref{def:1}.

\end{definition}

\begin{definition}\label{def:det1}
For system \eqref{mf}, $(A,\bar{A},C,\bar{C},\mathcal{Q}^{1/2})$ is said to be exact detectable, if for any $T\geq 0$,
\begin{equation*}
  \mathcal{Y}_{t}= 0, ~a.s., ~\forall~ t\in[0,T]~\Rightarrow~\lim_{t\rightarrow+\infty}E(x_{t}'x_{t})=0.
\end{equation*}
\end{definition}
 Like in previous works \cite{huang}, \cite{zhw} and \cite{zhangw}, the exact observability (exact detectability) introduced in Definition \ref{def:4} (Definition \ref{def:det1}) is a basic condition in tackling the stabilization problems for stochastic control systems.  Now we make the following assumptions.

\begin{assumption}\label{ass:ass4}
$(A,\bar{A},C,\bar{C},\mathcal{Q}^{1/2})$ is exact detectable.%cankaowenxian
\end{assumption}

\begin{assumption}\label{ass:ass3}
$(A,\bar{A},C,\bar{C},\mathcal{Q}^{1/2})$ is exact observable.%cankaowenxian
\end{assumption}

\begin{remark}
From Definition \ref{def:4} and Definition \ref{def:det1}, obviously we can conclude that if $(A,\bar{A},C,\bar{C},\mathcal{Q}^{1/2})$ is exact observable, then $(A,\bar{A},C,\bar{C},\mathcal{Q}^{1/2})$ is exact detectable. Thus the exact detectability assumption is weaker than the exact observability assumption.
\end{remark}

At the end of this section, the stabilization and control problems for mean-field systems of infinite horizon case to be investigated can be described as:
\begin{problem}\label{prob:prob2}
Find the casual and $\mathcal{F}_{t}$-adapted controller $u_{t}\in\mathcal{U}[0,\infty)$ to minimize cost function \eqref{ps200} and stabilize mean-field system \eqref{ps1} in the mean square sense.
\end{problem}

\subsection{Solution to Problem \ref{prob:prob2}}
For convenience of discussions in the below, we now re-denote $P_{t}$, $\bar{P}_{t}$, $K_{t}$ and $\bar{K}_{t}$ in \eqref{Ric1}, \eqref{Ric2}, \eqref{kt} and \eqref{barkt} respectively as $K_{t}(T)$, $\bar{K}_{t}(T)$, $P_{t}(T)$ and $\bar{P}_{t}(T)$ for purpose of showing they are dependent of the terminal time $T$. As usual, $P_{T}(T)$ and $\bar{P}_{T}(T)$ in \eqref{ps2} are set to be zero for infinite horizon control.

Firstly, the following lemma will be introduced without proof, which can be easily obtained from the proof of Theorem \ref{thm:main}.

\begin{lemma}\label{lem:lemma3}
Under Assumption \ref{ass:ass2}, suppose the following coupled Riccati equation is solvable:
\begin{align}
-\dot{P}_{t}(T)&=Q+P_{t}(T)A+A'P_{t}(T)\hspace{-1mm}+\hspace{-1mm}C'P_{t}(T)C\notag\\
&-[M_{t}^{(1)}(T)]'[\Upsilon_{t}^{(1)}(T)]^{\dag}M_{t}^{(1)}(T), \label{2-Ric1}\\
-\dot{\bar{P}}_{t}(T)&=\bar{Q}+P_{t}(T)\bar{A}+\bar{A}'P_{t}(T)+(A+\bar{A})'\bar{P}_{t}(T)\notag\\
&+\bar{P}_{t}(T)(A+\bar{A})+\bar{C}'P_{t}(T)\bar{C}+C'P_{t}(T)\bar{C}\notag\\
&+\bar{C}'P_{t}(T)C+[M_{t}^{(1)}(T)]'[\Upsilon_{t}^{(1)}(T)]^{\dag}M_{t}^{(1)}(T)\notag\\
&-[M_{t}^{(2)}(T)]'[\Upsilon_{t}^{(2)}(T)]^{\dag}M_{t}^{(2)}(T), \label{2-Ric2}
\end{align}
with final condition $P_T(T)=\bar{P}_T(T)=0$, and the following regular condition holds:
\begin{align}\label{refg}
 \Upsilon_{t}^{(i)}(T) [\Upsilon_{t}^{(i)}(T)]^{\dag}M_{t}^{(i)}(T)=M_{t}^{(i)}(T),~i=1,2,
\end{align}
where
\begin{align}
\Upsilon_{t}^{(1)}(T)&=R+D'P_{t}(T)D,\label{4-upsi1}\\
\Upsilon_{t}^{(2)}(T)&=R+\bar{R}+(D+\bar{D})'P_{t}(T)(D+\bar{D}),\label{4-upsi2}\\
M_{t}^{(1)}(T)&=B'P_{t}(T)+D'P_{t}(T)C,\label{4-mt1}\\
M_{t}^{(2)}(T)&=(B+\bar{B})'[P_{t}(T)+\bar{P}_{t}(T)]\notag\\
&+(D+\bar{D})'P_{t}(T)(C+\bar{C}).\label{4-mt2}
\end{align}
Then the cost function \eqref{ps2} with $P_{T}(T)=\bar{P}_{T}(T)=0$ is minimized by the following controller
\begin{equation} \label{2-opti}
u_{t}=\mathcal{K}_{t}(T)x_{t}+\bar{\mathcal{K}}_{t}(T)Ex_{t},
\end{equation}
where $\mathcal{K}_{t}(T), \bar{\mathcal{K}}_{t}(T)$ are given as
\begin{align}\label{2-k2}
&\mathcal{K}_t(T)=-[\Upsilon^{(1)}_t(T)]^{\dag}M^{(1)}_t(T),\\
&\bar{\mathcal{K}}_t(T)=-\{[\Upsilon^{(2)}_t(T)]^{\dag}M^{(2)}_t(T)-[\Upsilon^{(1)}_t(T)]^{\dag}M^{(1)}_t(T)\}.\notag
\end{align}
 The optimal cost function is
 \begin{align}\label{16}
   J_T^* & =E[x_{0}'P_{0}(T)x_{0}]+Ex_{0}'\bar{P}_{0}(T)Ex_{0}.
 \end{align}
\end{lemma}

\begin{lemma}\label{lem3}
Under Assumption \ref{ass:ass2}, the solution to Riccati equation \eqref{2-Ric1}-\eqref{2-Ric2} satisfy $P_t(T)\geq 0$ and $P_t(T)+\bar{P}_t(T)\geq 0$ for $t\in[0,T]$.
\end{lemma}
\begin{proof}
  See Appendix C.
\end{proof}

%\begin{theorem}\label{thm:theorem2}
%Under Assumptions \ref{ass:ass2} and \ref{ass:ass3}, if system \eqref{ps1} is mean square stabilizable, then we have:
%
%(i) $P_{t}(T)$ and $\bar{P}_{t}(T)$ are convergent as $T\rightarrow +\infty$ for any $t\geq 0$, i.e.,
%$$\lim_{T\rightarrow +\infty}P_{t}(T)=P,~\lim_{T\rightarrow +\infty}\bar{P}_{t}(T)=\bar{P},$$
%and $P$ and $\bar{P}$ are the solution to the following coupled AREs as
%\begin{align}
%  0&=Q+PA+A'P+C'PC-[M^{(1)}]'[\Upsilon^{(1)}]^{-1}M^{(1)},\label{are1}\\
%  0&=\bar{Q}+P\bar{A}+\bar{A}'P+(A+\bar{A})'\bar{P}+\bar{P}(A+\bar{A})+C'P\bar{C}+\bar{C}'PC+\bar{C}'P\bar{C}\notag\\
%  &+[M^{(1)}]'[\Upsilon^{(1)}]^{-1}M^{(1)}-[M^{(2)}]'[\Upsilon^{(2)}]^{-1}M^{(2)},\label{are2}
%\end{align}
%where $\Upsilon^{(1)}$, $M^{(1)}$, $\Upsilon^{(2)}$, $M^{(2)}$ are given by
%\begin{align}
%\Upsilon^{(1)}&=R+D'PD\geq R>0,\label{up1}\\
%M^{(1)}&=B'P+D'PC,\label{hh1}\\
%\Upsilon^{(2)}&=R+\bar{R}+(D+\bar{D})'P(D+\bar{D})\geq R+\bar{R}>0,\label{up2}\\
%M^{(2)}&=(B+\bar{B})'(P+\bar{P})+(D+\bar{D})'P(C+\bar{C}).\label{hh2}
%\end{align}
%
%(ii) The solution to \eqref{are1}-\eqref{are2} satisfy $P>0$ and $P+\bar{P}>0$.
%\end{theorem}
%\begin{proof}
%See Appendix D.
%\end{proof}

In this section, we introduce the following coupled ARE:
\begin{align}
  0&=Q+PA+A'P+C'PC-[M^{(1)}]'[\Upsilon^{(1)}]^{\dag}M^{(1)},\label{are1}\\
  0&=\bar{Q}+P\bar{A}+\bar{A}'P+(A+\bar{A})'\bar{P}+\bar{P}(A+\bar{A})\notag\\
  &+C'P\bar{C}+\bar{C}'PC+\bar{C}'P\bar{C}\notag\\
  &+[M^{(1)}]'[\Upsilon^{(1)}]^{\dag}M^{(1)}-[M^{(2)}]'[\Upsilon^{(2)}]^{\dag}M^{(2)},\label{are2}
\end{align}
with the following regular condition holds:
\begin{align}\label{rgh}
 \Upsilon^{(i)}[\Upsilon^{(i)}]^{\dag}M^{(i)}=M^{(i)}, i=1,2,
\end{align}
where $\Upsilon^{(1)}$, $M^{(1)}$, $\Upsilon^{(2)}$, $M^{(2)}$ are given by
\begin{align}
\Upsilon^{(1)}&=R+D'PD,\label{up1}\\
M^{(1)}&=B'P+D'PC,\label{hh1}\\
\Upsilon^{(2)}&=R+\bar{R}+(D+\bar{D})'P(D+\bar{D}),\label{up2}\\
M^{(2)}&=(B+\bar{B})'(P+\bar{P})+(D+\bar{D})'P(C+\bar{C}).\label{hh2}
\end{align}

In view of \eqref{up1}-\eqref{hh2}, there holds
\begin{align}
[M^{(1)}]'[\Upsilon^{(1)}]^{\dag}M^{(1)}&\hspace{-1mm}=\hspace{-1mm}-[M^{(1)}]'\mathcal{K}\hspace{-1mm}-\hspace{-1mm}
\mathcal{K}'M^{(1)}\hspace{-1mm}-\hspace{-1mm}\mathcal{K}'\Upsilon^{(1)}\mathcal{K},\label{equ1}\\
[M^{(2)}]'[\Upsilon^{(2)}]^{\dag}M^{(2)}&=-[M^{(2)}]'(\mathcal{K}+\bar{\mathcal{K}})-(\mathcal{K}+\bar{\mathcal{K}})'M^{(2)}\notag\\
&-(\mathcal{K}+\bar{\mathcal{K}})'\Upsilon^{(2)}(\mathcal{K}+\bar{\mathcal{K}}),\label{equ2}
\end{align}
where $\mathcal{K}$ and $\bar{\mathcal{K}}$ satisfy
\begin{align}
\mathcal{K}&=-[\Upsilon^{(1)}]^{\dag}M^{(1)},\label{k1}\\
\bar{\mathcal{K}}&=-\{[\Upsilon^{(2)}]^{\dag}M^{(2)}-[\Upsilon^{(1)}]^{\dag}M^{(1)}\}.\label{k2}
\end{align}

Therefore, for the sake of convenience, we rewrite the coupled ARE \eqref{are1}-\eqref{are2} as follows:
\begin{align}
  0&=\mathbf{Q}+\mathbf{A}'P+P\mathbf{A}+\mathbf{C}'P\mathbf{C},\label{4-ly1}\\
  0&=\bar{\mathbf{Q}}+\bar{\mathbf{A}}'(P+\bar{P})+(P+\bar{P})\bar{\mathbf{A}}+\bar{\mathbf{C}}'P\bar{\mathbf{C}},\label{4-ly2}
\end{align}
where
\begin{align}\label{symlk}
 &\mathbf{Q}=Q+\mathcal{K}'R\mathcal{K}\geq 0, \mathbf{A}=A+B\mathcal{K}, \mathbf{C}=C+D\mathcal{K},\notag\\
 &\bar{\mathbf{Q}}=Q+\bar{Q}+(\mathcal{K}+\bar{\mathcal{K}})'(R+\bar{R})(\mathcal{K}+\bar{\mathcal{K}})\geq 0,\notag\\
 &\bar{\mathbf{A}}=A+\bar{A}+(B+\bar{B})(\mathcal{K}+\bar{\mathcal{K}}), \notag\\ &\bar{\mathbf{C}}=C+\bar{C}+(D+\bar{D})(\mathcal{K}+\bar{\mathcal{K}}).
\end{align}

In what follows, the definition concerning the solution to \eqref{are1}-\eqref{are2} is introduced as below:
\begin{definition}\label{defas}
If the coupled ARE \eqref{are1}-\eqref{are2} has solution $P$, $\bar{P}$ satisfying $P\geq 0$ and $P+\bar{P}\geq 0$ ($P>0$ and $P+\bar{P}>0$), we call the coupled ARE \eqref{are1}-\eqref{are2} has a positive semi-definite (positive definite) solution.
\end{definition}

Before stating the main results of this paper, the following Lemma will be introduced at first, which is essential in exploring the stabilizing conditions of mean-field systems.
\begin{lemma}\label{ess}Under Assumptions \ref{ass:ass2} and \ref{ass:ass4}, the following two assertions hold:

 1) The following system $(\tilde{\mathbb{A}}, \tilde{\mathbb{C}}, \tilde{\mathcal{Q}}^{1/2})$ is exact detectable:
 \begin{equation}\label{mf01}
 \left\{ \begin{array}{ll}
 d\mathbb{X}_{t}=\tilde{\mathbb{A}}\mathbb{X}_{t}dt+\tilde{\mathbb{C}}\mathbb{X}_{t}dW_{t},~~\mathbb{X}_{0},\\
\tilde{\mathcal{Y}}_{t}=\tilde{\mathcal{Q}}^{1/2}\mathbb{X}_{t},
\end{array} \right.
\end{equation}
where $\tilde{\mathbb{A}}\hspace{-1mm}=\hspace{-1mm}\left[\hspace{-2mm}
  \begin{array}{cc}
   \mathbf{A}& 0\\
   0        & \mathbf{\bar{A}}    \\
  \end{array}
\hspace{-2mm}\right]$, $\tilde{\mathbb{C}}\hspace{-1mm}=\hspace{-1mm}\left[\hspace{-2mm}
  \begin{array}{cc}
   \mathbf{C}& \mathbf{\bar{C}}\\
   0        & 0\\
  \end{array}
\hspace{-2mm}\right]$, $\tilde{\mathcal{Q}}\hspace{-1mm}=\hspace{-1mm}\left[\hspace{-1mm}
  \begin{array}{cc}
   \mathbf{Q}& 0\\
   0        & \bar{\mathbf{Q}}     \\
  \end{array}
\hspace{-1mm}\right]\geq  0$, and $\mathbb{X}_{t}=\left[\hspace{-1mm}
  \begin{array}{cc}
   \hspace{-1mm} x_{t}-Ex_{t}\hspace{-1mm}\\
   \hspace{-1mm} Ex_{t}     \hspace{-1mm}           \\
  \end{array}
\hspace{-1mm}\right], t\in[0,T]$. i.e., for any $T\geq 0$,
$\tilde{\mathcal{Y}}_{t}= 0,~a.s.,  ~\forall~ t\in[0,T]\Rightarrow\lim_{t\rightarrow+\infty}E(\mathbb{X}_{t}'\mathbb{X}_{t})=0.$

 2) If $\mathbb{P}\geq 0$, then $\mathbb{X}_0$ is an unobservable state of system $(\tilde{\mathbb{A}}, \tilde{\mathbb{C}}, \tilde{\mathcal{Q}}^{1/2})$ if and only if $E(\mathbb{X}_0'\mathbb{P}\mathbb{X}_0)=0$, where $\mathbb{P}=\left[\hspace{-1mm}
  \begin{array}{cc}
   P& 0\\
   0        & P+\bar{P}      \\
  \end{array}
\hspace{-1mm}\right]$ and $P,\bar{P}$ satisfy \eqref{are1}-\eqref{are2}.
\end{lemma}
\begin{proof}
  See Appendix D.
\end{proof}

\begin{remark}\label{rem:5}
  Following from Lemma \ref{ess} and its proof, it can be verified that the exact observability of system \eqref{mf01} $(\tilde{\mathbb{A}}, \tilde{\mathbb{C}}, \tilde{\mathcal{Q}}^{1/2})$ can be implied from the exact observability of $(A,\bar{A},C,\bar{C},\mathcal{Q}^{1/2})$ given in Assumption \ref{ass:ass3}.
\end{remark}

The main results of this section will be presented, which can be stated as the following two theorems. One is under the assumption of exact detectability assumption (Assumption \ref{ass:ass4}), the other is under the exact observability (Assumption \ref{ass:ass3}) which is stronger than exact detectability.
\begin{theorem}\label{thm:succeed}
Consider mean-field system \eqref{ps1} and cost function \eqref{ps200}, under Assumptions \ref{ass:ass2} and \ref{ass:ass4} (exact detectability), mean-field system \eqref{ps1} is mean square stabilizable if and only if the coupled ARE \eqref{are1}-\eqref{are2} admits a unique positive semi-definite solution.

In this case, the stabilizing controller is given by
\begin{align}\label{control}
  u_{t}=\mathcal{K}x_{t}+\bar{\mathcal{K}}Ex_{t},
\end{align}
where $\mathcal{K}, \bar{\mathcal{K}}$ are as in \eqref{k1}-\eqref{k2}.

Furthermore, the stabilizing controller \eqref{control} minimizes the cost function \eqref{ps200}, and the optimal cost function is given as
\begin{equation}\label{cost}
  J^{*}=E(x_{0}'Px_{0})+Ex_{0}'\bar{P}Ex_{0}.
\end{equation}
\end{theorem}
\begin{proof}
See Appendix E.
\end{proof}

\begin{theorem}\label{thm:succeed2}
Under Assumptions \ref{ass:ass2} and \ref{ass:ass3} (exact observability), system \eqref{ps1} is mean square stabilizable if and only if the coupled ARE \eqref{are1}-\eqref{are2} has a unique positive definite solution. In this case, the stabilizing controller is given by \eqref{control}. Moreover, the controller \eqref{control} also  minimizes the cost function \eqref{ps200} and the optimal cost function is given by \eqref{cost}.
\end{theorem}
\begin{proof}
See Appendix F.
\end{proof}

\begin{remark}
Theorems \ref{thm:succeed} and \ref{thm:succeed2} provide a thorough solution to stabilization and control problems for linear mean-field systems under basic assumptions. It is worth of pointing out that the stabilization results in Theorem \ref{thm:succeed} and \ref{thm:succeed2} are obtained under the condition of $R\geq 0$, which is a more relaxed condition than the standard assumption of $R>0$ in classical control \cite{fll}, \cite{zhangw} and \cite{zhw}.
\end{remark}

\section{Numerical Examples}\label{sec:4}
%\subsection{Finite Horizon Case}
%Without loss of generality, we might as well consider system \cref{ps1} and cost function \cref{ps2} with the following time-invariant parameters: {\color{blue} (showing the iff condition ??????????, how about $\bar{Q}$ is negative? )}
%\begin{align*}
%A_{t}&=1.1,\bar{A}_{t}=0.2,B_{t}=0.4,\bar{B}_{t}=0.1,C_{t}=0.9,\bar{C}_{t}=0.5,\\
%D_{t}&=0.8,\bar{D}_{t}=0.2,Q_{t}=2,\bar{Q}_{t}=1,R_{t}=1,\bar{R}_{t}=1,G=1,\bar{G}=0,T=10.
%\end{align*}
%
%By \cref{thm:main}, the optimal controller is given as:
%\begin{equation*}
%  u_{t}=-\frac{1.12P_{t}}{1+0.64P_{t}}x_{t}-\left(\frac{1.9P_{t}+0.5\bar{P}_{t}}{2+P_{t}}-\frac{1.12P_{t}}{1+0.64P_{t}}{\color{blue} x_{t} ???} \right)Ex_{t},
%\end{equation*}
%where $P_{t}$ and $\bar{P}_{t}$ are the solution to the coupled Riccati equation: {\color{blue} (give the analytical solution ??????)}
%\begin{align*}
%-\dot{P}_{t}&=\hspace{-1mm}2+3.01P_{t}-\frac{(1.12P_{t})^{2}}{1+0.64P_{t}},\\
%-\dot{\bar{P}}_{t}&=\hspace{-1mm}1+0.4P_{t}+2.6\bar{P}_{t}+\frac{(1.12P_{t})^{2}}{1+0.64P_{t}}-\frac{(1.9P_{t}+0.5\bar{P}_{t})^{2}}{2+P_{t}}.
%\end{align*}
%
%\subsection{Infinite Horizon Case}

Consider system \eqref{ps1} and cost function \eqref{ps200} with
\begin{align*}
A&=0.2,\bar{A}=0.4,B=0.6,\bar{B}=0.2,C=0.1,\bar{C}=0.7,\\
D&=0.9,\bar{D}=0.3,Q=1,\bar{Q}=1,R=1,\bar{R}=1.
\end{align*}

Obviously, Assumptions \ref{ass:ass2} and \ref{ass:ass3} are satisfied, also the regular condition \eqref{refg} holds. By solving the coupled ARE \eqref{are1}-\eqref{are2}, we know that $P=18.4500$ and $\bar{P}=-1.6609$ is the unique solution satisfying $P=18.4500>0$ $P+\bar{P}=16.7891>0$.  Thus, according to Theorem \ref{thm:succeed}, the system is mean square stabilizable. From \eqref{k1}-\eqref{k2}, $\mathcal{K}=-0.7986$ and $\bar{\mathcal{K}}=-0.2915$ can be obtained, i.e., the controller is given from  \eqref{control} as $u_{t}=-0.7986x_{t}-0.2915Ex_{t}$.

The simulation result is shown in Fig.\ref{fig:2}. It can be seen that system state $x_{t}$ converges to zero in the mean square sense with the controller given above, as expected.
\begin{figure}[htbp]
  \centering
  \includegraphics[width=0.47\textwidth]{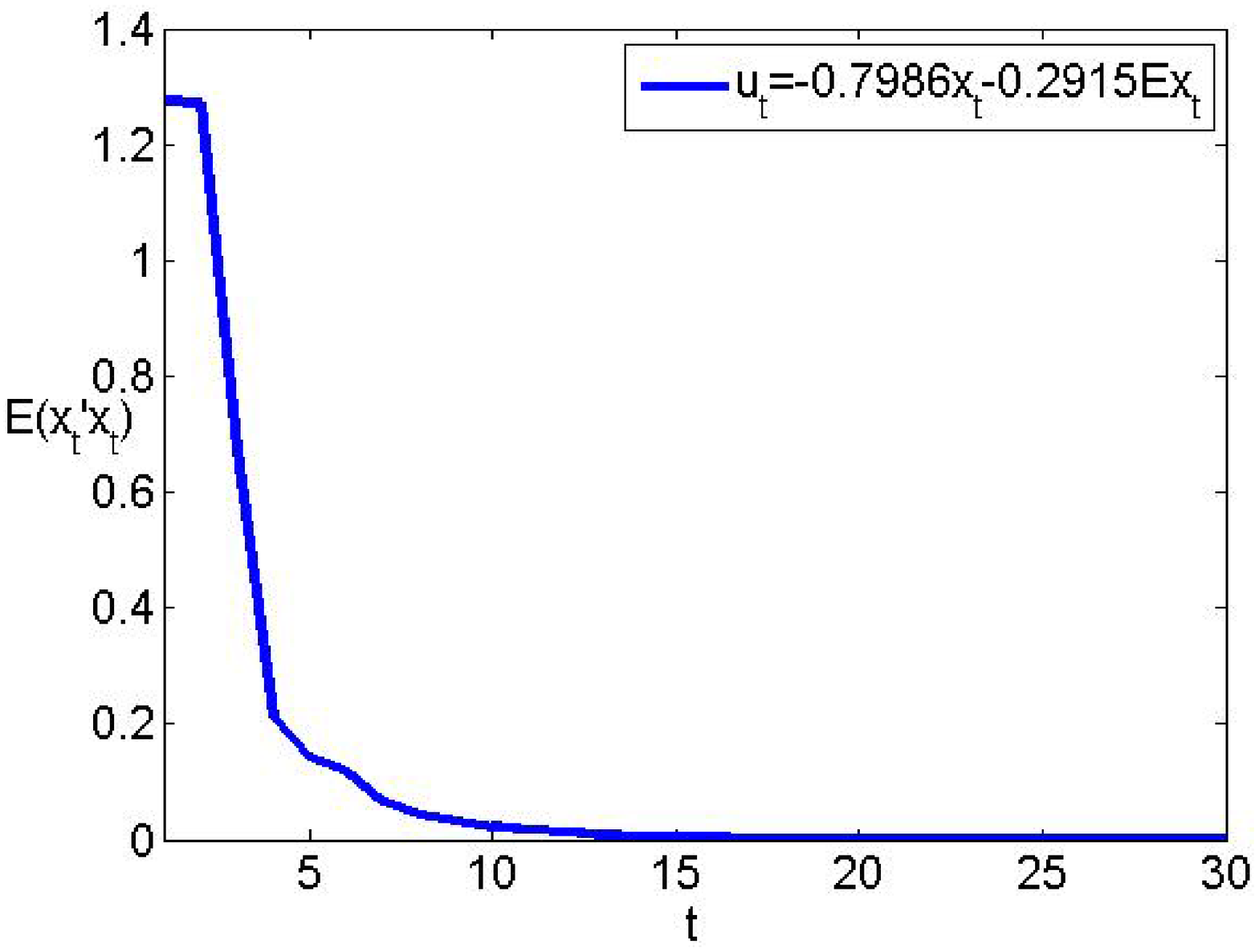}
  \caption{The mean square stabilization of mean-field system.}\label{fig:2}
\end{figure}

On the other hand, another example is presented in order to show necessity condition given in Theorem \ref{thm:succeed} and Theorem \ref{thm:succeed2}. Consider the mean-field system \eqref{ps1} and cost function \eqref{ps200} with the following coefficients:
\begin{align}\label{coe}
A&=2,\bar{A}=0.8,B=1,\bar{B}=1.2,C=0.1,\bar{C}=0.6,\notag\\
D&=-0.8,\bar{D}=-0.2,Q=1,\bar{Q}=1,R=1,\bar{R}=3.
\end{align}

By solving the coupled ARE \eqref{are1}-\eqref{are2} with coefficients \eqref{coe}, we find that $P$ has two negative roots as $-0.2356$ and $-2.4679$. From Theorem \ref{thm:succeed} we know that system \eqref{ps1} is not stabilizable in the mean square sense.

In fact, in the case $P=-2.4679$, from \eqref{are2} we know that there is no real roots for $\bar{P}$. While in the case of $P=-0.2356$, by solving \eqref{are2}, $\bar{P}$ has two real roots as $4.7637$ and $-0.0869$.

In the latter case, with $P=-0.2356$ and $\bar{P}=4.7637$, $\mathcal{K}$ and $\bar{\mathcal{K}}$ can be respectively calculated as $\mathcal{K}=0.2552$ and $\bar{\mathcal{K}}=-0.1106$. Similarly, when $P=-0.2356$ and $\bar{P}=-0.0869$, we can obtain $\mathcal{K}=0.2552$ and $\bar{\mathcal{K}}=-2.9453$. Accordingly, the controller can be designed as $u_{t}=0.2552x_{t}-0.1106Ex_{t}$ and $u_{t}=0.2552x_{t}-2.9453Ex_{t}$ respectively.

With the designed controller, the simulation results of the state trajectories are shown as in Fig. \ref{fig:3} and Fig. \ref{fig:4}. It can be seen that the system states are divergent.
\begin{figure}[htbp]
  \centering
  \includegraphics[width=0.47\textwidth]{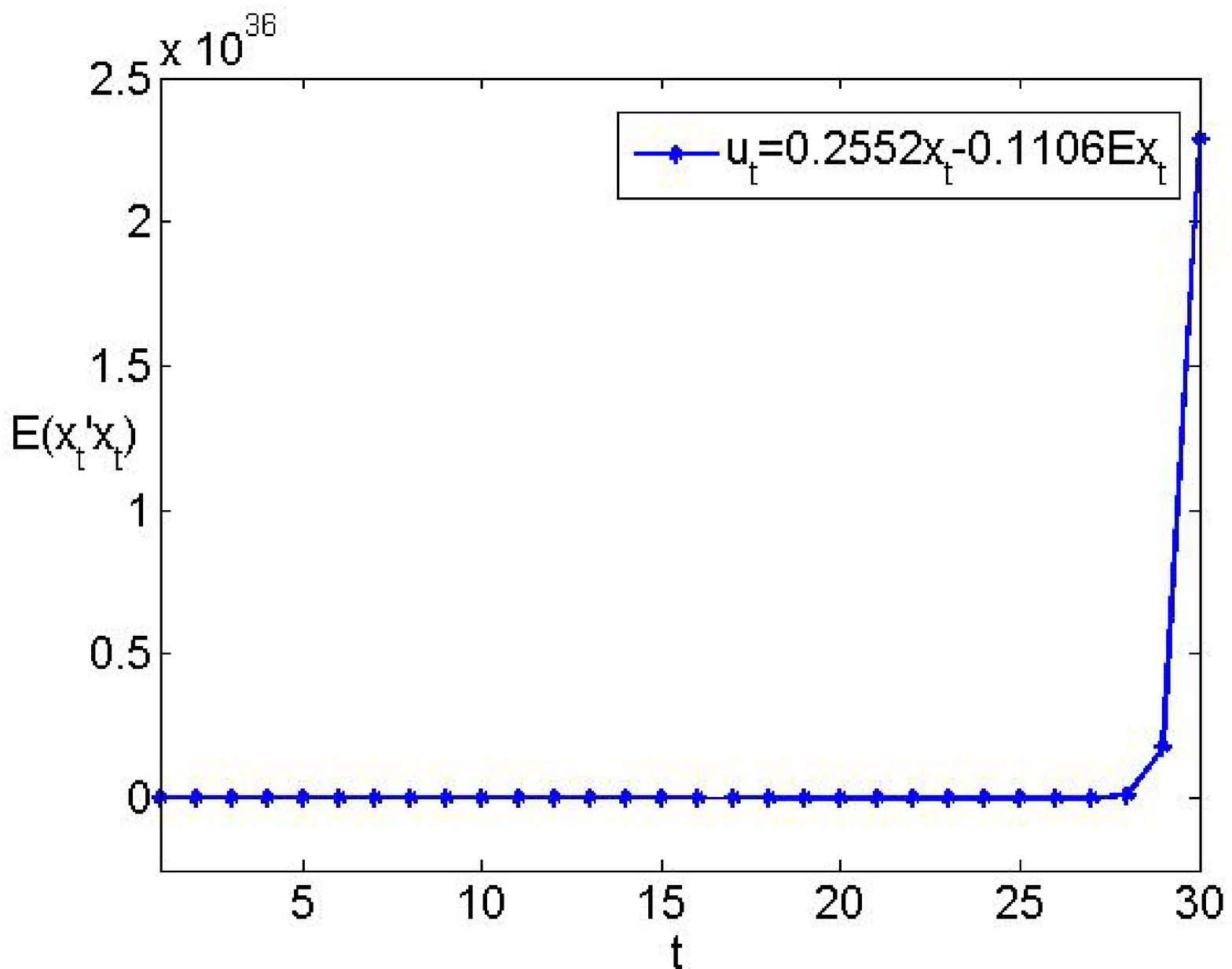}
  \caption{Simulations for system state trajectory $E(x_{t}'x_{t})$.}\label{fig:3}
\end{figure}
\begin{figure}[htbp]
  \centering
  \includegraphics[width=0.47\textwidth]{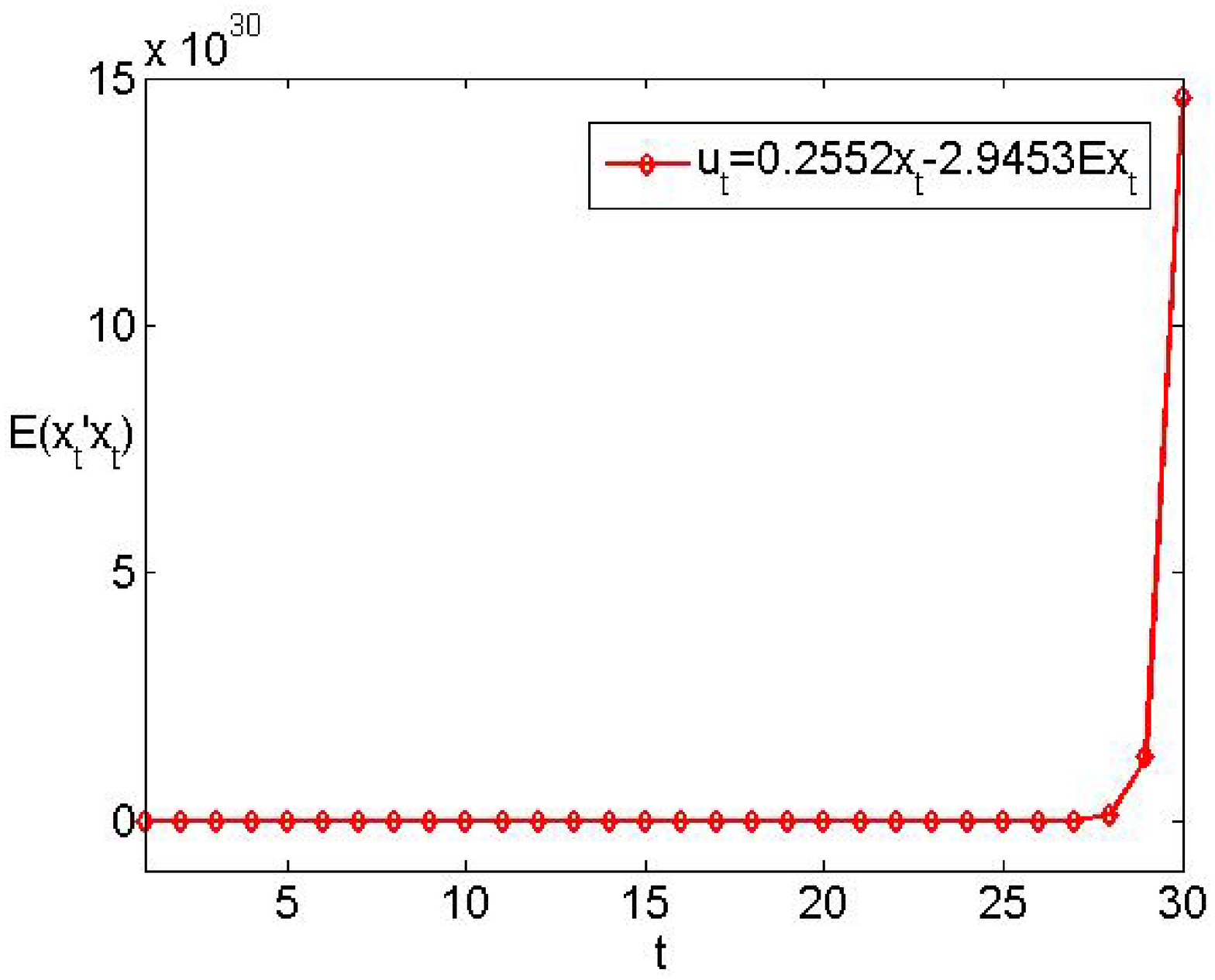}
  \caption{Simulations for system state trajectory $E(x_{t}'x_{t})$.}\label{fig:4}
\end{figure}
\section{Conclusions}\label{sec:5}
In this paper, the stabilization and control problems for linear continuous-time mean-field systems have been studied. By using the maximum principle and the solution to the FBSDE developed in this paper, we have presented the sufficient and necessary solvability condition of the finite horizon mean-field LQ control based on a coupled Riccati equation which is derived from the FBSDE. By defining an Lyapunov functional with the optimal cost function and applying the solution of the FBSDE, we have explored  the necessary and sufficient stabilization conditions for mean-field systems. It has been shown that, under exact detectability assumption, the system is mean square stabilizable if and only if the coupled ARE admits a unique positive semi-definite solution. Furthermore, we have also shown that, under exact observability assumption, the mean-field system is mean square stabilizable if and only if the coupled ARE admits a unique positive definite solution.

%This paper, together with our earlier works \cite{zxl}, \cite{zxl2}, \cite{qz}, propose a novel approach to stochastic control for multiplicative noise systems. It can solve a number of challenging problems such as stochastic control with input delay, control for mean-field systems.
%problems  The results explored in this paper give a complete solution to continuous-time mean-field LQ control and stabilization problems. For further study, problems of optimal control and stabilization for mean-field systems with random coefficients needs to be concerned.

\section*{Appendix A: Proof of Theorem \ref{thm:maximum}}\label{app:A}
\begin{proof}
Clearly from \eqref{adm} we know that $\mathcal{U}[0,T]$ is non-empty, closed and convex subset of $\mathcal{R}^m$, any $u_{t}\in \mathcal{U}[0,T]$ is called admissible control. For $u_{t},~\delta u_{t}\in\mathcal{U}[0,T]$ and $\varepsilon\in[0,1]$, there always holds that $u_{t}^{\varepsilon}=u_{t}+\varepsilon \delta u_{t}\in \mathcal{U}[0,T]$.

Let $x_{t}^{\varepsilon}$, $J_{T}^{\varepsilon}$ be the corresponding state and cost function with $u_{t}^{\varepsilon}$, and $x_{t}$, $J_{T}$ are respectively the state and cost function associated with $u_{t}$.

Denote $\delta x_{t}=x_{t}^{\varepsilon}-x_{t}$ and set $\delta X_{t}=\left[
  \begin{array}{cc}
   \hspace{-1mm} \delta x_{t} \hspace{-1mm}\\
    \hspace{-1mm}E\delta x_{t}\hspace{-1mm}
  \end{array}
\hspace{-1mm}\right]$, then by \eqref{ps1} and \eqref{ps20}, the following assertion holds,
\begin{align}\label{asser}
 d\delta X_{t}&=\left[\mathcal{A}\delta X_{t}+\mathcal{B}\varepsilon\delta u_{t}+\bar{\mathcal{B}}E\varepsilon\delta u_{t}\right]dt\notag\\
&+\left[\mathcal{C}\delta X_{t}+\mathcal{D}\varepsilon\delta u_{t}+\bar{\mathcal{D}}E\varepsilon\delta u_{t}\right]dW_{t},
\end{align}
where
 \begin{align}\label{label}
 \mathcal{A}&=\left[
  \begin{array}{cc}
   \hspace{-1mm} A\hspace{-1mm} & \hspace{-1mm}\bar{A} \\
   \hspace{-1mm} 0    \hspace{-1mm}            & \hspace{-1mm}A+\bar{A}\\
  \end{array}
\hspace{-1mm}\right],~\mathcal{B}=\left[
  \begin{array}{cc}
   \hspace{-1mm}  B\hspace{-1mm}\\
    \hspace{-1mm}0\hspace{-1mm}
  \end{array}
\hspace{-1mm}\right],~\bar{\mathcal{B}}=\left[
  \begin{array}{cc}
 \hspace{-1mm}\bar{B} \\
  \hspace{-1mm}B+\bar{B}\\
  \end{array}
\hspace{-1mm}\right],\notag\\
 \mathcal{C}&=\left[
  \begin{array}{cc}
   \hspace{-1mm} C\hspace{-1mm} & \hspace{-1mm}\bar{C} \\
   \hspace{-1mm} 0    \hspace{-1mm}            & \hspace{-1mm}0\\
  \end{array}
\hspace{-1mm}\right],~\mathcal{D}=\left[
  \begin{array}{cc}
   \hspace{-1mm}  D\hspace{-1mm}\\
    \hspace{-1mm}0\hspace{-1mm}
  \end{array}
\hspace{-1mm}\right],~\bar{\mathcal{D}}=\left[
  \begin{array}{cc}
 \hspace{-1mm}\bar{D} \\
  \hspace{-1mm}0\\
  \end{array}
\hspace{-1mm}\right].
\end{align}

By \emph{Theorem 6.14} in \cite{yongj}, from \eqref{asser} we know $\delta x_{t}$ can be obtained as
\begin{align}\label{aug2}
\delta x_{t}
&=[I_{n}~0]\Phi_{t}\int_{0}^{t}\Phi_{s}^{-1}[(\mathcal{B}\varepsilon\delta u_{s}+\bar{\mathcal{B}}E\varepsilon\delta u_{s})\notag\\
&~~-\mathcal{C}(\mathcal{D}\varepsilon\delta u_{s}+\bar{\mathcal{D}}E\varepsilon\delta u_{s})]ds\notag\\
&+[I_{n}~0]\Phi_{t}\int_{0}^{t}\Phi_{s}^{-1}(\mathcal{D}\varepsilon\delta u_{s}+\bar{\mathcal{D}}E\varepsilon\delta u_{s})dW_{s}.
\end{align}
where $\Phi_{t}$ is the unique solution of the following SDE
\begin{equation}\label{phi}
\left\{ \begin{array}{ll}
d\Phi_{t}=\mathcal{A}\Phi_{t}dt+\mathcal{C}\Phi_{t}dW_{t},\\
\Phi_{0}=I_{2n},\\
\end{array} \right.
\end{equation}
and $\Phi_{t}^{-1}=\Psi_{t}$ exists, satisfying
\begin{equation}\label{psi}
\left\{ \begin{array}{ll}
d\Psi_{t}=\Psi_{t}(-\mathcal{A}+\mathcal{C}^{2})dt-\Psi_{t}\mathcal{C}dW_{t},\\
\Psi_{0}=I_{2n}.\\
\end{array} \right.
\end{equation}

Since the coefficient matrices in \eqref{ps1} are given deterministic and $E\int_{0}^{T}u_t'u_t dt<\infty$, thus \eqref{aug2} indicates that there exists constant $C_0>0$ satisfying
\begin{align}\label{43}E(\delta x_t'\delta x_t)<C_0 \varepsilon^2. \end{align}

Following from cost function \eqref{ps2} and using \eqref{43}, the increment of $J_{T}$, i.e., $\delta J_{T}=J_{T}^{\varepsilon}-J_{T}$ can be calculated as follows
\begin{align}\label{mp0001}
&\delta J_T=J_T^{\varepsilon}-J_{T}\notag\\
&=2E\Big\{\int_{0}^{T}\Big[x_{t}'Q\delta x_{t}+Ex_{t}'\bar{Q}E\delta x_{t}+u_{t}'R\varepsilon\delta u_{t}\notag\\
&~~+Eu_{t}'\bar{R}E\varepsilon\delta u_{t}\Big]dt
+x_{T}'P_T\delta x_{T}+Ex_{T}'\bar{P}_TE\delta x_{T}\Big\}+o(\varepsilon)\notag\\
&=2E\Big\{\int_{0}^{T}\Big[(x_{t}'Q\hspace{-1mm}+\hspace{-1mm}Ex_{t}'\bar{Q})\delta x_{t}+(u_{t}'R+\hspace{-1mm}Eu_{t}'\bar{R})\varepsilon\delta u_{t}\Big]dt\notag\\
&~~+(x_{T}'P_T+Ex_{T}'\bar{P}_T)\delta x_{T}\Big\}+o(\varepsilon),
\end{align}
where $o(\varepsilon)$ means infinitesimal of higher order with $\varepsilon$.

By plugging \eqref{aug2} into \eqref{mp0001}, then $\delta J_{T}$ can be given as follows
\begin{align}\label{var}
&~~~\delta J_{T}=J_{T}^{\varepsilon}-J_{T}\notag\\
&=2E\Bigg\{\int_{0}^{T}\Big\{\Big[\int_{s}^{T}(x_{t}'Q+Ex_{t}'\bar{Q})[I_{n}~0]\Phi_{t}dt\notag\\
&~~+(x_{T}'P_T+Ex_{T}'\bar{P}_T)[I_{n}~0]\Phi_{T}\Big]\Phi_{s}^{-1}(\mathcal{B}-\mathcal{C}\mathcal{D})\notag\\
&~~+E\Big\{\Big[\int_{s}^{T}(x_{t}'Q\hspace{-1mm}+\hspace{-1mm}Ex_{t}'\bar{Q})[I_{n}~0]\Phi_{t}dt\notag\\
&~~+(x_{T}'P_T+Ex_{T}'\bar{P}_T))[I_{n}~0]\Phi_{T}\Big]
\Phi_{s}^{-1}(\bar{\mathcal{B}}-\mathcal{C}\bar{\mathcal{D}})\Big\}\notag\\
&~~+u_{s}'R+Eu_{s}'\bar{R}\Big\}\varepsilon\delta u_{s}ds\Bigg\}\notag\\
&+2E\Big\{\Big[\int_{0}^{T}(x_{t}'Q+Ex_{t}'\bar{Q})[I_{n}~0]\Phi_{t}dt\notag\\
&~~+(x_{T}'P_T+Ex_{T}'\bar{P}_T))[I_{n}~0]\Phi_{T}\Big]\notag\\
&~~\times \int_{0}^{T}\Phi_{s}^{-1}(\mathcal{D}\varepsilon\delta u_{s}+\hspace{-1mm}\bar{\mathcal{D}}E\varepsilon\delta u_{s})dW_{s}\Big\}\hspace{-1mm}+\hspace{-1mm}o(\varepsilon),
\end{align}
where the following relationship has been used
\begin{align*}
&E\Big\{\int_{0}^{T}(x_{t}'Q+Ex_{t}'\bar{Q})[I_{n}~0]\Phi_{t}\notag\\
&~~\times \int_{0}^{t}\Phi_{s}^{-1}(\mathcal{D}\varepsilon\delta u_{s}+\bar{\mathcal{D}}E\varepsilon\delta u_{s})dW_{s}dt\Big\}\\
&=E\Big\{\int_{0}^{T}(x_{t}'Q+Ex_{t}'\bar{Q})[I_{n}~0]\Phi_{t}dt \notag\\ &~~\times E\Big[\int_{0}^{T}\Phi_{s}^{-1}(\mathcal{D}\varepsilon\delta u_{s}+\bar{\mathcal{D}}E\varepsilon\delta u_{s})dW_{s}\Big|\mathcal{F}_{t}\Big]\Big\}\\
&=E\Big\{\int_{0}^{T}(x_{t}'Q+Ex_{t}'\bar{Q})[I_{n}~0]\Phi_{t}dt\notag\\
&~~\times \int_{0}^{T}\Phi_{s}^{-1}(\mathcal{D}\varepsilon\delta u_{s}+\bar{\mathcal{D}}E\varepsilon\delta u_{s})dW_{s}\Big\}.
\end{align*}

Denote $\xi=\int_{0}^{T}(x_{t}'Q+Ex_{t}'\bar{Q})[I_{n}~0]\Phi_{t}dt+(x_{T}'P_T+Ex_{T}'\bar{P}_T))[I_{n}~0]\Phi_{T}$, then $E[\xi|\mathcal{F}_{s}]$ is a martingale with respect to $s$. By the Martingale Representation Theorem, there exists a unique $\mathcal{F}_{t}$-adapted process $\eta_{t}$ such that $E[\xi|\mathcal{F}_{s}]=E[\xi]+\int_{0}^{s}\eta_{t}'dW_{t}$. With $s=T$ and $\xi=E[\xi|\mathcal{F}_{T}]$, we have that
\begin{align}\label{xi}\xi=E[\xi]+\int_{0}^{T}\eta_{t}'dW_{t}.\end{align}

Substituting \eqref{xi} into the last but one term in \eqref{var}, we can obtain
\begin{align}\label{ls}
&E\hspace{-1mm}\Bigg\{\hspace{-1.6mm}\left[\int_{0}^{T}\hspace{-1mm}(x_{t}'Q\hspace{-1mm}+\hspace{-1mm}Ex_{t}'\bar{Q})[I_{n}~0]\Phi_{t}dt\hspace{-1mm}
+\hspace{-1mm}(x_{T}'P_T\hspace{-1mm}+\hspace{-1mm}Ex_{T}'\bar{P}_T))[I_{n}~0]\Phi_{T}\hspace{-1mm}\right]\notag\\
&\times \int_{0}^{T}\Phi_{s}^{-1}(\mathcal{D}\varepsilon\delta u_{s}+\bar{\mathcal{D}}E\varepsilon\delta u_{s})dW_{s}\Bigg\}\notag\\
&=E\Bigg\{E[\xi]\int_{0}^{T}\Phi_{s}^{-1}(\mathcal{D}\varepsilon\delta u_{s}+\bar{\mathcal{D}}E\varepsilon\delta u_{s})dW_{s}\notag\\
&+\int_{0}^{T}\eta_{t}'dW_{t}\int_{0}^{T}\Phi_{s}^{-1}(\mathcal{D}\varepsilon\delta u_{s}+\bar{\mathcal{D}}E\varepsilon\delta u_{s})dW_{s}\Bigg\}\notag\\
&=E\left\{\int_{0}^{T}\eta_{s}'\Phi_{s}^{-1}(\mathcal{D}\varepsilon\delta u_{s}+\bar{\mathcal{D}}E\varepsilon\delta u_{s})ds\right\}\notag\\
&=E\left\{\int_{0}^{T}\left\{\eta_{s}'\Phi_{s}^{-1}\mathcal{D}+E[\eta_{s}'\Phi_{s}^{-1}\bar{\mathcal{D}}]\right\}\varepsilon\delta u_{s}ds\right\}.
\end{align}

Denote
\begin{align}\label{gs}
H'_{s}&=\Big[\int_{s}^{T}(x_{t}'Q+Ex_{t}'\bar{Q})[I_{n}~0]\Phi_{t}dt\notag\\
&~~+(x_{T}'P_T+Ex_{T}'\bar{P}_T)[I_{n}~0]\Phi_{T}\Big]\Phi_{s}^{-1}(\mathcal{B}-\mathcal{C}\mathcal{D})\notag\\
&+E\Big\{\Big[\int_{s}^{T}(x_{t}'Q+Ex_{t}'\bar{Q})[I_{n}~0]\Phi_{t}dt\notag\\
&~~+(x_{T}'P_T+Ex_{T}'\bar{P}_T)[I_{n}~0]\Phi_{T}\Big]\Phi_{s}^{-1}(\bar{\mathcal{B}}-\mathcal{C}\bar{\mathcal{D}})\Big\}\notag\\
&+u_{s}'R\hspace{-1mm}+\hspace{-1mm}Eu_{s}'\bar{R}+\eta_{s}'\Phi_{s}^{-1}\mathcal{D}
\hspace{-1mm}+\hspace{-1mm}E[\eta_{s}'\Phi_{s}^{-1}\bar{\mathcal{D}}].
\end{align}

By using \eqref{gs} and \eqref{ls}, equation \eqref{var} can be rewritten as
\begin{align}
\delta J_{T}&=2E\int_{0}^{T}H'_{s}\varepsilon\delta u_{s}ds+o(\varepsilon)\notag\\
&=2E\int_{0}^{T}E[H'_{s}|\mathcal{F}_{s}]\varepsilon\delta u_{s}ds+o(\varepsilon).
\end{align}

By the arbitrary of $\delta u_{s}$, we know the necessary condition of minimizing \eqref{ps2} is $E[H'_{s}|\mathcal{F}_{s}]=0$, i.e.,
\begin{align}\label{eq}
0&=Ru_{s}+\bar{R}Eu_{s}\notag\\
&+\hspace{-1mm}E\Bigg\{(\mathcal{B}\hspace{-1mm}-\hspace{-1mm}\mathcal{C}\mathcal{D})'(\Phi_{s}')^{-1}\notag\\
&\times \hspace{-1mm}\Big[\int_{s}^{T}\Phi_{t}'\hspace{-1mm}\left[
  \begin{array}{cc}
 \hspace{-2mm}I_{n}\hspace{-3mm} \\
  \hspace{-2mm}0\hspace{-3mm}\\
  \end{array}
\hspace{-2mm}\right]\hspace{-1mm}(Qx_{t}\hspace{-1mm}+\hspace{-1mm}\bar{Q}Ex_{t})dt\hspace{-1mm}+\hspace{-1mm}\Phi_{T}'\left[
  \begin{array}{cc}
 \hspace{-2mm}I_{n}\hspace{-3mm} \\
  \hspace{-2mm}0\hspace{-3mm}\\
  \end{array}
\hspace{-2mm}\right]\hspace{-1mm}(P_Tx_{T}\hspace{-1mm}+\hspace{-1mm}\bar{P}_TEx_{T})\Big]\notag\\
&+\hspace{-1mm}E\Big\{(\bar{\mathcal{B}}\hspace{-1mm}-\hspace{-1mm}\mathcal{C}\bar{\mathcal{D}})'(\Phi_{s}')^{-1} \notag\\ &\times \hspace{-1mm}\Big[\int_{s}^{T}\hspace{-1mm}\Phi_{t}'\hspace{-1mm}\left[
  \begin{array}{cc}
 \hspace{-2mm}I_{n}\hspace{-3mm} \\
  \hspace{-2mm}0\hspace{-3mm}\\
  \end{array}
\hspace{-2mm}\right]\hspace{-1mm}(Qx_{t}\hspace{-1mm}+\hspace{-1mm}\bar{Q}Ex_{t})dt\hspace{-1mm}+\hspace{-1mm}\Phi_{T}'\left[
  \begin{array}{cc}
 \hspace{-2mm}I_{n} \hspace{-3mm}\\
  \hspace{-2mm}0\hspace{-3mm}\\
  \end{array}
\hspace{-2mm}\right](P_Tx_{T}\hspace{-1mm}+\hspace{-1mm}\bar{P}_TEx_{T})\Big]\hspace{-1mm}\Big\}\notag\\
&+\hspace{-1mm}\mathcal{D}'(\Phi_{s}')^{-1}\eta_{s}+E[\bar{\mathcal{D}}'(\Phi_{s}')^{-1}\eta_{s}]\Bigg|\mathcal{F}_{s}\Bigg\}.
\end{align}

Define $p_{s}$, $q_{s}$ respectively as follows,
\begin{align}
p_{s}&\hspace{-1mm}=\hspace{-1mm}E\Bigg\{(\Phi_{s}')^{-1} \Big[\int_{s}^{T}\Phi_{t}'\left[
  \begin{array}{cc}
 \hspace{-2mm}I_{n}\hspace{-3mm} \\
  \hspace{-2mm}0\hspace{-3mm}\\
  \end{array}
\hspace{-1mm}\right](Qx_{t}\hspace{-1mm}+\hspace{-1mm}\bar{Q}Ex_{t})dt
\hspace{-1mm}\notag\\
&+\hspace{-1mm}\Phi_{T}'\left[
  \begin{array}{cc}
 \hspace{-2mm}I_{n}\hspace{-3mm} \\
  \hspace{-2mm}0\hspace{-3mm}\\
  \end{array}
\hspace{-1mm}\right](P_Tx_{T}+\bar{P}_TEx_{T})\Big]\Bigg|\mathcal{F}_{s}\Bigg\},\label{ps}\\
q_{s}&\hspace{-1mm}=\hspace{-1mm}(\Phi_{s}')^{-1}\eta_{s}-\mathcal{C}'p_{s}.\label{qs}
\end{align}

Thus, using \eqref{ps} and \eqref{qs}, equation \eqref{eq} implies that
\begin{align}\label{eq2}
0&=Ru_{s}+\bar{R}Eu_{s}+E\Big\{\mathcal{B}'p_{s}+\mathcal{D}'q_{s}
+E[\bar{\mathcal{B}}'p_{s}+\bar{\mathcal{D}}'q_{s}]|\mathcal{F}_{s}
\Big\}.
\end{align}

Finally, applying It\^{o}'s formula, it is easy to verify that $(p_{s},q_{s})$ in \eqref{ps} and \eqref{qs} is the solution pair of \eqref{bsde}.\end{proof}

\section*{Appendix B: Proof of Theorem \ref{thm:main}}\label{app:B}

\begin{proof}
``Sufficiency": Suppose $\Upsilon_{t}^{(1)}>0$ and $\Upsilon_{t}^{(2)}>0$ for $t\in[0,T]$, we will show \emph{Problem 1} admits a unique solution.

Applying It\^{o}'s formula to $x_{t}'P_{t}x_{t}+Ex_{t}'\bar{P}_{t}Ex_{t}$, taking integral from $0$ to $T$ and then
expectation, we have
\begin{align}\label{pp}
&~~~~E\int_{0}^{T}d(x_{t}'P_{t}x_{t}+Ex_{t}'\bar{P}_{t}Ex_{t})\notag\\
&=E\int_{0}^{T}\Big[(dx_{t})'P_{t}x_{t}\hspace{-1mm}+\hspace{-1mm}x_{t}'P_{t}dx_{t}\hspace{-1mm}+\hspace{-1mm}x_{t}'\dot{P}_{t}x_{t}d_{t}+(dx_{t})'P_{t}dx_{t}\notag\\
&+(dEx_{t})'\bar{P}_{t}Ex_{t}+Ex_{t}'\dot{\bar{P}}_{t}Ex_{t}dt+Ex_{t}'\bar{P}_{t}dEx_{t}\Big]\notag\\
&=E\hspace{-1mm}\int_{0}^{T}[x_{t}'(\dot{P}_{t}\hspace{-1mm}+\hspace{-1mm}A'P_{t}\hspace{-1mm}+\hspace{-1mm}P_{t}A\hspace{-1mm}+\hspace{-1mm}C'P_{t}C)x_{t}\hspace{-1mm}
+\hspace{-1mm}u_{t}'D'P_{t}Du_{t}\notag\\
&~~+2u_{t}'(B'P_{t}\hspace{-1mm}+\hspace{-1mm}D'P_{t}C)x_{t}]dt\notag\\
&+E\int_{0}^{T}Ex_{t}'[\dot{\bar{P}}_{t}+\bar{A}'P_{t}+P_{t}\bar{A}+(A+\bar{A})'\bar{P}_{t}\notag\\
&~~+\bar{P}_{t}(A+\bar{A})
+C'P_{t}\bar{C}+\bar{C}'P_{t}C+\bar{C}'P_{t}\bar{C}]Ex_{t}dt\notag\\
&+2\int_{0}^{T}Eu_{t}'[\bar{B}'P_{t}+D'P_{t}\bar{C}+\bar{D}'P_{t}C
+\bar{D}'P_{t}\bar{C}\notag\\
&~~+(B+\bar{B})'\bar{P}_{t}]Ex_{t}dt\notag\\
&+E\int_{0}^{T}Eu_{t}'(D'P_{t}\bar{D}+\bar{D}'P_{t}D+\bar{D}'P_{t}\bar{D})Eu_{t}dt.
\end{align}

Notice Riccati equation \eqref{Ric1}-\eqref{Ric2} and \eqref{kt}-\eqref{barkt}, \eqref{pp} implies that
\begin{align}\label{cost1}
&~~E(x_{T}'P_{T}x_{T})\hspace{-1mm}+\hspace{-1mm}Ex_{T}'\bar{P}_{T}Ex_{T}\hspace{-1mm}-\hspace{-1mm}[E(x_{0}'P_{0}x_{0})+Ex_{0}'\bar{P}_{0}Ex_{0}]\notag\\
&=-E\int_{0}^{T}[x_{t}'Qx_{t}\hspace{-1mm}+\hspace{-1mm}Ex_{t}'\bar{Q}Ex_{t}\hspace{-1mm}+\hspace{-1mm}
u_{t}'Ru_{t}\hspace{-1mm}+\hspace{-1mm}Eu_{t}'\bar{R}Eu_{t}]dt\notag\\
&+E\int_{0}^{T}[u_{t}-Eu_{t}-K_{t}(x_{t}-Ex_{t})]'\Upsilon_{t}^{(1)}\notag\\
&~~\times [u_{t}-Eu_{t}-K_{t}(x_{t}-Ex_{t})]dt\notag\\
&+E\int_{0}^{T}[Eu_{t}\hspace{-1mm}-\hspace{-1mm}(K_{t}\hspace{-1mm}+\hspace{-1mm}\bar{K}_{t})Ex_{t}]'\Upsilon_{t}^{(2)}
[Eu_{t}\hspace{-1mm}-\hspace{-1mm}(K_{t}\hspace{-1mm}+\hspace{-1mm}\bar{K}_{t})Ex_{t}]dt.
\end{align}
where $ K_{t}, \bar{K}_{t}$, $\Upsilon_{t}^{(1)}, \Upsilon_{t}^{(2)}$ are given by \eqref{kt}-\eqref{barkt} and \eqref{upsi1}-\eqref{upsi2}, respectively.

Since $\Upsilon_{t}^{(1)}>0$, $\Upsilon_{t}^{(2)}>0$, from \eqref{cost1} the cost function $J_{T}$ in \eqref{ps2} can be given as
\begin{align}\label{cost2}
&J_{T}=E(x_{0}'P_{0}x_{0})+Ex_{0}'\bar{P}_{0}Ex_{0}\notag\\
&+E\int_{0}^{T}[u_{t}-Eu_{t}-K_{t}(x_{t}-Ex_{t})]'\Upsilon_{t}^{(1)}\notag\\
&~~\times [u_{t}-Eu_{t}-K_{t}(x_{t}-Ex_{t})]dt\notag\\
&+E\int_{0}^{T}[Eu_{t}-(K_{t}+\bar{K}_{t})Ex_{t}]'\Upsilon_{t}^{(2)}[Eu_{t}\hspace{-1mm}-\hspace{-1mm}
(K_{t}\hspace{-1mm}+\hspace{-1mm}\bar{K}_{t})Ex_{t}]dt\notag\\
&\geq E(x_{0}'P_{0}x_{0})+Ex_{0}'\bar{P}_{0}Ex_{0}.
\end{align}

Thus the minimum of $J_{T}$ is given by \eqref{op}, i.e.,
\begin{equation*}
  J_{T}^{*}=E(x_{0}'P_{0}x_{0})+Ex_{0}'\bar{P}_{0}Ex_{0}.
\end{equation*}

In this case, the controller will satisfy
\begin{align}
u_{t}-Eu_{t}-K_{t}(x_{t}-Ex_{t})&=0,\label{l1}\\
Eu_{t}-(K_{t}+\bar{K}_{t})Ex_{t}&=0.\label{l2}
\end{align}
Thus the optimal controller can be uniquely obtained from \eqref{l1}-\eqref{l2} as \eqref{opti}. The sufficiency proof is complete.

``Necessity": If \emph{Problem 1} has a unique solution, we will show that under Assumption \ref{ass:ass1}, $\Upsilon_t^{(1)}>0$ and $\Upsilon_t^{(2)}>0$.

Since \emph{Problem 1} is solvable, thus the FBSDE from the necessary condition (Maximum Principle) is solvable. Without loss of generality, we assume $p_{t}$ in adjoint equation \eqref{bsde} and system state $x_{t}, Ex_{t}$ admit the following linear relationship,
\begin{align}\label{ptt}
p_{t}=\left[
  \begin{array}{cc}
   \hspace{-1mm} P_{t}\hspace{-1mm} & \hspace{-1mm}\bar{P}_{t}^{(1)} \\
   \hspace{-1mm} \bar{P}_{t}^{(2)}    \hspace{-1mm}            & \hspace{-1mm}\bar{P}_{t}^{(3)}\\
  \end{array}
\hspace{-1mm}\right]\left[
  \begin{array}{cc}
   \hspace{-1mm}  x_{t}\hspace{-1mm}\\
    \hspace{-1mm}Ex_{t}\hspace{-1mm}
  \end{array}
\hspace{-1mm}\right]+\Theta_t,
\end{align}
where $P_{t}$, $\bar{P}_{t}^{(1)}$, $\bar{P}_{t}^{(2)}$, $\bar{P}_{t}^{(3)}$ and $\Theta_t$ are differential functions to be determined.

Firstly, we give the following coupled Riccati equation:
\begin{align}
 -\dot{P}_{t}&=Q\hspace{-1mm}+\hspace{-1mm}P_{t}A\hspace{-1mm}+\hspace{-1mm}A'P_{t}\hspace{-1mm}+\hspace{-1mm}C'P_{t}C-[M_{t}^{(1)}]'[\Upsilon_{t}^{(1)}]^{\dag}M_{t}^{(1)},\label{lkj}\\
-\dot{\bar{P}}_{t}&=\bar{Q}+P_{t}\bar{A}+\bar{A}'P_{t}+(A+\bar{A})'\bar{P}_{t}+\bar{P}_{t}(A+\bar{A})\notag\\
&+\bar{C}'P_{t}\bar{C}+C'P_{t}\bar{C}+\bar{C}'P_{t}C\notag\\
&
+[M_{t}^{(1)}]'[\Upsilon_{t}^{(1)}]^{\dag}M_{t}^{(1)}-[M_{t}^{(2)}]'[\Upsilon_{t}^{(2)}]^{\dag}M_{t}^{(2)}.\label{barp}
\end{align}
where $\dag$ means the Moore-Penrose inverse of a matrix, final condition $P_T, \bar{P}_T$ are given in \eqref{ps2}. $\Upsilon_{t}^{(1)}$, $\Upsilon_{t}^{(2)}$, $M_{t}^{(1)}$, $M_{t}^{(2)}$ are the same form with \eqref{upsi1}-\eqref{upsi2} and \eqref{mt1}-\eqref{mt2} and $P_t$, $\bar{P}_t$ satisfying \eqref{lkj}-\eqref{barp}.

A solution to \eqref{lkj}-\eqref{barp} is called \emph{regular}, if
\begin{align}\label{reg1}
  \Upsilon_{t}^{(i)}[\Upsilon_{t}^{(i)}]^{\dag}M_{t}^{(i)}=M_{t}^{(i)}, i=1,2.
\end{align}

Now we will show if \emph{Problem 1} is solvable and a solution to Riccati equation \eqref{lkj}-\eqref{barp} is \emph{regular}, then $\Theta_t=0$.

Assume
\begin{align}\label{Theta}
 d\Theta_t=\Theta_t^1 dt+\Theta_t^2 dW_t,~\Theta_T=0,
\end{align}  where $\Theta_t^1, \Theta_t^2$ are to be determined. Then applying It\^{o}'s formula to $p_t$ in \eqref{ptt}, we obtain that
\begin{align}\label{ito}
dp_{t}&\hspace{-1mm}=\hspace{-1mm}\left[
  \begin{array}{cc}
   \hspace{-2mm} P_{t}\hspace{-2mm} & \hspace{-2mm}\bar{P}_{t}^{(1)}\hspace{-1mm} \\
   \hspace{-2mm} \bar{P}_{t}^{(2)}    \hspace{-2mm}            & \hspace{-2mm}\bar{P}_{t}^{(3)}\hspace{-1mm}\\
  \end{array}
\hspace{-2mm}\right]\hspace{-2mm}\Bigg\{\hspace{-2mm}\left[
  \begin{array}{cc}
   \hspace{-2mm} A\hspace{-2mm} & \hspace{-2mm}\bar{A} \\
   \hspace{-2mm} 0    \hspace{-2mm}            & \hspace{-2mm}A\hspace{-1mm}+\hspace{-1mm}\bar{A}\\
  \end{array}
\hspace{-2mm}\right]\hspace{-1mm}\left[
  \begin{array}{cc}
   \hspace{-2mm} x_{t} \hspace{-2mm}\\
    \hspace{-2mm}Ex_{t}\hspace{-2mm}
  \end{array}
\hspace{-2mm}\right]\hspace{-1mm}+\hspace{-1mm}\left[
  \begin{array}{cc}
   \hspace{-2mm}  B\hspace{-2mm}\\
    \hspace{-1mm}0\hspace{-2mm}
  \end{array}
\hspace{-2mm}\right]u_{t}
\hspace{-1mm}+\hspace{-1mm}\left[
  \begin{array}{cc}
 \hspace{-2mm}\bar{B} \\
  \hspace{-2mm}B\hspace{-1mm}+\hspace{-1mm}\bar{B}\\
  \end{array}
\hspace{-2mm}\right]Eu_{t}\Bigg\}dt\notag\\
&+\left[
  \begin{array}{cc}
   \hspace{-2mm} P_{t}\hspace{-2mm} & \hspace{-2mm}\bar{P}_{t}^{(1)} \\
   \hspace{-2mm} \bar{P}_{t}^{(2)}    \hspace{-2mm}            & \hspace{-2mm}\bar{P}_{t}^{(3)}\\
  \end{array}
\hspace{-2mm}\right]\Bigg\{\left[
  \begin{array}{cc}
   \hspace{-2mm} C\hspace{-2mm} & \hspace{-2mm}\bar{C} \\
   \hspace{-2mm} 0    \hspace{-2mm}& \hspace{-2mm}0\\
  \end{array}
\hspace{-2mm}\right]\hspace{-1mm}\left[
  \begin{array}{cc}
   \hspace{-2mm} x_{t} \hspace{-2mm}\\
    \hspace{-2mm}Ex_{t}\hspace{-2mm}
  \end{array}
\hspace{-2mm}\right]\hspace{-1mm}+\hspace{-1mm}\left[
  \begin{array}{cc}
   \hspace{-2mm}  D\hspace{-2mm}\\
    \hspace{-2mm}0\hspace{-2mm}
  \end{array}
\hspace{-2mm}\right]u_{t}\hspace{-1mm}+\hspace{-1mm}\left[
  \begin{array}{cc}
 \hspace{-2mm}\bar{D} \\
  \hspace{-2mm}0\\
  \end{array}
\hspace{-2mm}\right]Eu_{t}\Bigg\}dW_{t}\notag\\
&+\left[
  \begin{array}{cc}
   \hspace{-1mm} \dot{P}_{t}\hspace{-1mm} & \hspace{-1mm}\dot{\bar{P}}_{t}^{(1)} \\
   \hspace{-1mm} \dot{\bar{P}}_{t}^{(2)}    \hspace{-1mm}            & \hspace{-1mm}\dot{\bar{P}}_{t}^{(3)}\\
  \end{array}
\hspace{-1mm}\right]\left[
  \begin{array}{cc}
   \hspace{-1mm}  x_{t}\hspace{-1mm}\\
    \hspace{-1mm}Ex_{t}\hspace{-1mm}
  \end{array}
\hspace{-1mm}\right]dt+ \Theta_t^1dt+\Theta_t^2dW_t.
\end{align}

Comparing the $dW_{t}$ term in \eqref{ito} with that in \eqref{bsde}, we have that
\begin{align}\label{qt}
q_{t}\hspace{-1mm} =\hspace{-1mm}\left[
  \begin{array}{cc}
   \hspace{-2mm} P_{t}\hspace{-2mm} & \hspace{-2mm}\bar{P}_{t}^{(1)} \hspace{-2mm}\\
   \hspace{-2mm} \bar{P}_{t}^{(2)}    \hspace{-2mm}            & \hspace{-2mm}\bar{P}_{t}^{(3)}\hspace{-2mm}\\
  \end{array}
\hspace{-1mm}\right]\hspace{-1mm}\left\{\hspace{-1mm}\left[
  \begin{array}{cc}
   \hspace{-2mm} C\hspace{-2mm} & \hspace{-2mm}\bar{C} \hspace{-2mm}\\
   \hspace{-2mm} 0    \hspace{-2mm}            & \hspace{-2mm}0\hspace{-2mm}\\
  \end{array}
\hspace{-2mm}\right]\hspace{-1mm}\left[
  \begin{array}{cc}
   \hspace{-2mm}  x_{t} \hspace{-2mm}\\
    \hspace{-2mm} Ex_{t}\hspace{-2mm}
  \end{array}
\hspace{-2mm}\right]\hspace{-1mm}+\hspace{-1mm}\left[
  \begin{array}{cc}
   \hspace{-2mm}  D\hspace{-2mm}\\
    \hspace{-2mm}0\hspace{-2mm}
  \end{array}
\hspace{-2mm}\right] u_{t}\hspace{-1mm}+\hspace{-1mm}\left[
  \begin{array}{cc}
 \hspace{-2mm}\bar{D} \\
  \hspace{-2mm}0\\
  \end{array}
\hspace{-2mm}\right]Eu_{t}\hspace{-1mm}\right\}+\Theta_t^2.
\end{align}

Plugging \eqref{ptt} and \eqref{qt} into \eqref{uu}, and noting that $x_{t}$ is $\mathcal{F}_{t}$-adapted, we have
\begin{align}\label{4u}
0&=Ru_{t}+\bar{R}Eu_{t}+B'P_{t}x_{t}\hspace{-1mm}+\hspace{-1mm}
B'\bar{P}_{t}^{(1)}Ex_{t}\hspace{-1mm}\notag\\
&+\hspace{-1mm}D'P_{t}Cx_{t}+D'P_{t}\bar{C}Ex_{t}+D'P_{t}Du_{t}+D'P_{t}\bar{D}Eu_{t}\notag\\
&+[\bar{B}'P_{t}+(B+\bar{B})'(\bar{P}_{t}^{(2)}+\bar{P}_{t}^{(3)})+\bar{B}\bar{P}_{t}^{(1)}]Ex_{t}\notag\\
&+(\bar{D}'P_{t}C+\bar{D}'P_{t}\bar{C})Ex_{t}
+(\bar{D}'P_{t}D+\bar{D}'P_{t}\bar{D})Eu_{t}\notag\\
&+\mathcal{B}'\Theta_t+\mathcal{D}'\Theta_t^2+\bar{\mathcal{B}}'E \Theta_t
+\bar{\mathcal{D}}'E \Theta_t^2,
\end{align}
where $\mathcal{B}, \mathcal{D}, \bar{\mathcal{B}}, \bar{\mathcal{D}}$ are given in \eqref{label}.

By letting $\bar{P}_{t}=\bar{P}_{t}^{(1)}+\bar{P}_{t}^{(2)}+\bar{P}_{t}^{(3)}$, \eqref{4u} can also be presented as
\begin{align}\label{u}
0&=\Upsilon_{t}^{(1)}u_{t}\hspace{-1mm}+\hspace{-1mm}[\Upsilon_{t}^{(2)}\hspace{-1mm}-\hspace{-1mm}\Upsilon_{t}^{(1)}]Eu_{t}
+M_{t}^{(1)}x_{t}+[M_{t}^{(2)}\hspace{-1mm}-\hspace{-1mm}M_{t}^{(1)}]Ex_{t}\notag\\
&+\mathcal{B}'\Theta_t+\mathcal{D}'\Theta_t^2+\bar{\mathcal{B}}'E \Theta_t+\bar{\mathcal{D}}'E \Theta_t^2.
\end{align}

Taking expectation on both sides of \eqref{u}, there holds that
\begin{align}\label{ey}
  0 & =\Upsilon_{t}^{(2)}Eu_t\hspace{-1mm}+\hspace{-1mm}M_{t}^{(2)}Ex_t\hspace{-1mm}+\hspace{-1mm}(\mathcal{B}\hspace{-1mm}+\hspace{-1mm}
  \bar{\mathcal{B}})'E\Theta_t\hspace{-1mm}+\hspace{-1mm}(\mathcal{D}\hspace{-1mm}+\hspace{-1mm}\bar{\mathcal{D}})'E\Theta_t^2.
\end{align}

If \emph{Problem 1} is solvable, and \eqref{lkj}-\eqref{barp} is regular, i.e., $\Upsilon_t^{(2)}[\Upsilon_t^{(2)}]^{\dag}M_t^{(2)}=M_t^{(2)}$, then from \eqref{ey} we have
\begin{align}\label{thtt}
 \{I-\Upsilon_t^{(2)}[\Upsilon_t^{(2)}]^{\dag}\} [(\mathcal{B}+\bar{\mathcal{B}})'E\Theta_t+(\mathcal{D}+\bar{\mathcal{D}})'E\Theta_t^2]=0.
\end{align}

 $Eu_t$ can be calculated from \eqref{ey} as
\begin{align}\label{uuu2}
  Eu_t&=-[\Upsilon_t^{(2)}]^{\dag}M_t^{(2)}Ex_t+\bar{\mathbf{L}}-[\Upsilon_t^{(2)}]^{\dag} \{(\mathcal{B}+\bar{\mathcal{B}})'E\Theta_t\notag\\
  &~~+(\mathcal{D}+\bar{\mathcal{D}})'E\Theta_t^2\},\notag\\
\mathbf{\bar{L}}&=\bar{z}-[\Upsilon_t^{(2)}]^{\dag}\Upsilon_t^{(2)}\bar{z}.
\end{align}
where $\bar{z}\in \mathcal{R}^m$ is arbitrary deterministic constant.

Also note \eqref{lkj}-\eqref{barp} is regular (i.e., $\Upsilon_t^{(1)}[\Upsilon_t^{(1)}]^{\dag}M_t^{(1)}=M_t^{(1)}$), then using \eqref{uuu2} we have
\begin{align}\label{uuu}
  u_t & \hspace{-1mm} =\mathcal{K}_t x_t+\bar{\mathcal{K}}_tEx_t+\mathbf{L}+\mathbf{\bar{L}}\notag\\
  &-[\Upsilon_t^{(1)}]^\dag [\mathcal{B}'\Theta_t+\mathcal{D}'\Theta_t^2+\bar{\mathcal{B}}'E \Theta_t+\bar{\mathcal{D}}'E \Theta_t^2]\notag\\
  &-[\Upsilon_t^{(2)}]^\dag[(\mathcal{B}+\bar{\mathcal{B}})'E\Theta_t+(\mathcal{D}+\bar{\mathcal{D}})'E\Theta_t^2],\notag\\
\mathbf{L} &=z-[\Upsilon_t^{(1)}]^{\dag}\Upsilon_t^{(1)}z,
\end{align}
where $z\in \mathcal{R}^m$ is arbitrary, and $\mathcal{K}_t, \bar{\mathcal{K}}_t$ satisfy
 \begin{align}
   \mathcal{K}_t & =-[\Upsilon_t^{(1)}]^{\dag}M_t^{(1)},\label{mark1}\\
    \bar{\mathcal{K}}_t &=-\Big\{[\Upsilon_t^{(2)}]^{\dag}M_t^{(2)}
  -[\Upsilon_t^{(1)}]^{\dag}M_t^{(1)}\Big\}.\label{marks}
 \end{align}
Meanwhile, we can obtain
\begin{align}\label{tht1}
  \{I-\Upsilon_t^{(1)}[\Upsilon_t^{(1)}]^{\dag}\}( \mathcal{B}'\Theta_t+\mathcal{D}'\Theta_t^2+\bar{\mathcal{B}}'E \Theta_t+\bar{\mathcal{D}}'E \Theta_t^2)=0.
\end{align}

%{\color{blue}Since \emph{Problem 1} has a unique solution, i.e., the optimal controller is unique, it is necessary that
%\be I_m-[\Upsilon_t^{(2)}]^{\dag}\Upsilon_t^{(2)}=0, ~~t\in[0,T].\label{eut}
%\ee
%
%Noting that $\Upsilon_t^{(2)}$ is symmetric, thus \eqref{eut} indicates that
%\begin{align}\label{II}
%  [\Upsilon_t^{(2)}]^{\dag}\Upsilon_t^{(2)}=\Upsilon_t^{(2)}[\Upsilon_t^{(2)}]^{\dag}=I_m.
%\end{align}
%By using the property of Moore-Penrose inverse, we know that $\Upsilon_t^{(2)}$ is nonsingular, thus $Eu_t$ is given as
%\begin{align}\label{euu}
%  Eu_t & =-[\Upsilon_t^{(2)}]^{-1}M_t^{(2)}Ex_t=(K_t+\bar{K}_t)Ex_t,
%\end{align}
%where $K_t,\bar{K}_t$ are given by \eqref{kt}, \eqref{barkt}.}
%
%By plugging \eqref{euu} into \eqref{u}, using the uniqueness of the optimal controller and following the steps of \eqref{ey}-\eqref{euu}, the nonsingular of $\Upsilon_t^{(1)}$ can be proved and the optimal controller $u_{t}$ in \eqref{opti} can be verified, i.e.,
%\begin{align}\label{euu2}
% u_t&=-[\Upsilon_{t}^{(1)}]^{-1}M_{t}^{(1)}x_t\notag\\
% &-\left\{[\Upsilon_{t}^{(2)}]^{-1}M_{t}^{(2)}-[\Upsilon_{t}^{(1)}]^{-1}M_{t}^{(1)}\right\}Eu_t\notag\\
% &=K_{t}x_{t}+\bar{K}_{t}Ex_{t},
%\end{align}
%where $K_t,\bar{K}_t$ are as in \eqref{kt}, \eqref{barkt}.

Furthermore, comparing the $dt$ term in \eqref{ito} with that in \eqref{bsde} and noting \eqref{ptt} and \eqref{qt}, we can obtain
\begin{align}\label{p1}
&~~\left[
  \begin{array}{cc}
   \hspace{-2mm} P_{t}\hspace{-2mm} & \hspace{-2mm}\bar{P}_{t}^{(1)} \\
   \hspace{-2mm} \bar{P}_{t}^{(2)}\hspace{-2mm}& \hspace{-2mm}\bar{P}_{t}^{(3)}\\
  \end{array}
\hspace{-2mm}\right]\hspace{-1mm}\Bigg\{\hspace{-1mm}\left[
  \begin{array}{cc}
   \hspace{-2mm} A\hspace{-2mm} & \hspace{-2mm}\bar{A} \\
   \hspace{-2mm} 0\hspace{-2mm}& \hspace{-2mm}A\hspace{-1mm}+\hspace{-1mm}\bar{A}\\
  \end{array}
\hspace{-2mm}\right]\hspace{-1mm}\left[
  \begin{array}{cc}
   \hspace{-2mm}  x_{t} \hspace{-2mm}\\
    \hspace{-2mm}Ex_{t}\hspace{-2mm}
  \end{array}
\hspace{-2mm}\right]\hspace{-1mm}+\hspace{-1mm}\left[
  \begin{array}{cc}
   \hspace{-2mm}  B\hspace{-2mm}\\
    \hspace{-2mm}0\hspace{-2mm}
  \end{array}
\hspace{-2mm}\right] u_{t}\hspace{-1mm}+\hspace{-1mm}\left[
  \begin{array}{cc}
 \hspace{-2mm}\bar{B} \\
  \hspace{-2mm}B\hspace{-1mm}+\hspace{-1mm}\bar{B}\\
  \end{array}
\hspace{-2mm}\right] Eu_{t}\Bigg\}\hspace{-1mm}\notag\\
&+\hspace{-1mm}\left[
  \begin{array}{cc}
   \hspace{-2mm} \dot{P}_{t}\hspace{-2mm} & \hspace{-2mm}\dot{\bar{P}}_{t}^{(1)} \\
   \hspace{-2mm} \dot{\bar{P}}_{t}^{(2)}\hspace{-2mm}& \hspace{-2mm}\dot{\bar{P}}_{t}^{(3)}\\
  \end{array}
\hspace{-2mm}\right]\left[
  \begin{array}{cc}
   \hspace{-2mm}  x_{t}\hspace{-2mm}\\
    \hspace{-2mm}Ex_{t}\hspace{-2mm}
  \end{array}
\hspace{-2mm}\right]\hspace{-1mm}+\hspace{-1mm}\Theta_t^1\notag\\
&=-\Bigg\{\left[
  \begin{array}{cc}
   \hspace{-2mm} A\hspace{-2mm} & \hspace{-2mm}\bar{A} \\
   \hspace{-2mm} 0    \hspace{-2mm}            & \hspace{-2mm}A+\bar{A}\\
  \end{array}
\hspace{-2mm}\right]'\hspace{-1mm}\left[
  \begin{array}{cc}
   \hspace{-2mm} P_{t}\hspace{-2mm} & \hspace{-2mm}\bar{P}_{t}^{(1)}\hspace{-1mm} \\
   \hspace{-2mm} \bar{P}_{t}^{(2)}    \hspace{-2mm}            & \hspace{-2mm}\bar{P}_{t}^{(3)}\hspace{-1mm}\\
  \end{array}
\hspace{-2mm}\right]\hspace{-1mm}\left[
  \begin{array}{cc}
   \hspace{-2mm}  x_{t}\hspace{-2mm}\\
    \hspace{-2mm}Ex_{t}\hspace{-2mm}
  \end{array}
\hspace{-2mm}\right]\hspace{-1mm}+\hspace{-1mm}\left[
  \begin{array}{cc}
   \hspace{-2mm} C\hspace{-2mm} & \hspace{-2mm}\bar{C}\hspace{-1mm} \\
   \hspace{-2mm} 0    \hspace{-2mm}            & \hspace{-2mm}0\hspace{-1mm}\\
  \end{array}
\hspace{-2mm}\right]'\left[
  \begin{array}{cc}
   \hspace{-2mm} P_{t}\hspace{-2mm} & \hspace{-2mm}\bar{P}_{t}^{(1)} \hspace{-1mm}\\
   \hspace{-2mm} \bar{P}_{t}^{(2)}    \hspace{-2mm}            & \hspace{-2mm}\bar{P}_{t}^{(3)}\hspace{-1mm}\\
  \end{array}
\hspace{-2mm}\right]\notag\\
&\times\Big\{\left[
  \begin{array}{cc}
   \hspace{-2mm} C\hspace{-2mm} & \hspace{-2mm}\bar{C} \\
   \hspace{-2mm} 0\hspace{-2mm} & \hspace{-2mm}0\\
  \end{array}
\hspace{-2mm}\right]\hspace{-1mm}\left[
  \begin{array}{cc}
   \hspace{-2mm}  x_{t} \hspace{-2mm}\\
    \hspace{-2mm} Ex_{t}\hspace{-2mm}
  \end{array}
\hspace{-2mm}\right]\hspace{-1mm}+\hspace{-1mm}\left[
  \begin{array}{cc}
   \hspace{-2mm}  D\hspace{-2mm}\\
    \hspace{-2mm}0\hspace{-2mm}
  \end{array}
\hspace{-2mm}\right] u_{t}\hspace{-1mm}+\hspace{-1mm}\left[
  \begin{array}{cc}
 \hspace{-2mm}\bar{D} \\
  \hspace{-2mm}0\\
  \end{array}
\hspace{-2mm}\right]\hspace{-1mm} Eu_{t}\Big\}\hspace{-1mm}+\hspace{-1mm}\left[
  \begin{array}{cc}
   \hspace{-2mm}  I_{n}\hspace{-2mm}\\
    \hspace{-2mm}0\hspace{-2mm}
  \end{array}
\hspace{-2mm}\right](Qx_{t}\hspace{-1mm}+\hspace{-1mm}\bar{Q}Ex_{t})\Bigg\}.
\end{align}

Notice $u_{t}$, $Eu_{t}$ are given by \eqref{uuu}, \eqref{uuu2}, the following relationship can be obtained
\begin{align}
-\dot{\bar{P}}_{t}^{(1)}&=\bar{Q}+P_{t}\bar{A}+A'\bar{P}_{t}^{(1)}+\bar{P}_{t}^{(1)}(A+\bar{A})+C'P_{t}\bar{C}\notag\\
&+(P_{t}B+C'P_{t}D)\bar{\mathcal{K}}_{t}\hspace{-1mm}+\hspace{-1mm}C'P_{t}\bar{D}(\mathcal{K}_{t}+\bar{\mathcal{K}}_{t})\notag\\
&+[P_{t}\bar{B}+\bar{P}_{t}^{(1)}(B+\bar{B})](\mathcal{K}_{t}+\bar{\mathcal{K}}_{t}),\label{bart1}\\
-\dot{\bar{P}}_{t}^{(2)}&=\bar{P}_{t}^{(2)}A+\bar{A}'P_{t}+(A+\bar{A})'\bar{P}_{t}^{(2)}
+\bar{C}'P_{t}C\notag\\
&+[\bar{P}_{t}^{(2)}B+\bar{C}'P_{t}D]\mathcal{K}_{t},\label{bart2}\\
-\dot{\bar{P}}_{t}^{(3)}&=\bar{P}_{t}^{(2)}\bar{A}+\bar{P}_{t}^{(3)}(A+\bar{A})+\bar{A}'\bar{P}_{t}^{(1)}
+(A+\bar{A})'\bar{P}_{t}^{(3)}\notag\\
&+\bar{C}'P_{t}\bar{C}+[\bar{P}_{t}^{(2)}B\hspace{-1mm}+\hspace{-1mm}\bar{C}'P_{t}D]\bar{\mathcal{K}}_{t}\notag\\
&+[\bar{C}'P_{t}\bar{D}\hspace{-1mm}+\hspace{-1mm}\bar{P}_{t}^{(2)}\bar{B}\hspace{-1mm}+\hspace{-1mm}\bar{P}_{t}^{(3)}(B\hspace{-1mm}
+\hspace{-1mm}\bar{B})](\mathcal{K}_{t}\hspace{-1mm}+\hspace{-1mm}\bar{\mathcal{K}}_{t}),\label{bart3}\\
0&=(P_tB+C'P_tD)\mathbf{L}+[(P_t+\bar{P}_t^{(1)})(B+\bar{B})\notag\\
&+C'P_tD+C'P_t\bar{D}]\mathbf{\bar{L}},\label{LL}\\
0&=[\bar{P}_t^{(2)}B+\bar{C}'P_tD]\mathbf{L}+[(\bar{P}_t^{(2)}+\bar{P}_t^{(3)})(B+\bar{B})\notag\\
&+\bar{C}'P_tD+\bar{C}'P_t\bar{D}]\mathbf{\bar{L}},\label{barll}\\
\Theta_t^1&=\mathcal{P}_t\Big\{\mathcal{B}[\Upsilon_t^{(1)}]^\dag (\mathcal{B}'\Theta_t+\mathcal{D}'\Theta_t^2+\bar{\mathcal{B}}'E \Theta_t+\bar{\mathcal{D}}'E \Theta_t^2)\notag\\
&+(\mathcal{B}\hspace{-1mm}+\hspace{-1mm}\bar{\mathcal{B}})[\Upsilon_t^{(2)}]^{\dag} [(\mathcal{B}\hspace{-1mm}+\hspace{-1mm}\bar{\mathcal{B}})'E\Theta_t\hspace{-1mm}+\hspace{-1mm}(\mathcal{D}\hspace{-1mm}+\hspace{-1mm}\bar{\mathcal{D}})'E\Theta_t^2]\Big\},\label{combins}
\end{align}
with final condition $\bar{P}_{T}^{(1)}=\bar{P}_{T}$, $\bar{P}_{T}^{(2)}=\bar{P}_{T}^{(3)}=0$ and $\mathcal{P}_t=\left[
  \begin{array}{cc}
   \hspace{-2mm} P_{t}\hspace{-2mm} & \hspace{-2mm}\bar{P}_{t}^{(1)} \\
   \hspace{-2mm} \bar{P}_{t}^{(2)}    \hspace{-2mm}            & \hspace{-2mm}\bar{P}_{t}^{(3)}\\
  \end{array}
\hspace{-2mm}\right]$.

Taking summation on both sides of \eqref{bart1}-\eqref{bart3}, we know $\bar{P}_{t}=\bar{P}_{t}^{(1)}+\bar{P}_{t}^{(2)}+\bar{P}_{t}^{(3)}$ satisfies \eqref{barp}. Moreover, \eqref{p1} also indicates that $P_t$ satisfies \eqref{lkj}.

On the other hand, from \eqref{combins} we can obtain
\begin{align}\label{simpl}
  &\left[
  \begin{array}{cc}
   \hspace{-2mm}  \Theta_t^1\hspace{-2mm}\\
    \hspace{-2mm}E \Theta_t^1\hspace{-2mm}
  \end{array}
\hspace{-2mm}\right]  =\mathbf{B}_t\left[
  \begin{array}{cc}
   \hspace{-2mm}  \Theta_t\hspace{-2mm}\\
    \hspace{-2mm}E \Theta_t\hspace{-2mm}
  \end{array}
\hspace{-2mm}\right]+\mathbf{D}_t\left[
  \begin{array}{cc}
   \hspace{-2mm}  \Theta_t^2\hspace{-2mm}\\
    \hspace{-2mm}E \Theta_t^2\hspace{-2mm}
  \end{array}
\hspace{-2mm}\right],\\
&\mathbf{B}_t\hspace{-1mm}=\hspace{-2mm}\left[
  \begin{array}{cc}
   \hspace{-2mm} \mathcal{P}_t\mathcal{B}[\Upsilon_t^{(\hspace{-.3mm}1\hspace{-.3mm})}]^\dag \mathcal{B}'\hspace{-2mm} & \hspace{-2mm}\mathcal{P}_t\mathcal{B}[\Upsilon_t^{(\hspace{-.3mm}1\hspace{-.3mm})}]^\dag \bar{\mathcal{B}}' \\
   \hspace{-2mm}0    \hspace{-3mm}& \hspace{-3mm} \mathcal{P}_t\{\mathcal{B}[\Upsilon_t^{(1)}]^\dag (\mathcal{B}\hspace{-1.3mm}+\hspace{-1.3mm}\bar{\mathcal{B}})'\hspace{-1.2mm}+\hspace{-1.2mm}
   (\mathcal{B}\hspace{-1.3mm}+\hspace{-1.3mm}\bar{\mathcal{B}})[\Upsilon_t^{(2)}]^{\dag} (\mathcal{B}\hspace{-1.3mm}+\hspace{-1.3mm}\bar{\mathcal{B}})'\hspace{-.5mm}\}\\
  \end{array}
\hspace{-3mm}\right]\notag\\
&\mathbf{D}_t\hspace{-1.2mm}=\hspace{-2.2mm}\left[
  \begin{array}{cc}
   \hspace{-2mm}\mathcal{P}_t \mathcal{B}[\Upsilon_t^{(\hspace{-.3mm}1\hspace{-.3mm})}]^\dag \hspace{-.5mm}\mathcal{D}'\hspace{-3mm} & \hspace{-3mm}\mathcal{P}_t\mathcal{B}[\Upsilon_t^{(\hspace{-.3mm}1\hspace{-.3mm})}]^\dag \hspace{-.5mm}\bar{\mathcal{D}}'\hspace{-2mm} \\
   \hspace{-2mm}0    \hspace{-3mm}& \hspace{-3mm} \mathcal{P}_t\{\mathcal{B}[\Upsilon_t^{(1)}]^\dag \hspace{-.5mm} (\mathcal{D}\hspace{-1.3mm}+\hspace{-1.3mm}\bar{\mathcal{D}})'\hspace{-1.2mm}+\hspace{-1.2mm}
   (\mathcal{B}\hspace{-1.3mm}+\hspace{-1.3mm}\bar{\mathcal{B}})[\Upsilon_t^{(2)}]^{\dag} \hspace{-.5mm} (\mathcal{D}\hspace{-1.3mm}+\hspace{-1.3mm}\bar{\mathcal{D}})'\hspace{-.5mm}\}\hspace{-2mm}\\
  \end{array}
\hspace{-3mm}\right]\notag
\end{align}

From  \eqref{Theta}, we can obtain
\begin{align}\label{simpl2}
 d \left[
  \begin{array}{cc}
   \hspace{-2mm}  \Theta_t\hspace{-2mm}\\
    \hspace{-2mm}E \Theta_t\hspace{-2mm}
  \end{array}
\hspace{-2mm}\right] & =\Big\{\mathbf{B}_t\left[
  \begin{array}{cc}
   \hspace{-2mm}  \Theta_t\hspace{-2mm}\\
    \hspace{-2mm}E \Theta_t\hspace{-2mm}
  \end{array}
\hspace{-2mm}\right]\hspace{-1mm}+\hspace{-1mm}\mathbf{D}_t\left[
  \begin{array}{cc}
   \hspace{-2mm}  \Theta_t^2\hspace{-2mm}\\
    \hspace{-2mm}E \Theta_t^2\hspace{-2mm}
  \end{array}
\hspace{-2mm}\right]\Big\}dt\hspace{-1mm}+\hspace{-1mm}\left[
  \begin{array}{cc}
   \hspace{-2mm}  \Theta_t^2\hspace{-2mm}\\
    \hspace{-2mm}0\hspace{-2mm}
  \end{array}
\hspace{-2mm}\right]dW_t,
\end{align}
where $dE\Theta_t=E\Theta_t^1dt$ has been inserted and $\Theta_T=E\Theta_T=0$.

Since linear BSDE \eqref{simpl2} satisfies the Lipschitz conditions with linear growth, then it can be verified from \eqref{simpl2} that $(\Theta_t, \Theta_t^2)=(0,0)$ for $t\in[0,T]$ is the unique solution to BSDE \eqref{simpl2}, see \cite{yongj}. Therefore, the solution to FBSDE \eqref{ps1} and \eqref{bsde} is given by \eqref{pt}.

%On the other hand, if $\Theta_t=0$, we shall prove \eqref{lkj}-\eqref{barp} is \emph{regular}.
%
%In fact, if the regular condition \eqref{reg1} is not satisfied, similar to \eqref{Theta}-\eqref{tht1}, the controller can be given by \eqref{uuu}, while \eqref{thtt} and \eqref{tht1} can be respectively given as:
%\begin{align}
%  \{\Upsilon_t^{(2)}[\Upsilon_t^{(2)}]^{\dag}M_t^{(2)}-M_t^{(2)}\}Ex_t+[I-\Upsilon_t^{(2)}[\Upsilon_t^{(2)}]^{\dag}] \{(\mathcal{B}+\bar{\mathcal{B}})'E\Theta_t+(\mathcal{D}+\bar{\mathcal{D}})'E\Theta_t^2\}=0,\\
%  \{\Upsilon_t^{(1)}[\Upsilon_t^{(1)}]^{\dag}M_t^{(1)}-M_t^{(1)}\}x_t+[I-\Upsilon_t^{(1)}[\Upsilon_t^{(1)}]^{\dag}]\{ \mathcal{B}'\Theta_t+\mathcal{D}'\Theta_t^2+\bar{\mathcal{B}}'E \Theta_t+\bar{\mathcal{D}}'E \Theta_t^2\}=0
%\end{align}

In what follows, applying It\^{o}'s formula to $p_{t}'\left[
  \begin{array}{cc}
   \hspace{-1mm} x_{t} \hspace{-1mm}\\
   \hspace{-1mm} Ex_{t}\hspace{-1mm}\\
  \end{array}
\hspace{-1mm}\right]$, we have that
\begin{align}\label{p}
&dp_{t}'\left[
  \begin{array}{cc}
   \hspace{-1mm} x_{t} \hspace{-1mm}\\
   \hspace{-1mm} Ex_{t}\hspace{-1mm}\\
  \end{array}
\hspace{-1mm}\right]=d\left[
  \begin{array}{cc}
   \hspace{-1mm} x_{t} \hspace{-1mm}\\
   \hspace{-1mm} Ex_{t}\hspace{-1mm}\\
  \end{array}
\hspace{-1mm}\right]'\left[
  \begin{array}{cc}
   \hspace{-1mm} P_{t}\hspace{-1mm} & \hspace{-1mm}\bar{P}_{t}^{(1)} \\
   \hspace{-1mm} \bar{P}_{t}^{(2)}    \hspace{-1mm}            & \hspace{-1mm}\bar{P}_{t}^{(3)}\\
  \end{array}
\hspace{-1mm}\right]\left[
  \begin{array}{cc}
   \hspace{-1mm} x_{t} \hspace{-1mm}\\
   \hspace{-1mm} Ex_{t}\hspace{-1mm}\\
  \end{array}
\hspace{-1mm}\right]\\
&=d(x_{t}'P_{t}x_{t}+Ex_{t}'\bar{P}_{t}^{(2)}x_{t}+x_{t}'\bar{P}_{t}^{(1)}Ex_{t}+Ex_{t}'\bar{P}_{t}^{(3)}Ex_{t}).\notag
\end{align}

From system dynamics \eqref{ps1}, \eqref{ps20} and \eqref{pt} and noting $\bar{P}_{t}=\bar{P}_{t}^{(1)}+\bar{P}_{t}^{(2)}+\bar{P}_{t}^{(3)}$, taking the integral from $0$ to $T$ and then the expectation on both sides of \eqref{p}, we can obtain
\begin{align}\label{ppp}
&~~E\int_{0}^{T}d(x_{t}'P_{t}x_{t}+Ex_{t}'\bar{P}_{t}^{(2)}x_{t}\hspace{-1mm}
+\hspace{-1mm}x_{t}'\bar{P}_{t}^{(1)}Ex_{t}\hspace{-1mm}+\hspace{-1mm}Ex_{t}'\bar{P}_{t}^{(3)}Ex_{t})\notag\\
&=E(x_{T}'P_{T}x_{T})\hspace{-1mm}+\hspace{-1mm}Ex_{T}'\bar{P}_{T}Ex_{T}\hspace{-1mm}-\hspace{-1mm}
[E(x_{0}'P_{0}x_{0})\hspace{-1mm}+\hspace{-1mm}Ex_{0}'\bar{P}_{0}Ex_{0}]\notag\\
&=E\int_{0}^{T}d(x_{t}'P_{t}x_{t}+Ex_{t}'\bar{P}_{t}Ex_{t}).
\end{align}

Similar to the lines of \eqref{pp}, using \eqref{lkj}-\eqref{barp}, $J_{T}$ in \eqref{ps2} can be calculated as
\begin{align}\label{cost3}
&J_{T}=E(x_{0}'P_{0}x_{0})+Ex_{0}'\bar{P}_{0}Ex_{0}\\
&+E\int_{0}^{T}[u_{t}-Eu_{t}-\mathcal{K}_{t}(x_{t}-Ex_{t})]'\Upsilon_{t}^{(1)}\notag\\
&\times [u_{t}-Eu_{t}-\mathcal{K}_{t}(x_{t}-Ex_{t})]dt\notag\\
&+E\int_{0}^{T}[Eu_{t}-(\mathcal{K}_{t}\hspace{-1mm}+\hspace{-1mm}\bar{\mathcal{K}}_{t})Ex_{t}]'\Upsilon_{t}^{(2)}
[Eu_{t}\hspace{-1mm}-\hspace{-1mm}(\mathcal{K}_{t}\hspace{-1mm}+\hspace{-1mm}\bar{\mathcal{K}}_{t})Ex_{t}]dt,
\notag\end{align}
where $\mathcal{K}_t,\bar{\mathcal{K}}_t$ are given by \eqref{mark1}, \eqref{marks}.

Next, we will show $\Upsilon_t^{(1)}\geq 0$ and $\Upsilon_t^{(2)}\geq 0$.  %{\color{red}, the following proof is much complicated, please  prove the facts according my earlier suggestions, that is define

%\begin{align}\label{ps2}
 % J_{t}&=E\Big\{\int_{0}^{t}\Big[x_{t}'Q_{t}x_{t}+(Ex_{t})'\bar{Q}_{t}Ex_{t}+u_{t}'R_{t}u_{t}\notag\\
 % &+(Eu_{t})'\bar{R}_{t}Eu_{t}\Big]dt+x_{t}'P_tx_{t}+(Ex_{t})'\bar{P}_ttEx_{t}\Big\},~~~P_T=G, \bar{P}_T=\bar{G}.
%\end{align}
%It is easy to know that for $t_1\get t$
%\begin{align}\label{ps2}
 % J_{t_1}&=E\Big\{\int_{0}^{t}\Big[x_{t}'Q_{t}x_{t}+(Ex_{t})'\bar{Q}_{t}Ex_{t}+u_{t}'R_{t}u_{t}\notag\\
 % &+(Eu_{t})'\bar{R}_{t}Eu_{t}\Big]dt+x_{t}'P_tx_{t}+(Ex_{t})'\bar{P}_ttEx_{t}\Big\},~~~P_T=G, \bar{P}_T=\bar{G}.
%\end{align}

%}

Actually, we choose $\lambda_{t}^{(1)}$ be any fixed eigenvalue of $\Upsilon_{t}^{(1)}$ in \eqref{upsi1} and $\lambda_{t}^{(2)}$ be any fixed eigenvalue of $\Upsilon_{t}^{(2)}$ in \eqref{upsi2}. We will show $\mathcal{M}(\{t\in[0,T]|\lambda_{t}^{(1)}<0\})=0$ and $\mathcal{M}(\{t\in[0,T]|\lambda_{t}^{(2)}<0\})=0$, where $\mathcal{M}$ is the Lebesgue measure. Choose $v_{\lambda}^{(1)}$, $v_{\lambda}^{(2)}$ be unit eigenvector associated with $\lambda_{t}^{(1)}$ and $\lambda_{t}^{(2)}$ satisfying $[v_{\lambda}^{(1)}]'v_{\lambda}^{(1)}=1$ and $[v_{\lambda}^{(2)}]'v_{\lambda}^{(2)}=1$, respectively.

Define $\mathcal{I}^{(1)}_{l}$, $\mathcal{I}^{(2)}_{l}$ be the indicator function of the set $\{t\in[0,T]|\lambda_{t}^{(1)}<-\frac{1}{l}\}$, $\{t\in[0,T]|\lambda_{t}^{(2)}<-\frac{1}{l}\}, l=1,2,\cdots$, respectively. Thus we can obtain that
\begin{align}\label{lam1}
 |\lambda_{t}^{(i)}|^{-1}\mathcal{I}^{(i)}_{l}\Upsilon_{t}^{(i)}v_{\lambda}^{(i)}&=-\mathcal{I}^{(i)}_{l}v_{\lambda}^{(i)}, i=1,2.
\end{align}

Choose a fixed constant $\delta\in \mathcal{R}$, set
\begin{align}\label{set}
u_t=\mathcal{L}_t(x_t-Ex_t)+(\mathcal{L}_t+\bar{\mathcal{L}}_t)Ex_t,\end{align}
 with $\mathcal{L}_t, \bar{\mathcal{L}}_t$ designed as follows,
\begin{equation}\label{con1}
\mathcal{L}_{t}=\left\{ \begin{array}{ll}
0, ~~~~~~~~~~~~~~~~~~~~~~~~~~~~~~~~~~~\text{if} \lambda_{t}^{(1)}=0,\\
\frac{\delta \mathcal{I}^{(1)}_{l}}{|\lambda_{t}^{(1)}|^{1/2}}v_{\lambda}^{(1)}+\mathcal{K}_{t},~~~~~~~~~~~~~~~\text{if} \lambda_{t}^{(1)}\neq0,\\
\end{array} \right.
\end{equation}
\begin{equation}\label{con2}
\mathcal{L}_{t}+\bar{\mathcal{L}}_{t}=\left\{ \begin{array}{ll}
0, ~~~~~~~~~~~~~~~~~~~~~~~~~~~~~~\text{if} ~\lambda_{t}^{(2)}=0,\\
\frac{\delta \mathcal{I}^{(2)}_{l}}{|\lambda_{t}^{(2)}|^{1/2}}v_{\lambda}^{(2)}+\mathcal{K}_{t}+\bar{\mathcal{K}}_{t},~~~\text{if} ~\lambda_{t}^{(2)}\neq0,\\
\end{array} \right.
\end{equation}
and $\mathcal{K}_t, \bar{\mathcal{K}}_t$ are as in \eqref{mark1}, \eqref{marks}.

Using the controller designed in \eqref{set}-\eqref{con2}, and noting \eqref{mark1}-\eqref{marks} and \eqref{lam1}, thus $J_{T}$ can be calculated from \eqref{cost3} as
\begin{align}\label{con3}
J_{T}&=E(x_{0}'P_{0}x_{0})+Ex_{0}'\bar{P}_{0}Ex_{0}\\
&\hspace{-1mm}-\hspace{-1mm}\delta^{2}\hspace{-1mm}\int_{\mathcal{I}^{(1)}_{l}}E[(x_{t}\hspace{-1mm}-\hspace{-1mm}Ex_{t})'(x_{t}\hspace{-1mm}-\hspace{-1mm}Ex_{t})]dt\hspace{-1mm}
-\hspace{-1mm}\delta^{2}\hspace{-1mm}\int_{\mathcal{I}^{(2)}_{l}}Ex_{t}'Ex_{t}dt.\notag
\end{align}

Without loss of generality, we might as well assume that $x_{t}-Ex_{t}\neq 0$ in the set $\{t\in[0,T]|\lambda_{t}^{(1)}<-\frac{1}{l}\}$ and $Ex_{t}\neq 0$ in the set $\{t\in[0,T]|\lambda_{t}^{(2)}<-\frac{1}{l}\}$. Thus we have $\int_{\mathcal{I}^{(1)}_{l}}E[(x_{t}-Ex_{t})'(x_{t}-Ex_{t})]dt>0$ and $\int_{\mathcal{I}^{(2)}_{l}}Ex_{t}'Ex_{t}dt>0$.

By Assumption 1, obviously we have $J_{T}\geq 0$. While, for \eqref{con3}, if $\mathcal{M}\left[\mathcal{I}^{(1)}_{l}\right]>0$, by letting $\delta\rightarrow \infty$, we have $J_{T}\rightarrow -\infty$, which is a contradiction with $J_{T}\geq 0$. On the other hand, if $\mathcal{M}\left[\mathcal{I}^{(2)}_{l}\right]>0$, also by letting $\delta\rightarrow \infty$, we have $J_{T}\rightarrow -\infty$, which is a contradiction with $J_{T}\geq 0$. Hence, we have $\mathcal{M}\left[\mathcal{I}^{(1)}_{l}\right]=0$ and $\mathcal{M}\left[\mathcal{I}^{(2)}_{l}\right]=0$.

Notice that
\begin{align*}
\{t\in[0,T]|\lambda_{t}^{(i)}<0\}&=\bigcup_{l=1}^{\infty} \left\{ t\in[0,T]|\lambda_{t}^{(i)}<-\frac{1}{l}\right\}, i=1,2.
\end{align*}

Thus we can conclude that $\mathcal{M}(\{t\in[0,T]|\lambda_{t}^{(1)}<0\})=0$ and $\mathcal{M}(\{t\in[0,T]|\lambda_{t}^{(2)}<0\})=0$, i.e., $\Upsilon_{t}^{(1)}\geq 0$ and $\Upsilon_{t}^{(2)}\geq 0$.

Finally, we will show that $\Upsilon_t^{(1)}$ and $\Upsilon_t^{(2)}$ are invertible for $t\in[0,T]$.

Actually, if this is not the case, i.e., $\Upsilon_{t}^{(1)}, \Upsilon_{t}^{(2)}$ are assumed to be singular for $t\in[0,T]$. Notice $\Upsilon_{t}^{(1)}\geq 0$ and $\Upsilon_{t}^{(2)}\geq 0$, from \eqref{cost3} it can be obtained that
\begin{align}\label{cost4}
J_{T}&\geq E(x_{0}'P_{0}x_{0})+Ex_{0}'\bar{P}_{0}Ex_{0}.
\end{align}
On the other hand, from \eqref{uuu2}-\eqref{uuu} we can choose
\begin{align*}
  u_t^{(1)} & =\mathcal{K}_tx_t+\bar{\mathcal{K}}_tEx_t,u_t^{(2)} \hspace{-1mm} =\mathcal{K}_t x_t+\bar{\mathcal{K}}_tEx_t+\mathbf{\bar{L}},
\end{align*}
where $\bar{\mathbf{L}}$ is defined in \eqref{uuu2}, and $\bar{z}\neq 0$, i.e., $u_t^{(1)}\neq u_t^{(2)}$. In this case, $Eu_t^{(2)}=(\mathcal{K}_t+\bar{\mathcal{K}}_t)Ex_t+\mathbf{\bar{L}}$.

Noting the facts that
\begin{align}\label{fac}
  \Upsilon_t^{(2)}\bar{\mathbf{L}} & = \Upsilon_t^{(2)}\{\bar{z}-[\Upsilon_t^{(2)}]^{\dag}\Upsilon_t^{(2)}\bar{z}\}=0.
\end{align}
Thus, substituting $u_t^{(1)}$ and $u_t^{(2)}$ into \eqref{cost3}, respectively, we know the cost function $J_T$ can be calculated as the optimal cost function, i.e.,
\begin{align}\label{tw}
J_T(u_t^{(1)})=J_T(u_t^{(2)})=E(x_{0}'P_{0}x_{0})+Ex_{0}'\bar{P}_{0}Ex_{0}.
\end{align}
 Notice $u_t^{(1)}\neq u_t^{(2)}$, \eqref{tw} indicates that both $u_t^{(1)}$ and $u_t^{(2)}$ are the optimal controller, which contradicts with the unique solvability of \emph{Problem 1}. Therefore, the nonsingular of $\Upsilon_t^{(1)}$ and $\Upsilon_t^{(2)}$ can be proved. Combining with $\Upsilon_{t}^{(1)}\geq 0$ and $\Upsilon_{t}^{(2)}\geq 0$, we can conclude that if \emph{Problem 1} has a unique solution, then $\Upsilon_{t}^{(1)}> 0$ and $\Upsilon_{t}^{(2)}> 0$ for $t\in[0,T]$.

 Therefore, $[\Upsilon_{t}^{(1)}]^\dag, [\Upsilon_{t}^{(2)}]^{\dag}$ in \eqref{mark1}-\eqref{marks}, \eqref{bart1}-\eqref{bart3}, \eqref{barp}, \eqref{lkj} can be replaced by $[\Upsilon_{t}^{(1)}]^{-1}$, $[\Upsilon_{t}^{(2)}]^{-1}$, respectively. In other words, $P_t$, $\bar{P}_t$ given by \eqref{lkj}, \eqref{barp} are the coupled Riccati equation \eqref{Ric1}, \eqref{Ric2}; $\mathbf{L}=\mathbf{\bar{L}}=0$, the relationship \eqref{LL}-\eqref{barll} are obviously satisfied; $\mathcal{K}_t, \bar{\mathcal{K}}_t$ in \eqref{mark1}, \eqref{marks} can be replaced by $K_t, \bar{K}_t$ in \eqref{kt}, \eqref{barkt}.

 In conclusion, the optimal controller can be given as \eqref{opti}. Furthermore, from \eqref{cost3} we know that the optimal cost function can be given by \eqref{op}. The proof is complete.
\end{proof}

\section*{Appendix C: Proof of Lemma \ref{lem3}}
\begin{proof}
From Lemma \ref{lem:lemma3}, we know under Assumption \ref{ass:ass2}, if Riccati equation \eqref{2-Ric1}-\eqref{2-Ric2} is solvable and the regular condition \eqref{refg} holds, then cost function \eqref{ps2} with the final condition $P_{T}(T)=\bar{P}_{T}(T)=0$ can be minimized by the optimal controller \eqref{2-opti}, the optimal cost function is given by \eqref{16}:
\begin{align}\label{op1}
J_{T}^{*}=E[x_{0}'P_{0}(T)x_{0}]+Ex_{0}'\bar{P}_{0}(T)Ex_{0}.
\end{align}
Moreover, with Assumption \ref{ass:ass2}, obviously we have $J_T\geq 0$ for any controller $u_t$, then the optimal cost function $J_T^*\geq 0$.

If the initial state $x_0$ is chosen to be any random variable satisfying $Ex_0=0$, from \eqref{op1} we have $E[x_{0}'P_{0}(T)x_{0}]\geq 0$, then $P_0(T)\geq 0$ can be obtained. On the other hand, suppose $x_0$ is deterministic (i.e., $x_0=Ex_0$), there holds from \eqref{op1} that $x_0'[P_{0}(T)+\bar{P}_{0}(T)]x_0\geq 0$, thus we have $P_{0}(T)+\bar{P}_{0}(T)\geq 0$ for any $T$.

Since the coefficient matrices in \eqref{2-Ric1}-\eqref{2-k2} are time-invariant, we have that
\begin{equation}\label{tina}
  P_{t}(T)=P_{0}(T-t),~\bar{P}_{t}(T)=\bar{P}_{0}(T-t),~t\in[0,T].
\end{equation}

Thus, we can conclude that under Assumption \ref{ass:ass2}, the solution to \eqref{2-Ric1}-\eqref{2-Ric2} satisfies $P_t(T)\geq 0$, $P_{t}(T)+\bar{P}_{t}(T)\geq 0$ for $t\in[0,T]$.
\end{proof}

\section*{Appendix D: Proof of Lemma \ref{ess}}
\begin{proof}
  1) For system \eqref{ps1} and \eqref{ps20} with controller $u_t=0$, the exact detectable of system $(A,\bar{A},C,\bar{C},\mathcal{Q}^{1/2})$ in Assumption \ref{ass:ass4} is equivalent to the exact detectable of the following system $(\mathbb{A}, \mathbb{C}, \mathcal{Q}^{1/2})$:
  \begin{equation}\label{fb-0}
 \left\{ \begin{array}{ll}
 d\mathbb{X}_{t}=\mathbb{A}\mathbb{X}_{t}dt+\mathbb{C}\mathbb{X}_{t}dW_{t},~~\mathbb{X}_{0},\\
\mathcal{Y}_{t}=\mathcal{Q}^{1/2}\mathbb{X}_{t},
\end{array} \right.
\end{equation}
where $\mathbb{A}\hspace{-1mm}=\hspace{-1mm}\left[\hspace{-2mm}
  \begin{array}{cc}
   A& 0\\
   0        & A+\bar{A}   \\
  \end{array}
\hspace{-2mm}\right]$, $\mathbb{C}\hspace{-1mm}=\hspace{-1mm}\left[\hspace{-2mm}
  \begin{array}{cc}
   C& C+\bar{C}\\
   0        & 0\\
  \end{array}
\hspace{-2mm}\right]$ and $\mathcal{Q}$ is as in \eqref{mf}.

While system \eqref{ps1} associated with \eqref{ps20} with controller \eqref{control}  can be rewritten as
\begin{align}\label{fb}
d\mathbb{X}_{t}&=\tilde{\mathbb{A}}\mathbb{X}_{t}dt+\tilde{\mathbb{C}}\mathbb{X}_{t}dW_{t},
\end{align}
where $\mathbb{X}_t, \tilde{\mathbb{A}}$ and $\tilde{\mathbb{C}}$ are given below \eqref{mf01}.

From the symbols given in \eqref{symlk} and \eqref{mf01}, we know
$$\tilde{\mathbb{A}}=\mathbb{A}+\mathbb{B}\mathbb{K}, \tilde{\mathbb{C}}=\mathbb{C}+\mathbb{D}\mathbb{K},\tilde{\mathcal{Q}}=\mathcal{Q}+\mathbb{K}'\mathbb{R}\mathbb{K}$$
where $\mathbb{B}\hspace{-1mm}=\hspace{-1mm}\left[\hspace{-2mm}
  \begin{array}{cc}
   B\hspace{-1mm}& \hspace{-1mm}0\\
   0       \hspace{-1mm} & \hspace{-1mm}B\hspace{-1mm}+\hspace{-1mm}\bar{B}   \\
  \end{array}
\hspace{-2mm}\right], \mathbb{D}\hspace{-1mm}=\hspace{-1mm}\left[\hspace{-2mm}
  \begin{array}{cc}
   D\hspace{-1mm}&\hspace{-1mm} D\hspace{-1mm}+\hspace{-1mm}\bar{D}\\
   0        \hspace{-1mm}&\hspace{-1mm} 0\\
  \end{array}
\hspace{-2mm}\right], \mathbb{K}=\left[\hspace{-2mm}
  \begin{array}{cc}
   \mathcal{K}\hspace{-1mm}&\hspace{-1mm} 0\\
   0       \hspace{-1mm} &\hspace{-1mm} \mathcal{K}\hspace{-1mm}+\hspace{-1mm}\bar{\mathcal{K}}   \\
  \end{array}
\hspace{-2mm}\right]$ and
$\mathbb{R}=\left[\hspace{-2mm}
  \begin{array}{cc}
   R\hspace{-1mm}&\hspace{-1mm} 0\\
   0       \hspace{-1mm} &\hspace{-1mm} R\hspace{-1mm}+\hspace{-1mm}\bar{R}   \\
  \end{array}
\hspace{-2mm}\right]$.

Following from \emph{Theorem 4} and \emph{Proposition 1} in \cite{zhangw}, we know that if the exact detectability of system $(\mathbb{A}, \mathbb{C}, \mathcal{Q}^{1/2})$, (i.e., $(A,\bar{A},C,\bar{C},\mathcal{Q}^{1/2})$), then system \eqref{mf01} $(\tilde{\mathbb{A}}, \tilde{\mathbb{C}}, \tilde{\mathcal{Q}}^{1/2})$ is exact detectable for any feedback gain $\mathcal{K}, \bar{\mathcal{K}}$.

2) Applying It\^{o}'s formula to $E(x_{t}'Px_{t})+Ex_{t}'\bar{P}Ex_{t}$ and taking integral from $0$ to $T$, similar to \eqref{pp}-\eqref{cost2}, there holds that
\begin{align}\label{eqas}
&~~E(x_{T}'Px_{T})+(Ex_{T})'\bar{P}Ex_{T}
\hspace{-1mm}-\hspace{-1mm}[E(x_{0}'Px_{0})\hspace{-1mm}+\hspace{-1mm}(Ex_{0})'\bar{P}Ex_{0}]\notag\\
&=E(\mathbb{X}_{T}'\mathbb{P}\mathbb{X}_{T})-E(\mathbb{X}_{0}'\mathbb{P}\mathbb{X}_{0})\notag\\
&=-\int_{0}^TE\Big\{x_{t}'Qx_{t}+Ex_{t}'\bar{Q}Ex_{t}+u_{t}'Ru_{t}+Eu_{t}'\bar{R}Eu_{t}\notag\\
&~~+[u_{t}\hspace{-1mm}-\hspace{-1mm}Eu_{t}\hspace{-1mm}-\hspace{-1mm}\mathcal{K}(x_{t}-Ex_{t})]'\Upsilon^{(1)}
[u_{t}-Eu_{t}-\mathcal{K}(x_{t}-Ex_{t})]\notag\\
&~~+[Eu_{t}\hspace{-1mm}-\hspace{-1mm}(\mathcal{K}+\bar{\mathcal{K}})Ex_{t}]'\Upsilon^{(2)}[Eu_{t}\hspace{-1mm}-\hspace{-1mm}(\mathcal{K}+\bar{\mathcal{K}})Ex_{t}]\Big\}dt\notag\\
&=-\int_{0}^TE[x_{t}'Qx_{t}+Ex_{t}'\bar{Q}Ex_{t}\hspace{-1mm}+\hspace{-1mm}u_{t}'Ru_{t}
\hspace{-1mm}+\hspace{-1mm}Eu_{t}'\bar{R}Eu_{t}]dt\notag\\
&=-\int_{0}^TE\{x_{t}'(Q+\mathcal{K}'R\mathcal{K})x_{t}+Ex_{t}'[\bar{Q}+\bar{\mathcal{K}}'R\mathcal{K}\notag\\
&~~+\mathcal{K}'R\bar{\mathcal{K}}+\bar{\mathcal{K}}'R\bar{\mathcal{K}}+(\mathcal{K}+\bar{\mathcal{K}})'\bar{R}(\mathcal{K}+\bar{\mathcal{K}})]Ex_{t}\}dt\notag\\
&=-\int_{0}^TE\{(x_{t}-Ex_{t})'(Q+\mathcal{K}'R\mathcal{K})(x_{t}-Ex_{t})\notag\\
&~~+Ex_{t}'[Q+\bar{Q}+(\mathcal{K}+\bar{\mathcal{K}})'(R+\bar{R})(\mathcal{K}+\bar{\mathcal{K}})]Ex_{t}\}dt\notag\\
&=-\int_{0}^TE(\mathbb{X}_{t}'\tilde{\mathcal{Q}}\mathbb{X}_{t})dt\leq 0,
\end{align}
where controller \eqref{control} has been inserted above.

Suppose $E(\mathbb{X}_{0}'\mathbb{P}\mathbb{X}_{0})=0$, it holds from \eqref{eqas} that
\begin{align}\label{lll}
&0\leq \int_{0}^{T}E(\mathbb{X}_{t}'\tilde{\mathcal{Q}}\mathbb{X}_{t})dt=-E(\mathbb{X}_{T}'\mathbb{P}\mathbb{X}_{T})\leq 0,\notag\\
&\Rightarrow \int_{0}^{T}E(\tilde{\mathcal{Y}}_{t}'\tilde{\mathcal{Y}}_{t})dt=\int_{0}^{T}E(\mathbb{X}_{t}'\tilde{\mathcal{Q}}\mathbb{X}_{t})dt=0.
\end{align}
i.e., $\tilde{\mathcal{Y}}_{t}=\tilde{\mathcal{Q}}^{1/2}\mathbb{X}_{t}=0$, $t\geq 0$, and $\mathbb{P}\geq 0$ has been used. Thus, $\mathbb{X}_{0}$ is an unobservable state of system $(\tilde{\mathbb{A}},\tilde{\mathbb{C}},\tilde{\mathcal{Q}}^{1/2})$.

Conversely, suppose $\mathbb{X}_{0}$ is an unobservable state of system $(\tilde{\mathbb{A}},\tilde{\mathbb{C}},\tilde{\mathcal{Q}}^{1/2})$, i.e., for any $t\geq 0$, we have $\tilde{\mathcal{Y}}_{t}=\tilde{\mathcal{Q}}^{1/2}\mathbb{X}_{t}\equiv 0$. From the exact detectability of $(\tilde{\mathbb{A}},\tilde{\mathbb{C}},\tilde{\mathcal{Q}}^{1/2})$, we have $\lim_{T\rightarrow +\infty}E(\mathbb{X}_{T}'\mathbb{P}\mathbb{X}_{T})=0$. Thus, it follows from \eqref{eqas} that
\begin{equation}\label{130}
  E(\mathbb{X}_{0}'\mathbb{P}\mathbb{X}_{0})\hspace{-1mm}=\hspace{-1mm}
  \int_{0}^{\infty}\hspace{-1mm}E(\mathbb{X}_{t}'\tilde{\mathcal{Q}}\mathbb{X}_{t})dt=\int_{0}^{\infty}\hspace{-1mm}E(\tilde{\mathcal{Y}}_{t}'\tilde{\mathcal{Y}}_{t})dt
 \hspace{-1mm}=\hspace{-1mm}0.
\end{equation}

Therefore, it has been shown that the initial state $\mathbb{X}_{0}$ is an unobservable state if and only if $\mathbb{X}_{0}$ satisfies $E(\mathbb{X}_{0}'\mathbb{P}\mathbb{X}_{0})=0$.
\end{proof}

\section*{Appendix E: Proof of Theorem \ref{thm:succeed}}

\begin{proof}
``Necessity:" Under Assumptions \ref{ass:ass2} and \ref{ass:ass4}, assume mean-field system \eqref{ps1} is mean square stabilizable, we will show the coupled ARE \eqref{are1}-\eqref{are2} admits a unique positive semi-definite solution.

Firstly, we shall show $P_t(T)$ and $P_t(T)+\bar{P}_t(T)$ given in \eqref{2-Ric1}-\eqref{2-Ric2} are both monotonically increasing with respect to $T$.

In fact, using \eqref{16} and \eqref{tina}, for any $T_{1}>T>t$, and all $x_{0}\neq 0$, we can obtain
\begin{align}\label{com1}
J_{T_{1}-t}^{*}&=E[x_{0}'P_{0}(T_{1}-t)x_{0}]+Ex_{0}'\bar{P}_{0}(T_{1}-t)Ex_{0}\notag\\
&=E[x_{0}'P_{t}(T_{1})x_{0}]+Ex_{0}'\bar{P}_{t}(T_{1})Ex_{0}\notag\\
\geq J_{T-t}^{*}&=E[x_{0}'P_{0}(T-t)x_{0}]+Ex_{0}'\bar{P}_{0}(T-t)Ex_{0}\notag\\
&=E[x_{0}'P_{t}(T)x_{0}]+Ex_{0}'\bar{P}_{t}(T)Ex_{0}.
\end{align}

Similarly, for any $0\leq t_{1}<t_{2}\leq T$, we conclude that
\begin{align}\label{com2}
J_{T-t_{1}}^{*}&=E[x_{0}'P_{0}(T-t_{1})x_{0}]+Ex_{0}'\bar{P}_{0}(T-t_{1})Ex_{0}\notag\\
&=E[x_{0}'P_{t_{1}}(T)x_{0}]+Ex_{0}'\bar{P}_{t_{1}}(T)Ex_{0}\notag\\
\geq J_{T-t_{2}}^{*}&=E[x_{0}'P_{0}(T-t_{2})x_{0}]+Ex_{0}'\bar{P}_{0}(T-t_{2})Ex_{0}\notag\\
&=E[x_{0}'P_{t_{2}}(T)x_{0}]+Ex_{0}'\bar{P}_{t_{2}}(T)Ex_{0}.
\end{align}

For any initial state variable $x_{0}\neq 0$ with $Ex_{0}=0$, \eqref{com1} and \eqref{com2} yield that
\begin{align*}
  E[x_{0}'P_{t}(T_{1})x_{0}]&\geq E[x_{0}'P_{t}(T)x_{0}],\notag\\
  E[x_{0}'P_{t_{1}}(T)x_{0}]&\geq E[x_{0}'P_{t_{2}}(T)x_{0}],
\end{align*}
which indicates that $P_{t}(T_{1})\geq P_{t}(T)$ and $P_{t_{1}}(T)\geq P_{t_{2}}(T)$.

For any initial state $x_{0}\neq 0$ with $x_{0}=Ex_{0}$, i.e., $x_{0}\in \mathcal{R}^{n}$ is arbitrary deterministic, \eqref{com1} together with \eqref{com2} implies that
\begin{align*}
  x_{0}'[P_{t}(T_{1})+\bar{P}_{t}(T_{1})]x_{0}&\geq x_{0}'[P_{t}(T)+\bar{P}_{t}(T)]x_{0},\\
  x_{0}'[P_{t_{1}}(T)+\bar{P}_{t_{1}}(T)]x_{0}&\geq x_{0}'[P_{t_{2}}(T)+\bar{P}_{t_{2}}(T)]x_{0}.
\end{align*}
Then we have $P_{t}(T_{1})+\bar{P}_{t}(T_{1})\geq P_{t}(T)+\bar{P}_{t}(T)$ and $P_{t_{1}}(T)+\bar{P}_{t_{1}}(T)\geq P_{t_{2}}(T)+\bar{P}_{t_{2}}(T)$.

Thus, $P_{t}(T)$ and $P_{t}(T)+\bar{P}_{t}(T)$ are both monotonically increasing with respect to $T$ and are monotonically decreasing with respect to $t$.

Next we will show $P_{t}(T)$ and $P_{t}(T)+\bar{P}_{t}(T)$ are uniformly bounded.

Since there exists $u_{t}\in\mathcal{U}[0,\infty)$ stabilizing system \eqref{ps1} in the mean square sense,
\begin{equation}\label{lf}
u_{t}=Lx_{t}+\bar{L}Ex_{t}
\end{equation}
with constant matrices $L$, $\bar{L}$ to be determined, the closed loop system \eqref{ps1} with controller \eqref{lf} satisfies
\begin{equation}\label{as}
\lim_{t\rightarrow+\infty}E(x_{t}'x_{t})=0.
\end{equation}
Then, by plugging linear feedback \eqref{lf} into system \eqref{ps1}, we can obtain
\begin{align}
dx_{t}&=\{(A+BL)x_{t}\hspace{-1mm}+\hspace{-1mm}[\bar{A}+B\bar{L}\hspace{-1mm}+\hspace{-1mm}\bar{B}(L+\bar{L})]Ex_{t}\}dt\label{h1}\\
&+\{(C+DL)x_{t}\hspace{-1mm}+\hspace{-1mm}[\bar{C}+D\bar{L}+\bar{D}(L+\bar{L})]Ex_{t}\}dW_{t},\notag\\
dEx_{t}&=[(A+\bar{A})+(B+\bar{B})(L+\bar{L})]Ex_{t}dt.\label{h2}
\end{align}

%Denote $X_{t}=\left[
%  \begin{array}{cc}
%   \hspace{-2mm} x_{t} \hspace{-2mm}\\
%   \hspace{-2mm} Ex_{t}\hspace{-2mm}\\
%  \end{array}
%\hspace{-2mm}\right]$, it holds from \eqref{h1} and \eqref{h2} that
%\begin{equation}\label{h3}
%  dX_{t}=\mathcal{H}X_{t}dt+\bar{\mathcal{H}}X_{t}dW_{t},
%\end{equation}
%where $\mathcal{H}=\left[
%  \begin{array}{cc}
%   \hspace{-2mm} A\hspace{-1mm}+\hspace{-1mm}BL\hspace{-2mm} & \hspace{-2mm}\bar{A}\hspace{-1mm}+\hspace{-1mm}B\bar{L}\hspace{-1mm}+\hspace{-1mm}\bar{B}(L\hspace{-1mm}+\hspace{-1mm}\bar{L})\hspace{-2mm} \\
%   \hspace{-2mm} 0    \hspace{-2mm}            & \hspace{-2mm}(A\hspace{-1mm}+\hspace{-1mm}\bar{A})+(B\hspace{-1mm}+\hspace{-1mm}\bar{B})(L\hspace{-1mm}+\hspace{-1mm}\bar{L})\hspace{-2mm}\\
%  \end{array}
%\hspace{-2mm}\right]$, $\bar{\mathcal{H}}=\left[
%  \begin{array}{cc}
%   \hspace{-2mm} C+DL\hspace{-2mm} & \hspace{-2mm}\bar{C}+\hspace{-1mm}D\bar{L}+\hspace{-1mm}\bar{D}(L+\hspace{-1mm}\bar{L})\hspace{-2mm} \\
%   \hspace{-2mm} 0    \hspace{-2mm}            & \hspace{-2mm}0\hspace{-2mm}\\
%  \end{array}
%\hspace{-2mm}\right]$.

Since $Ex_{t}'Ex_{t}+E(x_{t}-Ex_{t})'(x_{t}-Ex_{t})=E(x_{t}'x_{t})$, then it follows from \eqref{as} that
$\lim_{t\rightarrow +\infty}Ex_{t}'Ex_{t}=0$.

Similar to the proof of \emph{Lemma 4.1 in \cite{rami2}}, from \eqref{as} we can obtain
\begin{equation*}
E\int_{0}^{\infty}x_{t}'x_{t}dt<\infty,~~\text{and}~~\int_{0}^{\infty}Ex_{t}'Ex_{t}dt<\infty.
\end{equation*}

In other words, there exists constant $c$ such that
\begin{equation}\label{as3}
  E\int_{0}^{\infty}x_{t}'x_{t}dt<cEx_{0}'x_{0}.
\end{equation}

Using Assumption \ref{ass:ass2}, we know there exists $\lambda>0$ such that
$\left[
  \begin{array}{cc}
   \hspace{-2mm} Q\hspace{-2mm} & \hspace{-2mm}0 \hspace{-2mm}\\
   \hspace{-2mm} 0    \hspace{-2mm}            & \hspace{-2mm}Q\hspace{-1mm}+\hspace{-1mm}\bar{Q}\hspace{-2mm}\\
  \end{array}
\hspace{-2mm}\right]\leq \lambda I$ and $\left[
  \begin{array}{cc}
   \hspace{-2mm} L'RL\hspace{-2mm} & \hspace{-2mm}0 \hspace{-2mm}\\
   \hspace{-2mm} 0    \hspace{-2mm}            & \hspace{-2mm}(L\hspace{-1mm}+\hspace{-1mm}\bar{L})'(R+\hspace{-1mm}\bar{R})(L\hspace{-1mm}+\hspace{-1mm}\bar{L})\hspace{-2mm}\\
  \end{array}
\hspace{-2mm}\right]\leq \lambda I$. Thus, from \eqref{lf} and \eqref{as3} we have
\begin{align}\label{in}
 J&=E\int_{0}^{\infty}[x_{t}'Qx_{t}\hspace{-1mm}+\hspace{-1mm}(Ex_{t})'\bar{Q}Ex_{t}+u_{t}'Ru_{t}+(Eu_{t})'\bar{R}Eu_{t}]dt\notag\\
&=E\int_{0}^{\infty}\Big\{x_{t}'(Q+L'RL)x_{t}+Ex_{t}'\Big[\bar{Q}+L'R\bar{L}+\bar{L}'RL\notag\\
&+\bar{L}'R\bar{L}+(L+\bar{L})'\bar{R}(L+\bar{L})\Big]Ex_{t}\Big\}dt\notag\\
&=E\int_{0}^{\infty}\Bigg\{\left[
  \begin{array}{cc}
   \hspace{-2mm} x_{t}-Ex_{t} \hspace{-2mm}\\
   \hspace{-2mm} Ex_{t}\hspace{-2mm}\\
  \end{array}
\hspace{-2mm}\right]'\left[
  \begin{array}{cc}
   \hspace{-2mm} Q\hspace{-2mm} & \hspace{-2mm}0 \hspace{-2mm}\\
   \hspace{-2mm} 0    \hspace{-2mm}            & \hspace{-2mm}Q+\bar{Q}\hspace{-2mm}\\
  \end{array}
\hspace{-2mm}\right]\left[
  \begin{array}{cc}
   \hspace{-2mm} x_{t}-Ex_{t} \hspace{-2mm}\\
   \hspace{-2mm} Ex_{t}\hspace{-2mm}\\
  \end{array}
\hspace{-2mm}\right]\Bigg\}dt\notag\\
&+E\int_{0}^{\infty}\Bigg\{\left[
  \begin{array}{cc}
   \hspace{-2mm} x_{t}-Ex_{t} \hspace{-2mm}\\
   \hspace{-2mm} Ex_{t}\hspace{-2mm}\\
  \end{array}
\hspace{-2mm}\right]'\left[
  \begin{array}{cc}
   \hspace{-2mm} L'RL\hspace{-2mm} & \hspace{-2mm}0 \hspace{-2mm}\\
   \hspace{-2mm} 0    \hspace{-2mm}            & \hspace{-2mm}(L+\bar{L})'(R\hspace{-1mm}+\hspace{-1mm}\bar{R})(L+\bar{L})\hspace{-2mm}\\
  \end{array}
\hspace{-2mm}\right] \notag\\
&\times\left[
  \begin{array}{cc}
   \hspace{-2mm} x_{t}-Ex_{t} \hspace{-2mm}\\
   \hspace{-2mm} Ex_{t}\hspace{-2mm}\\
  \end{array}
\hspace{-2mm}\right]\Bigg\}dt\notag\\
&\leq 2\lambda E\int_{0}^{\infty}E[(x_{t}-Ex_{t})'(x_{t}-Ex_{t})+Ex_{t}'Ex_{t}]dt\notag\\
&=2\lambda E\int_{0}^{\infty} x_{t}'x_{t}dt\leq 2\lambda c E(x_{0}'x_{0}).
\end{align}

Recall \eqref{op1}, then \eqref{in} implies
\begin{align}\label{yil}
 0&\leq E[x_{0}'P_{0}(T)x_{0}]+Ex_{0}'\bar{P}_{0}(T)Ex_{0}\notag\\
 &=J_{T}^{*}\leq J\leq 2\lambda cE(x_{0}'x_{0}).
\end{align}

Now we choose the initial state $x_{0}$ to be any random vector with zero mean, i.e., $Ex_{0}=0$, equation \eqref{yil} indicates that
\begin{equation*}
  0\leq E[x_{0}'P_{0}(T)x_{0}]\leq 2\lambda cE(x_{0}'x_{0}),
\end{equation*}
i.e., $0\leq P_{0}(T)\leq 2\lambda c I$.

Similarly, with $x_{0}=Ex_{0}$, i.e., the initial state $x_{0}$ is chosen to be arbitrary deterministic, then \eqref{yil} can be reduced to
\begin{equation*}
 0\leq  x_{0}'[P_{0}(T)+\bar{P}_{0}(T)]x_{0}\leq 2\lambda cx_{0}'x_{0},
\end{equation*}
hence, there holds $0\leq P_{0}(T)+\bar{P}_{0}(T)\leq 2\lambda cI$.

The boundedness of $P_{0}(T)$ and $P_{0}(T)+\bar{P}_{0}(T)$ has been proven. Recall that $P_{t}(T)$ and $P_{t}(T)+\bar{P}_{t}(T)$ are both monotonically increasing with respect to $T$ and are monotonically decreasing with respect to $t$, thus there exists constant matrices $P$ and $\bar{P}$ satisfying
\begin{align*}
&\lim_{t\rightarrow-\infty}P_{t}(T)=\lim_{t\rightarrow-\infty}P_{0}(T-t)=\lim_{T\rightarrow+\infty}P_{0}(T)=P\geq 0,\\
&\lim_{t\rightarrow-\infty}\hspace{-2mm}\bar{P}_{t}(T)\hspace{-1mm}=\hspace{-1mm}\lim_{t\rightarrow-\infty}\bar{P}_{0}(T\hspace{-1mm}-\hspace{-1mm}t)
\hspace{-1mm}=\hspace{-1mm}\lim_{T\rightarrow+\infty}\hspace{-2mm}\bar{P}_{0}(T)\hspace{-1mm}=\hspace{-1mm}\bar{P}, P+\bar{P}\geq 0.
\end{align*}

Thirdly, we will show $\dot{P}_{t}(T)\rightarrow0$ and $\dot{\bar{P}}_{t}(T)\rightarrow0$.

Actually, from \eqref{2-Ric1} and \eqref{2-Ric2} and noting that $P_{t}(T)$ and $\bar{P}_{t}(T)$ are bounded, we know that $\dot{P}_{t}(T)$ and $\dot{\bar{P}}_{t}(T)$ are uniformly bounded. Thus the uniformly continuousness of $P_{t}(T)$ and $\bar{P}_{t}(T)$ with respect to $t$ can be obtained. Furthermore, by using \emph{Lemma 8.2} in \cite{kha}, we can conclude that
\begin{equation*}
  \lim_{t\rightarrow-\infty}\dot{P}_{t}(T)=0,~~~\lim_{t\rightarrow-\infty}\dot{\bar{P}}_{t}(T)=0.
\end{equation*}

Taking limitations of $t\rightarrow -\infty$ on both sides of \eqref{4-upsi1}-\eqref{4-mt2}, we can conclude $\Upsilon_{t}^{(1)}(T)$, $\Upsilon_{t}^{(2)}(T)$, $M_{t}^{(1)}(T)$ and $M_{t}^{(2)}(T)$ are convergent, i.e.,
\begin{align}
 \lim_{t\rightarrow-\infty}\hspace{-2mm}\Upsilon_{t}^{(i)}(T)&=\Upsilon^{(i)},\lim_{t\rightarrow-\infty}\hspace{-2mm}M_{t}^{(i)}(T)=M^{(i)}, i=1,2,\label{u1}
\end{align}
where $\Upsilon^{(1)},M^{(1)},\Upsilon^{(2)},M^{(2)}$ are respectively as in \eqref{up1}-\eqref{hh2}. Moreover, taking limitation on both sides of \eqref{2-Ric1} and \eqref{2-Ric2}, $P$, $\bar{P}$ satisfies the coupled ARE \eqref{are1}-\eqref{are2}. On the other hand, the regular condition \eqref{refg} leads to
$\Upsilon^{(i)}[\Upsilon^{(i)}]^{\dag}M^{(i)}=M^{(i)}, i=1,2.$

In conclusion, we have proved that the coupled ARE \eqref{are1}-\eqref{are2} admits positive semi-definite solution.

In what follows, we will show the stabilizing controller \eqref{control} minimizes \eqref{ps200}.

 In fact, similar to \eqref{pp} and \eqref{cost1}, applying It\^{o}'s formula to $x_{t}'Px_{t}+Ex_{t}'\bar{P}Ex_{t}$, taking integral from $0$ to $T$, then taking expectation, we have that
\begin{align}\label{cost5}
&E(x_{T}'Px_{T})+Ex_{T}'\bar{P}Ex_{T}-[E(x_{0}'Px_{0})+Ex_{0}'\bar{P}Ex_{0}]\notag\\
&=-E\int_{0}^{T}[x_{t}'Qx_{t}+Ex_{t}'\bar{Q}Ex_{t}+u_{t}'Ru_{t}+Eu_{t}'\bar{R}Eu_{t}]dt\notag\\
&+E\int_{0}^{T}[u_{t}-Eu_{t}-\mathcal{K}(x_{t}-Ex_{t})]'\Upsilon^{(1)}\notag\\
&\times [u_{t}-Eu_{t}-\mathcal{K}(x_{t}-Ex_{t})]dt\notag\\
&+\hspace{-1mm}E\int_{0}^{T}\hspace{-1mm}[Eu_{t}\hspace{-1mm}-\hspace{-1mm}(\mathcal{K}\hspace{-1mm}+\hspace{-1mm}\bar{\mathcal{K}})Ex_{t}]'
\Upsilon^{(2)}[Eu_{t}\hspace{-1mm}-\hspace{-1mm}(\mathcal{K}\hspace{-1mm}+\hspace{-1mm}\bar{\mathcal{K}})Ex_{t}]dt,
\end{align}
where $\Upsilon^{(1)}$, $M^{(1)}$, $\Upsilon^{(2)}$, $M^{(2)}$ are given in \eqref{up1}-\eqref{hh2}, and $\mathcal{K}$, $\bar{\mathcal{K}}$ satisfy \eqref{k1}-\eqref{k2}.

From $\lim_{t\rightarrow +\infty}Ex_{t}'x_{t}=0$, obviously we can obtain that $\lim_{T\rightarrow+\infty}[E(x_{T}'Px_{T})+Ex_{T}'\bar{P}Ex_{T}]=0$. Thus, letting $T\rightarrow +\infty$, the cost function \eqref{ps200} can be rewritten from \eqref{cost5} as follows,
\begin{align}\label{cost21}
J&=E(x_{0}'P_{0}x_{0})+Ex_{0}'\bar{P}_{0}Ex_{0}\\
&+E\int_{0}^{\infty}[u_{t}-Eu_{t}-\mathcal{K}(x_{t}-Ex_{t})]'\Upsilon^{(1)}\notag\\
&\times [u_{t}-Eu_{t}-\mathcal{K}(x_{t}-Ex_{t})]dt\notag\\
&\hspace{-1mm}+\hspace{-1mm}E\int_{0}^{\infty}\hspace{-1mm}[Eu_{t}\hspace{-1mm}-\hspace{-1mm}(\mathcal{K}
\hspace{-1mm}+\hspace{-1mm}\bar{\mathcal{K}})Ex_{t}]'\Upsilon^{(2)}
[Eu_{t}\hspace{-1mm}-\hspace{-1mm}(\mathcal{K}\hspace{-1mm}+\hspace{-1mm}\bar{\mathcal{K}})Ex_{t}]dt.\notag
\end{align}

Therefore, the optimal controller can be given from \eqref{cost21} as \eqref{control}, and the optimal cost function \eqref{cost} can also be verified.

Finally, the uniqueness of $P$, $\bar{P}$ is proved as below. Assume that the coupled ARE \eqref{are1}-\eqref{are2} has another solution $S$, $\bar{S}$ satisfying $S\geq 0$ and $S+\bar{S}\geq 0$, i.e.,
\begin{align}
0&=Q+SA+A'S+C'SC-[T^{(1)}]'[\Delta^{(1)}]^{\dag}T^{(1)},\label{s1}\\
0&=\bar{Q}+S\bar{A}+\bar{A}'S+(A+\bar{A})'\bar{S}+\bar{S}(A+\bar{A})\notag\\
&+C'S\bar{C}+\bar{C}'SC+\bar{C}'S\bar{C}+[T^{(1)}]'[\Delta^{(1)}]^{\dag}T^{(1)}\notag\\
&-[T^{(2)}]'[\Delta^{(2)}]^{\dag}T^{(2)},\label{s2}
\end{align}
where
\begin{align*}
\Delta^{(1)}&=R+D'SD,~~T^{(1)}=B'S+D'SC,\notag\\
\Delta^{(2)}&=R+\bar{R}+(D+\bar{D})'S(D+\bar{D}),\notag\\
T^{(2)}&=(B+\bar{B})'(S+\bar{S})+(D+\bar{D})'S(C+\bar{C}),
\end{align*}
and the regular condition holds
$$\Delta^{(i)}[\Delta^{(i)}]^{\dag}T^{(i)}=T^{(i)}, i=1,2.$$

It is noted that the optimal cost function is given by \eqref{cost}, i.e.,
\begin{align}\label{un}
J^{*}&\hspace{-1mm}=\hspace{-1mm}E(x_{0}'Px_{0})\hspace{-1mm}+\hspace{-1mm}Ex_{0}'\bar{P}Ex_{0}
\hspace{-1mm}=\hspace{-1mm}E(x_{0}'Sx_{0})\hspace{-1mm}+\hspace{-1mm}Ex_{0}'\bar{S}Ex_{0}.
\end{align}

Choosing any $x_{0}\neq 0$ with $Ex_{0}=0$, from \eqref{un} we can obtain
\begin{equation*}
  E[x_{0}'(P-S)x_{0}]=0.
\end{equation*}
In other words, $P=S$ can be verified.

On the other hand, choosing arbitrary deterministic initial state, i.e., $x_{0}=Ex_{0}$, then it holds from \eqref{un} that
\begin{equation*}
  x_{0}'(P+\bar{P}-S-\bar{S})x_{0}=0.
\end{equation*}
Thus, $P+\bar{P}=S+\bar{S}$ can be obtained.

Hence, we have verified that $S=P$ and $\bar{S}=\bar{P}$, i.e., the solution to the coupled ARE \eqref{are1}-\eqref{are2} is unique.

``Sufficiency:" Under Assumptions \ref{ass:ass2} and \ref{ass:ass4}, if $P$, $\bar{P}$ is the unique positive semi-definite solution to \eqref{are1}-\eqref{are2}, i.e., $P\geq 0$ and $P+\bar{P}\geq 0$, we will show that mean-field system \eqref{ps1} with specific controller \eqref{control} is mean square stabilizable.

For stability analysis, the Lyapunov function candidate $V(t,x_{t})$ is introduced as,
\begin{align}\label{lya}
V(t,x_{t})\triangleq E(x_{t}'Px_{t})+Ex_{t}'\bar{P}Ex_{t},
\end{align}
where $P$ and $\bar{P}$ satisfy \eqref{are1}-\eqref{are2}.

From \eqref{16} and \eqref{p}, we know that the Lyapunov function candidate \eqref{lya} is defined with the \emph{optimal cost function} and the \emph{solution to the FBSDE} obtained in Theorems \ref{thm:maximum} and \ref{thm:main}.

Since $P\geq 0$ and $P+\bar{P}\geq 0$, then \eqref{lya} indicates that
\begin{align}\label{lya0}
V(t,x_{t})&=E[(x_{t}-Ex_{t})'P(x_{t}-Ex_{t})+Ex_{t}'(P+\bar{P})Ex_{t}]\notag\\
&\geq 0.
\end{align}

Following from \eqref{eqas}, we can obtain
\begin{align}\label{lya1}
&\dot{V}(t,x_{t})=-E(\mathbb{X}_{t}'\tilde{\mathcal{Q}}\mathbb{X}_{t})\leq 0.
\end{align}

\eqref{lya1} implies that $V(t,x_{t})\leq V(0,x_{0})$, i.e., $V(t,x_{t})$ is nonincreasing, from \eqref{lya0} we know $V(t,x_{t})$ is bounded below, thus $\lim_{t\rightarrow +\infty}V(t,x_{t})$ exists.

%Integrating on both sides of \eqref{lya1} from $0$ to $T$ for any $T>0$, we have that
%\begin{align}\label{lya03}
%&\int_{0}^{T}E(\mathbb{X}_{t}'\tilde{\mathcal{Q}}\mathbb{X}_{t})dt=V(0,x_{0})-V(T,x_{T})\notag\\
%&=E(x_{0}'Px_{0})+(Ex_{0})'\bar{P}Ex_{0}
%\hspace{-1mm}-\hspace{-1mm}[E(x_{T}'Px_{T})\hspace{-1mm}+\hspace{-1mm}(Ex_{T})'\bar{P}Ex_{T}]\notag\\
%&=E(\mathbb{X}_{0}'\mathbb{P}\mathbb{X}_{0})-E(\mathbb{X}_{T}'\mathbb{P}\mathbb{X}_{T}),
%\end{align}
%while $\mathbb{P}=\left[\hspace{-1mm}
%  \begin{array}{cc}
%   P& 0\\
%   0        & P+\bar{P}      \\
%  \end{array}
%\hspace{-1mm}\right]$, $\mathbb{X}_{t}=\left[\hspace{-1mm}
%  \begin{array}{cc}
%   \hspace{-1mm} x_{t}-Ex_{t}\hspace{-1mm}\\
%   \hspace{-1mm} Ex_{t}     \hspace{-1mm}           \\
%  \end{array}
%\hspace{-1mm}\right]$, $t\in[0,T]$.

From Lemma \ref{ess}, we know that the stability of system \eqref{fb}, $(\tilde{\mathbb{A}},\tilde{\mathbb{C}})$ for simplicity, is equivalent to the stabilization of system \eqref{ps1} with controller \eqref{control}. Thus, two different cases are considered to show the stability of system \eqref{fb} in the mean square sense as below.

\emph{Case 1}: $\mathbb{P}>0$.

With $\mathbb{P}>0$, if $E(\mathbb{X}_{0}'\mathbb{P}\mathbb{X}_{0})=0$,  then using Lemma \ref{ess} we know $\mathbb{X}_{0}=0$ is the unique unobservable state of system $(\tilde{\mathbb{A}},\tilde{\mathbb{C}}, \tilde{\mathcal{Q}}^{1/2})$. Then we can conclude system $(\tilde{\mathbb{A}},\tilde{\mathbb{C}},\tilde{\mathcal{Q}}^{1/2})$ is exact observable.

Similar to the derivation of \eqref{pp}-\eqref{cost1}, for system \eqref{mf01}, $(\tilde{\mathbb{A}},\tilde{\mathbb{C}},\tilde{\mathcal{Q}}^{1/2})$ we have
\begin{align}\label{a}
  &\int_{0}^{T}E(\mathbb{X}_{t}'\tilde{\mathcal{Q}}\mathbb{X}_{t})dt \notag\\ &  =\int_{0}^{T}\{E[x(t)-Ex_t]'\mathbf{Q}[x(t)-Ex_t]+Ex_t'\bar{\mathbf{Q}}Ex_t\}dt\notag\\
&=  E[(x_0-Ex_0)'H_0(T)(x_0-Ex_0)]\notag\\
&+Ex_0'[H_0(T)+\bar{H}_0(T)]Ex_0,
\end{align}
where $H_t,\bar{H}_t$ satisfy the following differential equation:
\begin{align}
  &-\dot{H}_t(T) =\mathbf{Q}\hspace{-1mm}+\hspace{-1mm}
  \mathbf{A}'H_t(T)\hspace{-1mm}+\hspace{-1mm}H_t(T)\mathbf{A}\hspace{-1mm}+\hspace{-1mm}\mathbf{C}'H_t(T)\mathbf{C},\label{4-ly3} \\
 & -\dot{H}_t(T)+\dot{\bar{H}}_t(T) =\bar{\mathbf{Q}}+\bar{\mathbf{A}}'[H_t(T)+\bar{H}_t(T)]\notag\\
  &~~~~~~~~~~~~~~~~~~~+[H_t(T)+\bar{H}_t(T)]  \bar{\mathbf{A}}+\bar{\mathbf{C}}'H_t(T)\bar{\mathbf{C}},\label{4-ly4}
\end{align}
with final condition $H_T(T)=\bar{H}_T(T)=0. $

Since $\tilde{\mathcal{Q}}\geq 0$ in \eqref{a}, then similar to the proof of Lemma \ref{lem3}, we know that \eqref{4-ly3}-\eqref{4-ly4} admit a unique solution $H_t(T)\geq 0, H_t(T)+\bar{H}_t(T)\geq 0$ for $t\in[0,T]$.

We claim $H_0(T)>0$ and $H_0(T)+\bar{H}_0(T)>0$. If this is not the case, we know that there exists nonzero $y$ and $\bar{y}$ satisfying
\begin{align}
y&\neq 0,~E[y'H_{0}(T)y]=0,~Ey=0,\label{y1}\\
\bar{y}&\neq 0,~\bar{y}'[H_{0}(T)+\bar{H}_{0}(T)]\bar{y}=0,~\bar{y}=E\bar{y}.\label{y2}
\end{align}

Then we choose the initial state be $y$, \eqref{a} can be reduced to
\begin{align}\label{4-a}
  \int_{0}^{T}E(\mathbb{X}_{t}'\tilde{\mathcal{Q}}\mathbb{X}_{t})dt & =E[y'H_0(T)y]=0
\end{align}
which indicates that $\tilde{\mathcal{Q}}^{1/2}\mathbb{X}_{t}=0, a.s.$ for any $t\in[0,T]$, then from the exact observability of system $(\tilde{\mathbb{A}},\tilde{\mathbb{C}},\tilde{\mathcal{Q}}^{1/2})$, we can obtain the initial state $y=0$, which contradicts with $y\neq 0$ defined in \eqref{y1}.

On the other hand, if the initial state is chosen to be $\bar{y}$, from \eqref{a} we know
\begin{align}\label{4-a2}
\int_{0}^{T}E(\mathbb{X}_{t}'\tilde{\mathcal{Q}}\mathbb{X}_{t})dt & =\bar{y}'[H_0(T)+\bar{H}_0(T)]\bar{y}=0.
\end{align}

Similar to the discussion above, it follows from the exact observability of system \eqref{mf01} that $\bar{y}= 0$, this contradicts with $\bar{y}\neq 0$ in \eqref{y2}.

In conclusion, we have proved $H_0(T)>0$ and $H_0(T)+\bar{H}_0(T)>0$.

Via a time shift of $t$, combining \eqref{a}, we have
\begin{align}\label{a-5}
  &\int_{t}^{t+T}E(\mathbb{X}_{s}'\tilde{\mathcal{Q}}\mathbb{X}_{s})dt \notag\\ &\hspace{-1mm}=\hspace{-1mm}E[(x_t\hspace{-1mm}-\hspace{-1mm}Ex_t)'H_0(T)(x_t\hspace{-1mm}-\hspace{-1mm}Ex_t)]
  \hspace{-1mm}+\hspace{-1mm}Ex_t'[H_0(T)\hspace{-1mm}+\hspace{-1mm}\bar{H}_0(T)]Ex_t\notag\\
  &=V(t,x_t)-V(t+T,x_{t+T}),
\end{align}
where $H_t(T+t)=H_0(T), \bar{H}_t(T+t)=\bar{H}_0(T)$ has been used.

Taking limitation on both sides of \eqref{a-5} , using the convergence of $V(t,x_t)$, we have that
\begin{align}\label{a-6}
  \lim_{t\rightarrow+\infty}\hspace{-2mm}E[(x_t\hspace{-1mm}-\hspace{-1mm}Ex_t)'(x_t\hspace{-1mm}-\hspace{-1mm}Ex_t)]\hspace{-1mm}=\hspace{-1mm}0,
   \lim_{t\rightarrow+\infty}\hspace{-2mm}Ex_t'Ex_t\hspace{-1mm}=\hspace{-1mm}0,
\end{align}
In other words, $\lim_{t\rightarrow+\infty}E(x_t'x_t)=0$, i.e., system $(\tilde{\mathbb{A}},\tilde{\mathbb{C}})$ is stable, thus system \eqref{ps1} is mean square stabilizable with controller \eqref{control}.

\emph{Case 2}: $\mathbb{P}\geq 0$.

Firstly, combining \eqref{4-ly1} and \eqref{4-ly2} we know that $\mathbb{P}$ obeys the following Lyapunov equation:
\begin{equation}\label{ly4}
  0=\tilde{\mathcal{Q}}+\tilde{\mathbb{A}}'\mathbb{P}+\mathbb{P}\tilde{\mathbb{A}}
  +\{\tilde{\mathbb{C}}^{(1)}\}'\mathbb{P}\tilde{\mathbb{C}}^{(1)}+\{\tilde{\mathbb{C}}^{(2)}\}'\mathbb{P}\tilde{\mathbb{C}}^{(2)},
\end{equation}
where $\tilde{\mathbb{C}}^{(1)}\hspace{-1mm}=\hspace{-1mm}\left[\hspace{-2mm}
  \begin{array}{cc}
   \mathbf{C}\hspace{-2mm}&\hspace{-2mm} 0\\
   0       \hspace{-2mm} &\hspace{-2mm} 0\\
  \end{array}
\hspace{-2mm}\right]$, $\tilde{\mathbb{C}}^{(2)}\hspace{-1mm}=\hspace{-1mm}\left[\hspace{-2mm}
  \begin{array}{cc}
   0\hspace{-2mm}& \hspace{-2mm}\bar{\mathbf{C}}\\
   0        \hspace{-2mm}&\hspace{-2mm} 0\\
  \end{array}
\hspace{-2mm}\right]$ and $\tilde{\mathbb{C}}^{(1)}+\tilde{\mathbb{C}}^{(2)}=\tilde{\mathbb{C}}$.

The positive semi-definiteness of $\mathbb{P}$ indicates that there exists orthogonal matrix $U$ with $U'=U^{-1}$ such that
\begin{align}\label{upu}
U'\mathbb{P}U=\left[
  \begin{array}{cc}
   0& 0\\
   0        & \mathbb{P}_{2}     \\
  \end{array}
\right], \mathbb{P}_{2}>0.
\end{align}
From \eqref{ly4}, it can be obtained
\begin{align}\label{ly5}
  0&=U'\tilde{\mathcal{Q}}U+U'\tilde{\mathbb{A}}'U\cdot U'\mathbb{P}U+U'\mathbb{P}U\cdot U'\tilde{\mathbb{A}}U\notag\\
  &+U'\{\tilde{\mathbb{C}}^{(1)}\}'U\cdot U'\mathbb{P}U\cdot U'\tilde{\mathbb{C}}^{(1)}U\notag\\
  &+U'\{\tilde{\mathbb{C}}^{(2)}\}'U\cdot U'\mathbb{P}U\cdot U'\tilde{\mathbb{C}}^{(2)}U.
\end{align}

Without loss of generality, assume $U'\tilde{\mathbb{A}}U=\left[\hspace{-1mm}
  \begin{array}{cc}
   \tilde{\mathbb{A}}_{11}\hspace{-1mm}&\hspace{-1mm} \tilde{\mathbb{A}}_{12}\\
   \tilde{\mathbb{A}}_{21}\hspace{-1mm}&\hspace{-1mm} \tilde{\mathbb{A}}_{22}    \\
  \end{array}
\hspace{-1mm}\right]$, $U'\tilde{\mathbb{C}}^{(1)}U\hspace{-1mm}=\hspace{-1mm}\left[\hspace{-1mm}
  \begin{array}{cc}
   \tilde{\mathbb{C}}_{11}^{(1)}\hspace{-2mm}&\hspace{-2mm} \tilde{\mathbb{C}}_{12}^{(1)}\\
   \tilde{\mathbb{C}}_{21}^{(1)}\hspace{-2mm}&\hspace{-2mm} \tilde{\mathbb{C}}_{22}^{(1)}    \\
  \end{array}
\hspace{-1mm}\right]$, $U'\tilde{\mathbb{C}}^{(2)}U\hspace{-1mm}=\hspace{-1mm}\left[\hspace{-1mm}
  \begin{array}{cc}
   \tilde{\mathbb{C}}_{11}^{(2)}\hspace{-2mm}&\hspace{-2mm} \tilde{\mathbb{C}}_{12}^{(2)}\\
   \tilde{\mathbb{C}}_{21}^{(2)}\hspace{-2mm}&\hspace{-2mm} \tilde{\mathbb{C}}_{22}^{(2)}    \\
  \end{array}
\hspace{-1mm}\right]$, $U'\tilde{\mathcal{Q}}U=\left[\hspace{-1mm}
  \begin{array}{cc}
   \tilde{\mathcal{Q}}_{1}\hspace{-1mm}&\hspace{-1mm} \tilde{\mathcal{Q}}_{12}\\
   \tilde{\mathcal{Q}}_{12}'\hspace{-1mm}&\hspace{-1mm} \tilde{\mathcal{Q}}_{2}    \\
  \end{array}
\hspace{-1mm}\right]$, it holds
\begin{align*}&U'\tilde{\mathbb{A}}'U\cdot U'\mathbb{P}U+ U'\mathbb{P}U\cdot U'\tilde{\mathbb{A}}U\hspace{-1mm}=\hspace{-1mm}\left[\hspace{-2mm}
  \begin{array}{cc}
   0& \tilde{\mathbb{A}}_{21}'\mathbb{P}_{2}\\
   \mathbb{P}_{2}\tilde{\mathbb{A}}_{21}\hspace{-2mm}&\hspace{-2mm} \tilde{\mathbb{A}}_{22}'\mathbb{P}_{2}+\mathbb{P}_{2}\tilde{\mathbb{A}}_{22}    \\
  \end{array}
\hspace{-2mm}\right],\\
&U'[\tilde{\mathbb{C}}^{(1)}]'U\hspace{-1mm}\cdot\hspace{-1mm} U'\mathbb{P}U\hspace{-1mm}\cdot \hspace{-1mm} U'\tilde{\mathbb{C}}^{(1)}U\hspace{-1.2mm}=\hspace{-1.2mm}\left[\hspace{-2mm}
  \begin{array}{cc}
   \{\tilde{\mathbb{C}}_{21}^{(1)}\}'\mathbb{P}_{2}\tilde{\mathbb{C}}_{21}^{(1)}\hspace{-2mm}& \hspace{-2mm} \{\tilde{\mathbb{C}}_{21}^{(1)}\}'\mathbb{P}_{2}\tilde{\mathbb{C}}_{22}^{(1)}\\
    \{\tilde{\mathbb{C}}^{(1)}_{22}\}'\mathbb{P}_{2}\tilde{\mathbb{C}}_{21}^{(1)}\hspace{-2mm}& \hspace{-2mm} \{\tilde{\mathbb{C}}_{22}^{(1)}\}'\mathbb{P}_{2}\tilde{\mathbb{C}}_{22}^{(1)}\\
  \end{array}
\hspace{-2.5mm}\right]\\
&U'[\tilde{\mathbb{C}}^{(2)}]'U\hspace{-1mm}\cdot \hspace{-1mm}U'\mathbb{P}U\hspace{-1mm}\cdot\hspace{-1mm} U'\tilde{\mathbb{C}}^{(2)}U\hspace{-1.2mm}=\hspace{-1.2mm}\left[\hspace{-2mm}
  \begin{array}{cc}
   \{\tilde{\mathbb{C}}_{21}^{(2)}\}'\mathbb{P}_{2}\tilde{\mathbb{C}}_{21}^{(2)}\hspace{-2mm}& \hspace{-2mm} \{\tilde{\mathbb{C}}_{21}^{(2)}\}'\mathbb{P}_{2}\tilde{\mathbb{C}}_{22}^{(2)}\\
    \{\tilde{\mathbb{C}}^{(2)}_{22}\}'\mathbb{P}_{2}\tilde{\mathbb{C}}_{21}^{(2)}\hspace{-2mm}& \hspace{-2mm} \{\tilde{\mathbb{C}}_{22}^{(2)}\}'\mathbb{P}_{2}\tilde{\mathbb{C}}_{22}^{(2)}\\
  \end{array}
\hspace{-2.5mm}\right].\end{align*}
Hence, comparing each block element on both sides of \eqref{ly5} and noticing $\mathbb{P}_{2}>0$, we can obtain \begin{align}
&\tilde{\mathbb{A}}_{21}'\mathbb{P}_{2}\hspace{-1mm}+\hspace{-1mm}\tilde{\mathcal{Q}}_{12}\hspace{-1mm}=\hspace{-1mm}0, \tilde{\mathbb{C}}_{21}^{(1)}\hspace{-1mm}=\hspace{-1mm}\tilde{\mathbb{C}}_{21}^{(2)}\hspace{-1mm}=\hspace{-1mm}0, \tilde{\mathcal{Q}}_{1}=0, U'\tilde{\mathcal{Q}}U\hspace{-1mm}=\hspace{-1mm}\left[\hspace{-2mm}
  \begin{array}{cc}
   0\hspace{-2mm}& \hspace{-2mm}\tilde{\mathcal{Q}}_{12}\\
   \tilde{\mathcal{Q}}_{12}'\hspace{-2mm} & \hspace{-2mm}\tilde{\mathcal{Q}}_{2}     \\
  \end{array}
\hspace{-2mm}\right].\label{q12}
\end{align}

Now we will show $\tilde{\mathcal{Q}}_{12}=0$. Actually, for any $x=U\left[\hspace{-1mm}
  \begin{array}{cc}
   \hspace{-1mm} x_{1}\hspace{-1mm}\\
   \hspace{-1mm} x_{2}    \hspace{-1mm}           \\
  \end{array}
\hspace{-1mm}\right]\in \mathcal{R}^{2n}$ and the dimension of $x_{2}$ is the same as the dimension of $\tilde{\mathcal{Q}}_{2}$, we have that
\begin{align}\label{dim}
x'\tilde{\mathcal{Q}}x\hspace{-1.2mm}=\hspace{-1.2mm}\left[\hspace{-1mm}
  \begin{array}{cc}
   \hspace{-1mm} x_{1}\hspace{-2mm}\\
   \hspace{-1mm} x_{2}    \hspace{-2mm}           \\
  \end{array}
\hspace{-2mm}\right]'\hspace{-1mm}\left[\hspace{-1mm}
  \begin{array}{cc}
   0\hspace{-2mm}& \hspace{-2mm}\tilde{\mathcal{Q}}_{12}\\
   \tilde{\mathcal{Q}}_{12}' \hspace{-2mm}&\hspace{-2mm} \tilde{\mathcal{Q}}_{2}     \\
  \end{array}
\hspace{-2mm}\right]\hspace{-1mm}\left[\hspace{-1mm}
  \begin{array}{cc}
   \hspace{-1mm} x_{1}\hspace{-2mm}\\
   \hspace{-1mm} x_{2}    \hspace{-2mm}           \\
  \end{array}
\hspace{-2mm}\right]\hspace{-1.2mm}=\hspace{-1.2mm}x_{2}' \tilde{\mathcal{Q}}_{12}'x_{1}\hspace{-1mm}+\hspace{-1mm}x_{1}'\tilde{\mathcal{Q}}_{12}x_{2}\hspace{-1mm}+\hspace{-1mm}x_{2}'\tilde{\mathcal{Q}}_{2}x_{2}.
\end{align}

If $\tilde{\mathcal{Q}}_{12}\neq 0$, from \eqref{dim}, we can always choose $x_{1}$ and $x_{2}$ such that $x'\tilde{\mathcal{Q}}x<0$, which is a contradiction with $\tilde{\mathcal{Q}}\geq 0$. Thus, noting \eqref{q12}, we have
\begin{align}\label{q13}
&U'\tilde{\mathbb{A}}U\hspace{-1mm}=\hspace{-1mm}\left[\hspace{-1mm}
  \begin{array}{cc}
   \tilde{\mathbb{A}}_{11}\hspace{-2mm}&\hspace{-2mm} \tilde{\mathbb{A}}_{12}\\
   0\hspace{-2mm}&\hspace{-2mm} \tilde{\mathbb{A}}_{22}    \\
  \end{array}
\hspace{-1mm}\right]\hspace{-1mm},U'\tilde{\mathbb{C}}U\hspace{-1mm}=\hspace{-1mm}\left[\hspace{-1mm}
  \begin{array}{cc}
   \tilde{\mathbb{C}}_{11}\hspace{-2mm}& \hspace{-2mm}\tilde{\mathbb{C}}_{12}\\
   0\hspace{-2mm}&\hspace{-2mm} \tilde{\mathbb{C}}_{22}    \\
  \end{array}
\hspace{-1mm}\right]\hspace{-1mm},U'\tilde{\mathcal{Q}}U\hspace{-1mm}=\hspace{-1mm}\left[\hspace{-1mm}
  \begin{array}{cc}
   0\hspace{-2mm}& \hspace{-2mm}0\\
   0 \hspace{-2mm}& \hspace{-2mm}\tilde{\mathcal{Q}}_{2}     \\
  \end{array}
\hspace{-1mm}\right]\hspace{-1mm},\notag\\
&\tilde{\mathcal{Q}}_{12}\hspace{-1mm}=\hspace{-1mm}0,\tilde{\mathbb{A}}_{21}\hspace{-1mm}=\hspace{-1mm}0, \tilde{\mathcal{Q}}_{2}\geq 0,
\end{align}
where $\tilde{\mathbb{C}}_{11}=\tilde{\mathbb{C}}_{11}^{(1)}+\tilde{\mathbb{C}}_{11}^{(2)}$, $\tilde{\mathbb{C}}_{12}=\tilde{\mathbb{C}}_{12}^{(1)}+\tilde{\mathbb{C}}_{12}^{(2)}$, $\tilde{\mathbb{C}}_{22}=\tilde{\mathbb{C}}_{22}^{(1)}+\tilde{\mathbb{C}}_{22}^{(2)}$.

By plugging \eqref{upu} and \eqref{q13} into \eqref{ly5}, we have that
\begin{equation}\label{ly6}
 0=\tilde{\mathcal{Q}}_{2}+\tilde{\mathbb{A}}_{22}' \mathbb{P}_{2}+\mathbb{P}_{2}\tilde{\mathbb{A}}_{22}+ \{\tilde{\mathbb{C}}_{22}^{(1)}\}'\mathbb{P}_{2}\tilde{\mathbb{C}}_{22}^{(1)}+ \{\tilde{\mathbb{C}}_{22}^{(2)}\}'\mathbb{P}_{2}\tilde{\mathbb{C}}_{22}^{(2)}.
\end{equation}

Denote $U'\mathbb{X}_{t}=\bar{\mathbb{X}}_{t}=\left[\hspace{-1mm}
  \begin{array}{cc}
   \hspace{-1mm} \bar{\mathbb{X}}_{t}^{(1)}\hspace{-1mm}\\
   \hspace{-1mm} \bar{\mathbb{X}}_{t}^{(2)}    \hspace{-1mm}           \\
  \end{array}
\hspace{-1mm}\right]$, and the dimension of $\bar{\mathbb{X}}_{t}^{(2)} $ coincides with the rank of $\mathbb{P}_{2}$. Thus, \eqref{fb} can be rewritten as
\begin{align*}
U'd\mathbb{X}_{t}&=U'\tilde{\mathbb{A}}UU'\mathbb{X}_{t}dt+U'\tilde{\mathbb{C}}UU'\mathbb{X}_{t}dW_{t},
\end{align*}
that is
\begin{align}
d\bar{\mathbb{X}}_{t}^{(\hspace{-0.3mm}1\hspace{-0.3mm})}&\hspace{-1mm}=\hspace{-1mm}[\tilde{\mathbb{A}}_{11}\bar{\mathbb{X}}_{t}^{(\hspace{-0.3mm}1\hspace{-0.3mm})}
\hspace{-1mm}+\hspace{-1mm}\tilde{\mathbb{A}}_{12}\bar{\mathbb{X}}_{t}^{(\hspace{-0.3mm}2\hspace{-0.3mm})}]dt
\hspace{-1mm}+\hspace{-1mm}[\tilde{\mathbb{C}}_{11}\bar{\mathbb{X}}_{t}^{(\hspace{-0.3mm}1\hspace{-0.3mm})}\hspace{-1mm}+\hspace{-1mm}\tilde{\mathbb{C}}_{12}\bar{\mathbb{X}}_{t}^{(\hspace{-0.3mm}2\hspace{-0.3mm})}]dW_{t},\label{lly1}\\
d\bar{\mathbb{X}}_{t}^{(2)}&=\tilde{\mathbb{A}}_{22}\bar{\mathbb{X}}_{t}^{(2)}dt+\tilde{\mathbb{C}}_{22}\bar{\mathbb{X}}_{t}^{(2)}dW_{t}.\label{lly2}
\end{align}

Next, we will show the stability of $(\tilde{\mathbb{A}}_{22},\tilde{\mathbb{C}}_{22})$.

In fact, it follows from \eqref{eqas} and \eqref{q13} that
\begin{align}\label{lya30}
&~~\int_{0}^{T}E[(\bar{\mathbb{X}}_{t}^{(2)})'\tilde{\mathcal{Q}}_{2}\bar{\mathbb{X}}_{t}^{(2)}]dt=\int_{0}^{T}E(\mathbb{X}_{t}'\tilde{\mathcal{Q}}\mathbb{X}_{t})dt\notag\\
&=E(\mathbb{X}_{0}'\mathbb{P}\mathbb{X}_{0})-E(\mathbb{X}_{T}'\mathbb{P}\mathbb{X}_{T})\notag\\
&=E[(\bar{\mathbb{X}}_{0}^{(2)})'\mathbb{P}_{2}\bar{\mathbb{X}}_{0}^{(2)}]-E[(\bar{\mathbb{X}}_{T}^{(2)})'\mathbb{P}_{2}\bar{\mathbb{X}}_{T}^{(2)}].
\end{align}
Following the discussions as \eqref{lll}-\eqref{130}, we claim that $\bar{\mathbb{X}}_{0}^{(2)}$ is an unobservable state of $(\tilde{\mathbb{A}}_{22},\tilde{\mathbb{C}}_{22},\tilde{\mathcal{Q}}_{2}^{1/2})$ if and only if $\bar{\mathbb{X}}_{0}^{(2)}$ satisfies $E[(\bar{\mathbb{X}}_{0}^{(2)})'\mathbb{P}_{2}\bar{\mathbb{X}}_{0}^{(2)}]=0$. $\mathbb{P}_{2}>0$ implies that $(\tilde{\mathbb{A}}_{22},\tilde{\mathbb{C}}_{22},\tilde{\mathcal{Q}}_{2}^{1/2})$ is exact observable as discussed in Lemma \ref{ess} and Remark \ref{rem:5}. Following the discussions of \eqref{lya}-\eqref{a-6}, we can conclude that
\begin{equation}\label{l2l}
  \lim_{t\rightarrow +\infty}E(\bar{\mathbb{X}}_{t}^{(2)})'\bar{\mathbb{X}}_{t}^{(2)}=0,
\end{equation}
i.e., the mean square stability of $(\tilde{\mathbb{A}}_{22},\tilde{\mathbb{C}}_{22})$ has been verified.

Thirdly, to investigate the stability of $(\tilde{\mathbb{A}}_{11},\tilde{\mathbb{C}}_{11})$, we might as well choose $\bar{\mathbb{X}}_{0}^{(2)}=0$, and equation \eqref{lly2} indicates $\bar{\mathbb{X}}_{t}^{(2)}=0$, $t\geq 0$. In this case, \eqref{lly1} can be reduced to
\begin{equation}\label{zz}
  d\mathbb{Z}_{t}=\tilde{\mathbb{A}}_{11}\mathbb{Z}_{t}dt+\tilde{\mathbb{C}}_{11}\mathbb{Z}_{t}dW_{t},
\end{equation}
where $\mathbb{Z}_{t}$ is the value of $\bar{\mathbb{X}}_{t}^{(1)}$ with $\bar{\mathbb{X}}_{t}^{(2)}=0$. Hence, for $\bar{\mathbb{X}}_{0}^{(2)}=0$, it holds
\begin{equation}\label{lly3}
  E[\tilde{\mathcal{Y}}_{t}'\tilde{\mathcal{Y}}_{t}]=
  E[\mathbb{X}_{t}'\tilde{\mathcal{Q}}\mathbb{X}_{t}]=E[(\bar{\mathbb{X}}_{t}^{(2)})'\tilde{\mathcal{Q}}_{2}\bar{\mathbb{X}}_{t}^{(2)}]\equiv 0.
\end{equation}
While, the exact detectability of $(\tilde{\mathbb{A}},\tilde{\mathbb{C}},\tilde{\mathcal{Q}}^{1/2})$ implies that
\begin{align}\label{xxx}
\lim_{t\rightarrow +\infty}\hspace{-2mm}E(\bar{\mathbb{X}}_{t}'\bar{\mathbb{X}}_{t})\hspace{-1mm}=\hspace{-1mm}\lim_{t\rightarrow +\infty}\hspace{-2mm}E(\bar{\mathbb{X}}_{t}'U'U\bar{\mathbb{X}}_{t})\hspace{-1mm}=\hspace{-1mm}\lim_{t\rightarrow +\infty}\hspace{-2mm}E(\mathbb{X}_{t}'\mathbb{X}_{t})\hspace{-1mm}=\hspace{-1mm}0.
\end{align}
Therefore, in the case of $\bar{\mathbb{X}}_{0}^{(2)}=0$, for any initial $\mathbb{Z}_{0}=\bar{\mathbb{X}}_{0}^{(1)}$, from \eqref{xxx} we have that
\begin{align}\label{l1l}
 &~~ \lim_{t\rightarrow +\infty}\hspace{-2mm}E(\mathbb{Z}_{t}'\mathbb{Z}_{t})\hspace{-1mm}=\hspace{-1mm}
 \lim_{t\rightarrow +\infty}\hspace{-2mm} E[(\bar{\mathbb{X}}_{t}^{(1)})'\bar{\mathbb{X}}_{t}^{(1)}]\\
 &=\lim_{t\rightarrow +\infty}\{E[(\bar{\mathbb{X}}_{t}^{(1)})'\bar{\mathbb{X}}_{t}^{(1)}]+E[(\bar{\mathbb{X}}_{t}^{(2)})'\bar{\mathbb{X}}_{t}^{(2)}]\}\notag\\
 &=\lim_{t\rightarrow +\infty}E(\bar{\mathbb{X}}_{t}'\bar{\mathbb{X}}_{t})=0.\notag
\end{align}
which means $(\tilde{\mathbb{A}}_{11},\tilde{\mathbb{C}}_{11})$ is mean square stable.

Finally we will show that system \eqref{ps1} with controller \eqref{control} is stabilizable in the mean square sense. Actually, we denote $\tilde{\mathcal{A}}=\left[\hspace{-1mm}
  \begin{array}{cc}
   \tilde{\mathbb{A}}_{11}\hspace{-1mm}&\hspace{-1mm} 0\\
   0\hspace{-1mm}&\hspace{-1mm} \tilde{\mathbb{A}}_{22}    \\
  \end{array}
\hspace{-1mm}\right]$, $\tilde{\mathcal{C}}=\left[\hspace{-1mm}
  \begin{array}{cc}
   \tilde{\mathbb{C}}_{11}\hspace{-1mm}&\hspace{-1mm} 0\\
   0\hspace{-1mm}&\hspace{-1mm} \tilde{\mathbb{C}}_{22}    \\
  \end{array}
\hspace{-1mm}\right]$. Hence,  we can rewrite \eqref{lly1}-\eqref{lly2} as below
\begin{align}\label{133}
d\bar{\mathbb{X}}_{t}\hspace{-1mm}=\hspace{-1mm}\{\tilde{\mathcal{A}}\bar{\mathbb{X}}_{t}+\left[\hspace{-1mm}
  \begin{array}{cc}
   \tilde{\mathbb{A}}_{12}\\
   0\\
  \end{array}
\hspace{-1mm}\right]\mathbb{U}_{t}\}dt\hspace{-1mm}+\hspace{-1mm}\{\tilde{\mathcal{C}}\bar{\mathbb{X}}_{t}\hspace{-1mm}+\hspace{-1mm}\left[\hspace{-1mm}
  \begin{array}{cc}
   \tilde{\mathbb{C}}_{12}\\
   0\\
  \end{array}
\hspace{-1mm}\right]\mathbb{U}_{t}\}dW_{t},
\end{align}
where $\mathbb{U}_{t}$ is the solution to equation \eqref{lly2} with initial condition $\mathbb{U}_{0}=\mathbb{X}_{0}^{(2)}$. The stability of $(\tilde{\mathbb{A}}_{11},\tilde{\mathbb{C}}_{11})$ and $(\tilde{\mathbb{A}}_{22},\tilde{\mathbb{C}}_{22})$ as shown above indicates that $(\tilde{\mathcal{A}},\tilde{\mathcal{C}})$ is stable in the mean square sense. It is easily known from \eqref{l2l} that $\lim_{t\rightarrow +\infty}E(\mathbb{U}_{t}'\mathbb{U}_{t})=0$ and $\int_{0}^{\infty}E(\mathbb{U}_{t}'\mathbb{U}_{t})dt<+\infty$. Using \emph{Proposition 2.8} and \emph{Remark 2.9} in \cite{abb}, we can obtain that there exists constant $c_{0}$ satisfying
\begin{align}\label{xu}
\int_{0}^{\infty}E(\bar{\mathbb{X}}_{t}'\bar{\mathbb{X}}_{t})dt<c_{0}\int_{0}^{\infty}E(\mathbb{U}_{t}'\mathbb{U}_{t})dt<+\infty.
\end{align}
Hence, $\lim_{t\rightarrow +\infty}E(\bar{\mathbb{X}}_{t}'\bar{\mathbb{X}}_{t})=0$ can be verified from \eqref{xu}. Moreover, from \eqref{xxx} we have
\begin{align*}
&\lim_{t\rightarrow +\infty}\hspace{-2mm}E(x_{t}'x_{t})\hspace{-1mm}=\hspace{-1mm}\lim_{t\rightarrow +\infty}E[(x_{t}\hspace{-1mm}-\hspace{-1mm}Ex_{t})'(x_{t}\hspace{-1mm}-\hspace{-1mm}Ex_{t})\hspace{-1mm}+\hspace{-1mm}Ex_{t}'Ex_{t}]\notag\\
&=\lim_{t\rightarrow +\infty}\hspace{-2mm}E(\mathbb{X}_{t}'\mathbb{X}_{t})\hspace{-1mm}=\hspace{-1mm}\lim_{t\rightarrow +\infty}\hspace{-2mm}E(\bar{\mathbb{X}}_{t}'\bar{\mathbb{X}}_{t})\hspace{-1mm}=\hspace{-1mm}0.
\end{align*}
It is noted that system $(\tilde{\mathbb{A}},\tilde{\mathbb{C}})$ given in \eqref{fb} is just mean-field system \eqref{ps1} with controller \eqref{control}. In conclusion, mean-field system \eqref{ps1} can be stabilizable with controller \eqref{control} in the mean square sense.

Finally, for stabilizing controller \eqref{control}, there holds
\begin{align}\label{uinf}
E(u_t'u_t)\hspace{-1mm}=\hspace{-1mm}E[x_t'\mathcal{K}'\mathcal{K}x_t\hspace{-1mm}+\hspace{-1mm}Ex_t'(\bar{\mathcal{K}}'\mathcal{K}+\mathcal{K}'\bar{\mathcal{K}}\hspace{-1mm}+\hspace{-1mm}\bar{\mathcal{K}}'\bar{\mathcal{K}})Ex_t].\end{align}
From \eqref{as3} we know that $\int_{0}^{\infty}E(x_t'x_t)dt<+\infty$, therefore $\int_{0}^{\infty}E(u_t'u_t)dt<+\infty$ can be obtained from \eqref{uinf}. Thus we have proved $u_t\in \mathcal{U}[0,\infty)$. The proof is complete.
\end{proof}

\section*{Appendix F: Proof of Theorem \ref{thm:succeed2}}
\begin{proof}
``Sufficiency": Under Assumptions \ref{ass:ass2} and \ref{ass:ass3}, if the coupled ARE \eqref{are1}-\eqref{are2} has a unique positive definite solution, we will show mean-field system \eqref{ps1} is mean square stabilizable with controller \eqref{control}.

In fact, from Remark \ref{rem:5}, we know that the exact observability of system $(\tilde{\mathbb{A}},\tilde{\mathbb{C}},\tilde{\mathcal{Q}}^{1/2})$ can be implied by Assumption \ref{ass:ass3}.

In what follows, by following \eqref{a}-\eqref{a-6} in the proof of Theorem \ref{thm:succeed}, the mean square stability of system $(\tilde{\mathbb{A}},\tilde{\mathbb{C}})$ can be obtained, i.e., system \eqref{ps1} is mean square stabilizable with controller \eqref{control}. The proof is complete.

``Necessity": Under Assumptions \ref{ass:ass2} and \ref{ass:ass3}, suppose system \eqref{ps1} is mean square stabilizable, we will show the coupled ARE \eqref{are1}-\eqref{are2} admits a unique positive definite solution.

Firstly, under Assumption \ref{ass:ass2}, it is noted from \eqref{com1}-\eqref{un} that the coupled ARE \eqref{are1}-\eqref{are2} admit a unique positive semi-definite solution, i.e., $P\geq 0$ and $P+\bar{P}\geq 0$. Next we shall prove the positive definiteness of $P$ and $P+\bar{P}$.

Actually, if this is not the case, since $E(x_0'x_0)=E(\mathbb{X}_0'\mathbb{X}_0)$, then there exists $\mathbb{X}_0\neq 0$ (i.e., $x_0\neq 0$) satisfying $E(\mathbb{X}_0'\mathbb{P}\mathbb{X}_0)=0$, the symbols $\mathbb{P},\mathbb{X}_t$ are given in \eqref{mf01} and \eqref{eqas}.

From Lemma \ref{ess}, we know the mean square stabilization of system \eqref{ps1} with controller \eqref{control}, indicates system \eqref{mf01} $(\tilde{\mathbb{A}},\tilde{\mathbb{C}},\tilde{\mathcal{Q}}^{1/2})$ is mean square stable, and the solution to ARE $P,\bar{P}$ satisfy Lyapunov function \eqref{4-ly1}-\eqref{4-ly2}. Next, by following from the derivation of \eqref{eqas} and letting the initial state be $\mathbb{X}_0$ defined above, we can obtain
\begin{align}\label{5-lya03}
0\leq \int_{0}^{T}E(\mathbb{X}_{t}'\tilde{\mathcal{Q}}\mathbb{X}_{t})dt=-E(\mathbb{X}_{T}'\mathbb{P}\mathbb{X}_{T})\leq 0,
\end{align}
which indicates $\tilde{\mathcal{Q}}^{1/2}\mathbb{X}_{t}\equiv 0, ~a.s.,~\forall t\in[0,T]$.

On the other hand, noting from Remark \ref{rem:5} that the exact observability of system $(\tilde{\mathbb{A}},\tilde{\mathbb{C}},\tilde{\mathcal{Q}}^{1/2})$ can be implied by Assumption \ref{ass:ass3}. Thus, we can conclude $\mathbb{X}_0=0$. This contradicts with $\mathbb{X}_0\neq 0$. Therefore, $P>0$ and $P+\bar{P}>0$ has been shown.

Finally, the infinite horizon optimal controller \eqref{control} can be obtained by following \eqref{cost5}-\eqref{cost21}, and $u_t\in\mathcal{U}[0,\infty)$ can be verified by \eqref{uinf}.
\end{proof}

%\section*{Acknowledgments}

%\bibliographystyle{siamplain}
%\bibliography{references}

\begin{IEEEbiography}[{\includegraphics[width=1in,height=1.25in,clip,keepaspectratio]{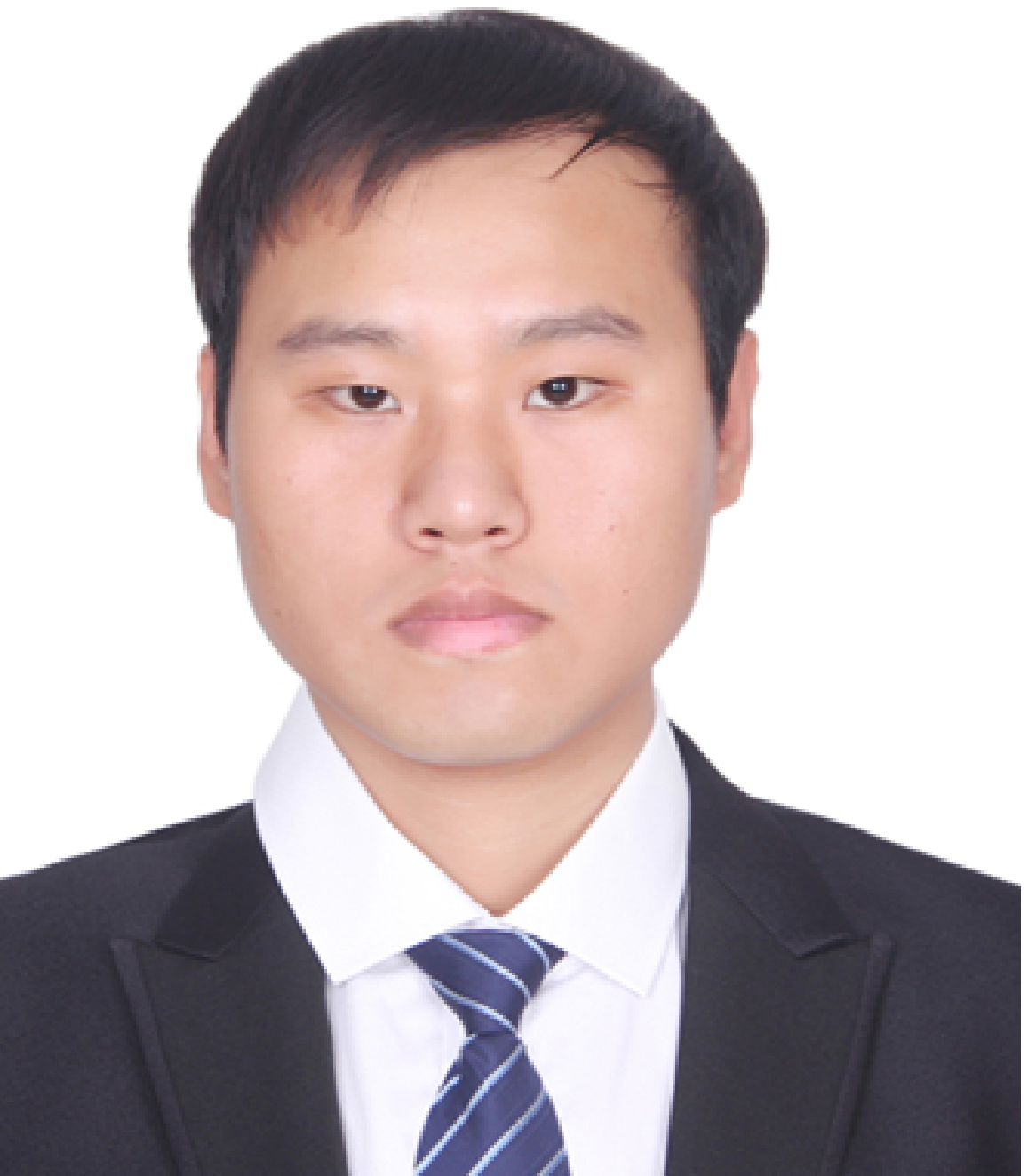}}]{Qingyuan~Qi}
received the B.S. degree in mathematics from Shandong University, Jinan, Shandong, China, in 2012. He is currently working toward the Ph.D. degree at the School of Control Science and Engineering, Shandong University, Jinan, Shandong, China.

His research interests include optimal control, optimal estimation, stabilization and stochastic systems.
\end{IEEEbiography}

\begin{IEEEbiography}[{\includegraphics[width=1in,height=1.25in,clip,keepaspectratio]{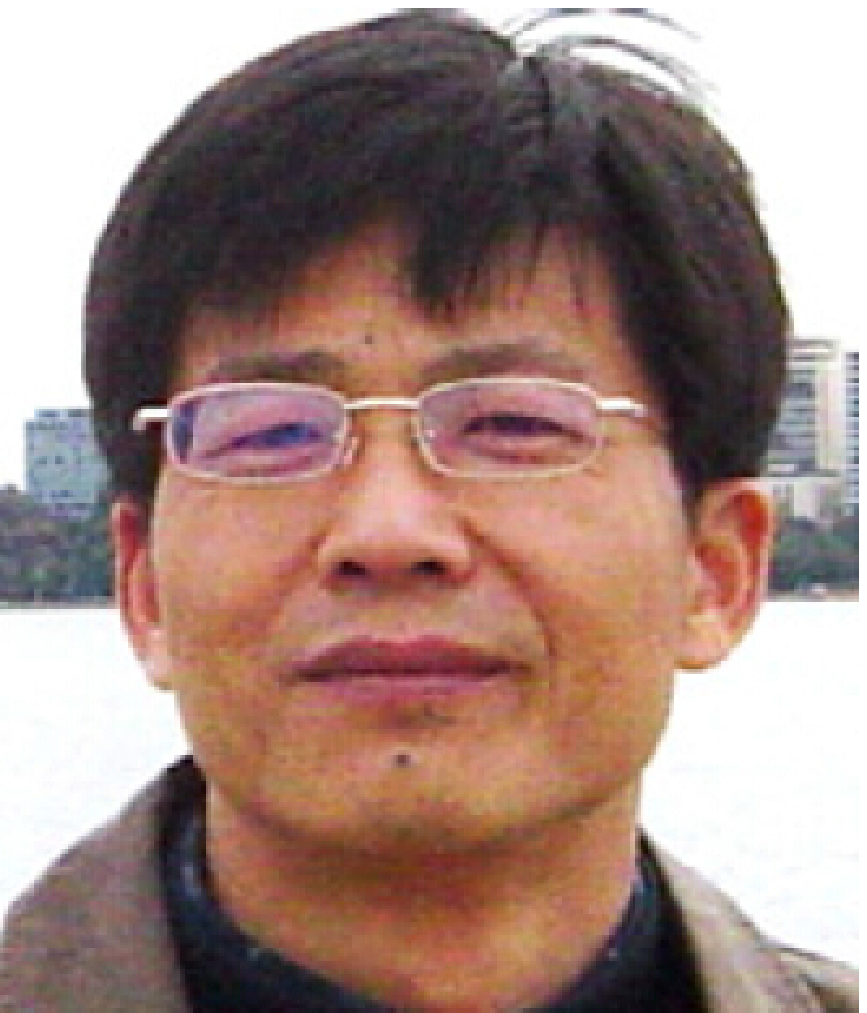}}]{Huanshui~Zhang}
(SM'06) received the B.S. degree in mathematics from Qufu Normal University, Shandong, China, in 1986, the M.Sc. degree in control theory from Heilongjiang University, Harbin, China, in 1991, and the Ph.D. degree in control theory from Northeastern University, China, in 1997.

He was a Postdoctoral Fellow at Nanyang Technological University, Singapore, from 1998 to 2001 and Research Fellow at Hong Kong Polytechnic University, Hong Kong, China, from 2001 to 2003. He is currently holds a Professorship at Shandong University, Shandong, China. He was a Professor with the Harbin Institute of Technology, Harbin, China, from 2003 to 2006. He also held visiting appointments as a Research Scientist and Fellow with Nanyang Technological University, Curtin University of Technology, and Hong Kong City University from 2003 to 2006. His interests include optimal estimation and control, time-delay systems, stochastic systems, signal processing and wireless sensor networked systems.
\end{IEEEbiography}
%


\vfill


\end{document}